\documentclass[12pt,letterpaper]{amsart}

\newcommand{\la}{\langle}
\newcommand{\ra}{\rangle}
\renewcommand{\Re}{\operatorname{Re}}
\renewcommand{\Im}{\operatorname{Im}}
\newcommand{\ds}{\displaystyle}
\newcommand{\sech}{\operatorname{sech}}
\newcommand{\defeq}{\stackrel{\rm{def}}{=}}

\usepackage{amssymb}
\usepackage{amsthm}
\usepackage{amsxtra}
\usepackage{graphicx}
\usepackage{afterpage}

\renewenvironment{itemize}{\begin{list}{\labelitemi}{\leftmargin=1em}
}{\end{list}}

\newtheorem{theorem}{Theorem}

\newtheorem{proposition}[theorem]{Proposition}
\newtheorem{lemma}[theorem]{Lemma}
\newtheorem{corollary}[theorem]{Corollary}

\newtheorem{example}[theorem]{Example}
\theoremstyle{remark}
\newtheorem{remark}[theorem]{Remark}

\newcommand{\cR}{\mathbb{R}}

\newcommand{\cC}{\mathbb{C}}

\newcommand{\grad}{\nabla}

\numberwithin{equation}{section}

\numberwithin{theorem}{section}

\numberwithin{table}{section}

\numberwithin{figure}{section}

\ifx\pdfoutput\undefined
  \DeclareGraphicsExtensions{.pstex, .eps}
\else
  \ifx\pdfoutput\relax
    \DeclareGraphicsExtensions{.pstex, .eps}
  \else
    \ifnum\pdfoutput>0
      \DeclareGraphicsExtensions{.pdf}
    \else
      \DeclareGraphicsExtensions{.pstex, .eps}
    \fi
  \fi
\fi

\title[Blow-up criteria for 3d cubic NLS]
{Blow-up criteria for the 3d cubic nonlinear Schr\"odinger equation}

\author{Justin Holmer}
\address{Brown University}
\author{Rodrigo Platte}
\address{University of Oxford}
\author{Svetlana Roudenko}
\address{Arizona State University}

\begin{document}

\begin{abstract}
We consider solutions $u$ to the 3d nonlinear Schr\"odinger equation
$i\partial_t u + \Delta u + |u|^2u=0$.  In particular, we are
interested in finding criteria on the initial data $u_0$ that
predict the asymptotic behavior of $u(t)$, e.g., whether $u(t)$
blows-up in finite time, exists globally in time but behaves like a
linear solution for large times (scatters), or exists globally in
time but does not scatter.  This question has been resolved (at
least for $H^1$ data) in \cite{HR2,DHR,DR,HR3} if $M[u]E[u]\leq
M[Q]E[Q]$, where $M[u]$ and $E[u]$ denote the mass and energy of
$u$, and $Q$ denotes the ground state solution to $-Q+\Delta Q +
|Q|^2Q=0$. Here we consider the complementary case
$M[u]E[u]>M[Q]E[Q]$.  In the first (analytical) part of the paper,
we present a result due to Lushnikov \cite{Lu95}, based on the
virial identity and the uncertainty principle, giving a sufficient
condition for blow-up. By replacing the uncertainty principle in his
argument with an interpolation-type inequality, we obtain a new
blow-up condition that in some cases improves upon Lushnikov's
condition. Our approach also allows for an adaptation to radial
infinite-variance initial data that has a conceptual interpretation:
for real-valued initial data, if a certain fraction of the mass is
contained within the ball of radius $M[u]$, then blow-up occurs.  We
also show analytically (if one takes the numerically computed value
of $\|Q\|_{\dot H^{1/2}}$) that there exist Gaussian initial data
$u_0$ with negative quadratic phase such that $\|u_0\|_{\dot
H^{1/2}} < \|Q\|_{\dot H^{1/2}}$ but the solution $u(t)$ blows-up.
In the second (numerical) part of the paper,  we examine several
different classes of initial data -- Gaussian, super-Gaussian,
off-centered Gaussian, and oscillatory Gaussian -- and for each
class give the theoretical predictions for scattering or blow-up
provided by the above theorems as well as the results of numerical
simulation.  On the basis of the numerical simulations, we formulate
several conjectures, among them that for \emph{real} initial data,
the quantity $\|Q\|_{\dot H^{1/2}}$ provides the threshold for
scattering.
\end{abstract}

\maketitle

\section{Introduction}

The nonlinear Schr\"odinger equation (NLS) or Gross-Pitaevskii equation is
\begin{equation}
 \label{E:NLSa}
i\partial_t u + \Delta u+ |u|^2u=0 \, ,
\end{equation}
with wave function $u=u(x,t)\in \mathbb{C}$. We consider $x\in
\mathbb{R}^n$ in dimensions $n=1,2,$ or $3$. The initial-value
problem is locally well-posed in $H^1$ (see Cazenave \cite{Caz-book}
for exposition and references therein).  In this now standard theory
obtained from the Strichartz estimates, initial data $u_0 \in H^1$
give rise to a unique solution $u(t)\in C([0,T];H^1)$ with the time
interval $[0,T]$ of existence specified in terms of $\|u_0\|_{H^1}$.
In some situations, an \emph{a priori} bound on $\|u(t)\|_{H^1}$ can
be deduced from conservation laws which implies the solution $u(t)$
exists globally in time. On the other hand, we say that a solution
$u(t)$ to NLS \emph{blows-up in finite time} $T^*$ provided
\begin{equation}
\label{E:blow-up}
\lim_{t\nearrow T^*} \|\nabla u(t)\|_{L^2} = +\infty\,.
\end{equation}
For $n=1$, all $H^1$ initial data yield global solutions, but large
classes of initial data leading to solutions blowing-up in finite
time are known for $n=2$ and $n=3$.  NLS arises as a model of
several physical phenomena.  We outline three important examples in
a supplement to this introduction (\S\ref{S:physics} below), and
emphasize that in each case the mathematical property of blow-up in
finite time is realistic and relevant.  It is therefore of interest
to determine mathematical conditions on the initial data that
guarantee the corresponding solution will blow-up in finite-time and
conditions that guarantee it will exist globally in time.  Moreover,
if we know the solution is global, it is natural to ask whether we
can predict the asymptotic ($t\to +\infty$) behavior of the
solution.  If the solution asymptotically approaches a solution of
the linear equation, we say it \emph{scatters}.  Nonlinear effects
can persist indefinitely, however; for example, leading to formation
of \emph{solitons} or \emph{long-range modulation of linear
solutions}.

Partial answers to the above mathematical problem are known, and we
will discuss separately the existing literature in the case of
dimensions $n=1,2,$ and $3$.  Afterward, we will state our new
findings in the $n=3$ case.

Before proceeding, we note that NLS satisfies conservation of mass
$M[u]$, momentum $P[u]$, and energy $E[u]$, where
$$
M[u] = \|u\|_{L^2}^2, \quad P[u]= \Im \int \bar u \nabla u \,,
$$
$$
E[u] = \frac12 \|\nabla u\|_{L^2}^2 - \frac14 \|u\|_{L^4}^4 \,.
$$
Also, NLS satisfies the scaling symmetry
\begin{equation}
\label{E:scaling}
u(x,t) \text{ solves NLS } \implies \lambda u(\lambda x, \lambda^2
t) \text{ solves NLS} .
\end{equation}
Consequently, the critical (scale-invariant) Sobolev space
$H^s(\mathbb{R}^n)$ is $s=\frac{n-2}{2}$. The NLS equation also
satisfies the Galilean invariance: For any $v\in \mathbb{R}^n$,
$$
u(x,t) \text{ solves NLS } \implies e^{ix\cdot
v}e^{-it|v|^2}u(x-2vt, t) \text{ solves NLS}\,,
$$
and thus, any solution can be transformed to one for which $P[u]=0$.
Let
$$
V[u](t) = \|xu(t)\|_{L_x^2}^2
$$
denote the variance.  Assuming $V[u](0)<\infty$, then the virial
identities (Vlasov-Petrishchev-Talanov \cite{VPT}, Zakharov
\cite{Zak72}, Glassey \cite{G})
\begin{equation}
\label{E:virial}
\begin{aligned}
&\partial_t V[u] = 4\Im \int x\cdot \nabla u \, \bar u \, dx \,,
&\partial_t^2 V[u] =  8nE[u] + (8-4n) %
\|\nabla u\|_{L_x^2}^2
\end{aligned}
\end{equation}
hold. Let $Q=Q(x)$ denote the real-valued, smooth, exponentially
decaying ground state solution to
\begin{equation}
\label{E:Q}
-Q+\Delta Q + Q^3=0 \,.
\end{equation}
Then $u(x,t) = e^{it}Q(x)$ solves NLS, and is called the
\emph{ground state soliton}.  The Pohozhaev identities are
\begin{equation}
 \label{E:Pohozhaev}
\|\nabla Q \|_{L^2}^2 = \frac{n}{4-n} \|Q\|_{L^2}^2 \,, \qquad
\|Q\|_{L^4}^4 = \frac{4}{4-n}\|Q\|_{L^2}^2\,.
\end{equation}
Weinstein \cite{W83} proved that the Gagliardo-Nirenberg inequality
\begin{equation}
 \label{E:Weinstein}
\|\phi\|_{L^4}^4 \leq c_{\text{GN}} \|\phi\|_{L^2}^{4-n} \|\nabla
\phi\|_{L^2}^n
\end{equation}
is saturated by $\phi=Q$, i.e.,
$$
c_{\text{GN}} = \frac{\|Q\|_{L^4}^4}{\|Q\|_{L^2}^{4-n}\|\nabla
Q\|_{L^2}^n} \,
$$
is the sharp constant.

\smallskip

\noindent\emph{2d case}.  Much of the mathematically rigorous
literature has been devoted to the 2d case, of particular relevance
to the optics model (item 1 in \S\ref{S:physics}), and has the
special mathematical property of being $L^2$-critical. The energy
$E[u]$ conservation combined with the Weinstein inequality
\eqref{E:Weinstein} implies that if $\|u_0\|_{L^2}<\|Q\|_{L^2}$, an
$H^1$ solution is global. This result is in fact sharp, in the
following sense.  The $L^2$ scale-invariance of the 2d equation
allows for an additional symmetry, the pseudo-conformal
transformation
\begin{equation}
 \label{E:pc}
u(x,t) \text{ solves 2d-NLS} \quad \implies \quad \tilde u(x,t) =
\frac{1}{t} \, e^{\frac{i|x|^2}{4t}}\bar u\Big(\frac{x}{t},
\;\frac{1}{t}\Big) \text{ solves 2d-NLS} .
\end{equation}
This gives rise to an explicit family of blow-up solutions
$$
u_T(x,t) = \frac{1}{(T-t)} \, e^{i/(T-t)} e^{i|x|^2/(T-t)}Q\left(
\frac{x}{T-t} \right)
$$
obtained by the pseudoconformal transformation, time translation,
and scaling.  They blow-up at the origin at time $T>0$ (and $T$ can
be taken arbitrarily small), but $\|u_T\|_{L^2} = \|Q\|_{L^2}$. Note
that they have initial data $(u_T)_0(x) = e^{i|x|^2/T}Q(x/T)/T$,
indicating that the inclusion of a quadratic phase prefactor can
create finite-time blow-up. Moreover, it was observed by
Vlasov-Petrishchev-Talanov \cite{VPT}, Zakharov \cite{Zak72} and
Glassey \cite{G} that if the initial data has finite variance
$\|xu_0\|_{L^2}<\infty$ and $E[u]<0$ (which implies by
\eqref{E:Weinstein} that $\|u_0\|_{L^2}\geq \|Q\|_{L^2}$), then the
solution $u(t)$ blows-up in finite time.  Blow-up solutions with
$E[u]>0$ exist and global solutions with $E[u]>0$ exist.  If
$E[u]>0$, then a sufficient condition for blow-up can be deduced
from the virial identity (see \cite{VPT},
\cite{Zak72})\footnote{Blow-up solutions are also possible when
$E=0$ provided $V_t(0) <0$, for a general review refer to
\cite{SS}.}:
\begin{equation}
\label{E:2dblow-up}
V_t(0)< - \sqrt{ 16EV(0)} \,.
\end{equation}

\smallskip
\noindent\emph{1d case}. The 1d case is $L^2$ subcritical; energy
conservation and \eqref{E:Weinstein} prove that solutions never
blow-up in finite time.  One can still ask if there is a
quantitative threshold for the formation of solitons.  Such a
threshold must be expressed in terms of a scale-invariant quantity,
and the $L^1$ norm is a natural candidate.  We note that soliton
solutions
$$
u(x,t) = e^{it}Q(x) \,, \quad Q(x)=\sqrt 2 \sech x
$$
have $\|u(t)\|_{L^1}= \|Q\|_{L^1}=\sqrt 2\pi$ (as do rescalings and
Galilean shifts of this solution).  As the equation is completely
integrable (see Zakharov-Shabat \cite{ZS}), one has available the
tools of inverse scattering theory (IST).  IST has been applied by
Klaus-Shaw \cite{KS} to show that if $\|u_0\|_{L^1}< \frac12
\|Q\|_{L^1}$, then no solitons form.  See Holmer-Marzuola-Zworski
\cite{HMZ}, Apx. B for a calculation showing that this is sharp --
for initial data $u_0(x)=\alpha Q(x)$ with $\alpha>\frac12$, a
soliton emerges in the $t\to +\infty$ asymptotic resolution.  We
remark that although no solitons appear if $\|u_0\|_{L^1}< \frac12
\|Q\|_{L^1}$, such solutions do not \emph{scatter}, i.e., they do
not approach a solution to the linear equation as $t\to +\infty$ --
see Barab \cite{Barab}.  In fact, there are long-range effects and
one conjectures \emph{modified scattering} -- see Hayashi-Naumkin
\cite{HN} for some results in this direction for small intial
data.\footnote{\cite{HN} does not cover the full range
$\|u_0\|_{L^1}<\frac12\|Q\|_{L^1}$, and, in fact, the smallness
condition is in terms of a stronger norm.} The Hayashi-Naumkin
paper, in fact, treats a more general equation and does not rely on
IST; presumably IST could be applied to prove modified scattering
for $\|u_0\|_{L^1}<\frac12\|Q\|_{L^1}$ for generic Schwartz $u_0$,
although we are not aware of a reference.

\smallskip
\noindent\emph{3d case}.  We have previously studied the 3d case of
NLS, which is $L^2$ supercritical, in Holmer-Roudenko \cite{HR1,
HR2, HR3}, Duyckaerts-Holmer-Roudenko \cite{DHR}, and
Duyckaerts-Roudenko \cite{DR}. Scattering and blow-up criteria are
most naturally expressed in terms of scale invariant quantities, and
natural candidates are the $L^3$ norm and the $\dot H^{1/2}$ norm.
We argue below that the $L^3$ norm is completely inadequate, and
while the $\dot H^{1/2}$ norm is a more reasonable choice, it too
appears deficient.  In \cite{HR2,DHR}, we work instead with two
scale-invariant quantities: $M[u]E[u]$ and
\begin{equation}
 \label{E:eta-def}
\eta(t) \defeq \frac{\|u(t)\|_{L^2}\|\nabla
u(t)\|_{L^2}}{\|Q\|_{L^2}\|\nabla Q\|_{L^2}} \,.
\end{equation}
By the Weinstein inequality \eqref{E:Weinstein} and the Pohozhaev
identities \eqref{E:Pohozhaev} we have
\begin{equation}
\label{E:energy1} 3\eta(t)^2 \geq \frac{M[u] E[u]}{M[Q]E[Q]} \geq
3\eta(t)^2-2\eta(t)^3 .
\end{equation}
This results in two ``forbidden regions'' in the $M[u]E[u]/M[Q]E[Q]$
versus $\eta^2$ phase-plane -- see the depiction in Figure
\ref{F:dichotomy}.   Note that since $M[u]$ and $E[u]$ are
conserved, all time evolution in Figure \ref{F:dichotomy} occurs
along horizontal lines. In what follows for brevity and simplicity
we assume $P[u] = 0$ which can be obtained via Galilean transform.
The most general case would follow as it is explained in Appendix B
of \cite{HR3}.

\begin{theorem}[Duyckaerts-Holmer-Roudenko \cite{DHR}, Holmer-Roudenko {\cite{HR2,HR3}}]
\label{T:DHR}
Suppose that $u_0\in H^1$ and $M[u]E[u]<M[Q]E[Q]$.
\begin{enumerate}
\item
If $\eta(0)<1$, then $u(t)$ is globally well-posed and, in fact,
scatters in both time directions.
\item
If $\eta(0)>1$ and either $u_0$ has finite variance or $u_0$ is
radial, then $u(t)$ blows-up in finite positive time and finite
negative time.
\item
If $\eta(0)>1$, then either $u(t)$ blows-up in finite forward time
or there exists a sequence $t_n\nearrow +\infty$ such that $\|\nabla
u(t_n)\|_{L^2} = \infty$.  A similar statement for negative time
holds.
\end{enumerate}
\end{theorem}

\begin{figure}
\includegraphics[width=400pt]{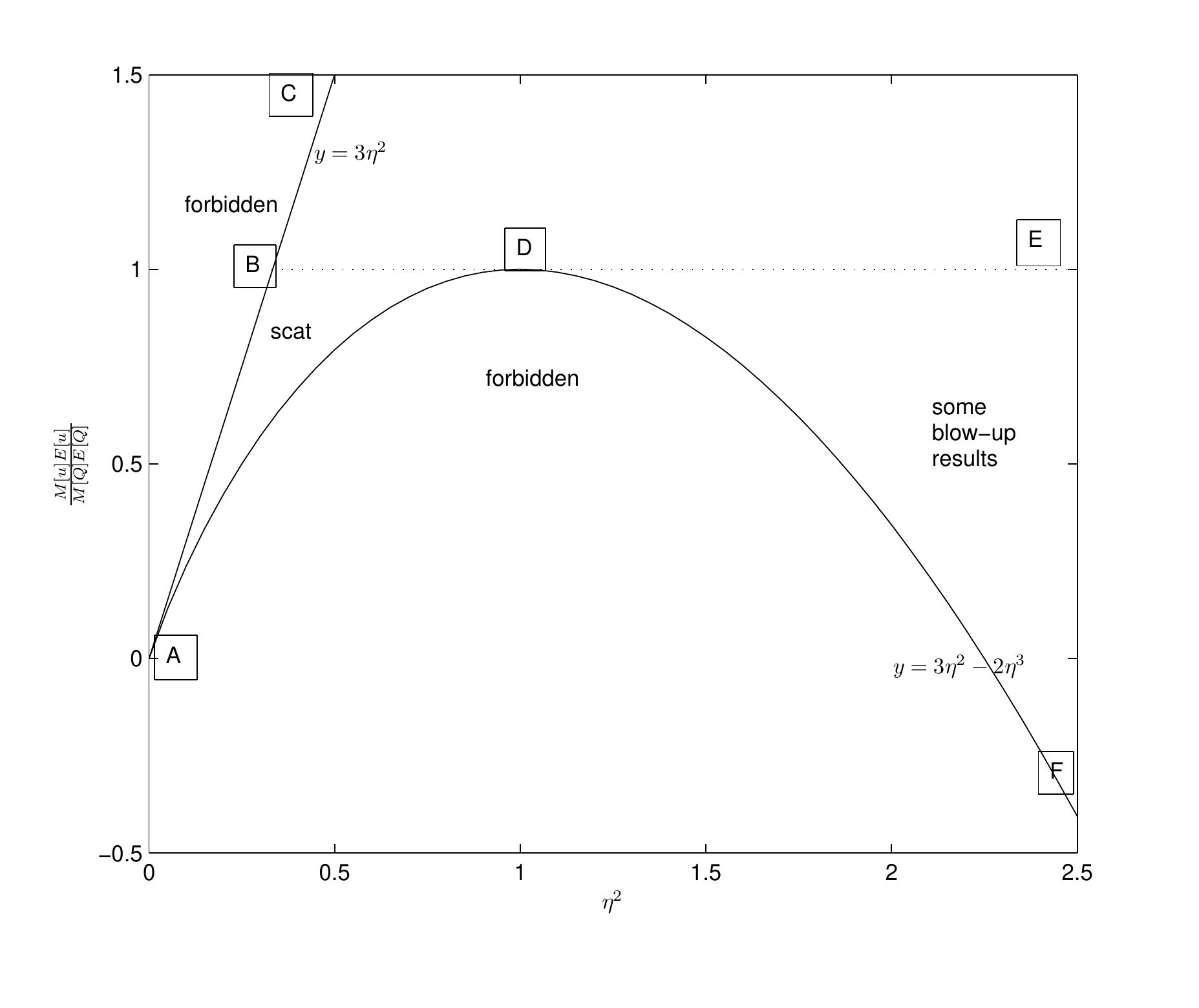}
\caption{A plot of $M[u]E[u]/M[Q]E[Q]$ versus $\eta^2$, where $\eta$
is defined by \eqref{E:eta-def}.  The area to the left of line ABC
and inside region ADF are excluded by \eqref{E:energy1}.  The region
inside ABD corresponds to case (1) of  Theorem \ref{T:DHR}
(solutions scatter).  The region EDF corresponds to case (2) of
Theorem \ref{T:DHR} (solutions blow-up in finite time).  Behavior of
solutions on the dotted line (mass-energy threshold line) is given
by Theorem \ref{T:DR}.}
 \label{F:dichotomy}
\end{figure}

It is a straightforward consequence of the linear decay estimate
that scattering solutions satisfy
$$
\lim_{t\nearrow +\infty} \|u(t)\|_{L^p} =0 \,, \qquad 2<p\leq 6 .
$$
It follows by using the $p=4$ case and the Pohozhaev identities
\eqref{E:Pohozhaev} that
$$
\lim_{t\to +\infty} \eta(t)^2 =  \frac{M[u]E[u]}{3M[Q]E[Q]} \,.
$$
That is, in Figure \ref{F:dichotomy}, a scattering solution has
$\eta(t)^2$ asymptotically approaching boundary line ABC.  On the
other hand, since blow-up solutions satisfy \eqref{E:blow-up}, such
solutions go off to right (along a horizontal line) in Figure
\ref{F:dichotomy}. We note that Merle-Rapha\"el \cite{MR}
strengthened \eqref{E:blow-up}:  they proved that if $u(t)$ blows-up
in finite forward time $T^*>0$, then
$$
\lim_{t \nearrow T^*} \|u(t)\|_{L^3} = +\infty \,.
$$

The blow up for finite variance as in Theorem \ref{T:DHR}, part (2),
has previously been obtained by Kuznetsov et al. in \cite{KRRT}.

The results of Duyckaerts-Roudenko {\cite{DR}} are contained in the
next two theorems.  First, they establish the existence of special
solutions (besides $e^{it}Q$) at the critical mass-energy threshold.
\begin{theorem}[Duyckaerts-Roudenko {\cite{DR}}]
 \label{T:DR-existence}
There exist two radial solutions $Q^+$ and $Q^-$ of NLS with initial
conditions $Q^{\pm}_0$ such that $\ds Q^{\pm}_0\in \cap_{s >0}
H^s(\mathbb{R}^3)$ and
\begin{enumerate}
\item \label{Q+Q-}
$\ds M[Q^+]=M[Q^-]=M[Q],\; E[Q^+]=E[Q^-]=E[Q]$, $[0,+\infty)$ is in
the (time) domain of definition of $Q^{\pm}$ and there exists
$e_0>0$ such that
$$
\forall t\geq 0,\quad \left\|Q^{\pm}(t)-e^{it}Q\right\|_{H^1}\leq
Ce^{-e_0t},
$$
\item
$\ds \|\nabla Q^-_0\|_2<\|\nabla Q\|_2$, $Q^-$ is globally defined
and scatters for negative time,
\item
$\ds \|\nabla Q^+_0\|_2>\|\nabla Q\|_2$, and the negative time of
existence of $Q^+$ is finite.
\end{enumerate}
\end{theorem}

Next, they characterize all solutions at the critical mass-energy
level as follows:

\begin{theorem}[Duyckaerts-Roudenko {\cite{DR}}]
\label{T:DR}
Let $u$ be a solution of NLS satisfying $M[u]E[u]=M[Q]E[Q]$.
\begin{enumerate}
\item \label{case_sub}
If $\eta(0)<1$, then
either $u$ scatters or $u=Q^-$ up to the symmetries.
\item
 \label{case_critical}
If $\eta(0)=1$, then
$u=e^{it}Q$ up to the symmetries.
\item
 \label{case_super}
If $\eta(0)>1$, and $u_0$
is radial or of finite variance, then either the interval of
existence of $u$ is of finite length or $u=Q^+$ up to the
symmetries.
\end{enumerate}
\end{theorem}

A recent result of Beceanu \cite{Bec} on \eqref{E:NLSa} states that
near the ground state soliton (and its Galilean, scaling and phase
transformations) there exists a real analytic (center-stable)
manifold in $\dot{H}^{1/2}$ such that any initial data taken from it
will produce a global in time solution decoupling into a moving
soliton and a dispersive term scattering in $\dot{H}^{1/2}$.

In part I of this paper, we provide some alternate criteria for
blow-up in the spirit of Lushnikov \cite{Lu95}.  First, we state his
result, adapted to our notation.  Due to the complexity of the
formulas, we will write $M=M[u]$, $E=E[u]$, etc.  For simplicity we
restrict to the case $E>0$, since $E\leq 0$ is comparately well
understood from Theorem \ref{T:DHR}.

\begin{theorem}[adapted from Lushnikov {\cite{Lu95}}]
\label{T:Lushnikov}
Suppose that $u_0\in H^1$ and $\|xu_0\|_{L^2}<\infty$.
The following is a sufficient condition for blow-up in finite time:
\begin{equation}
 \label{E:L}
\frac{V_t(0)}{M} < 2\sqrt 3 \; g\left( \frac{8E V(0)}{3M^2} \right),
\end{equation}
where
\begin{equation}
\label{E:g-def}
g(\omega) =
\begin{cases}
\sqrt{ \frac{2}{\omega^{1/2}}+\omega-3} & \text{if }0<\omega\leq 1 \\
-\sqrt{ \frac{2}{\omega^{1/2}}+\omega-3} & \text{if }\omega \geq 1,
\end{cases}
\end{equation}
which is graphed in Figure \ref{F:g-graph}.
\end{theorem}

\begin{figure}
\includegraphics[scale=0.7]{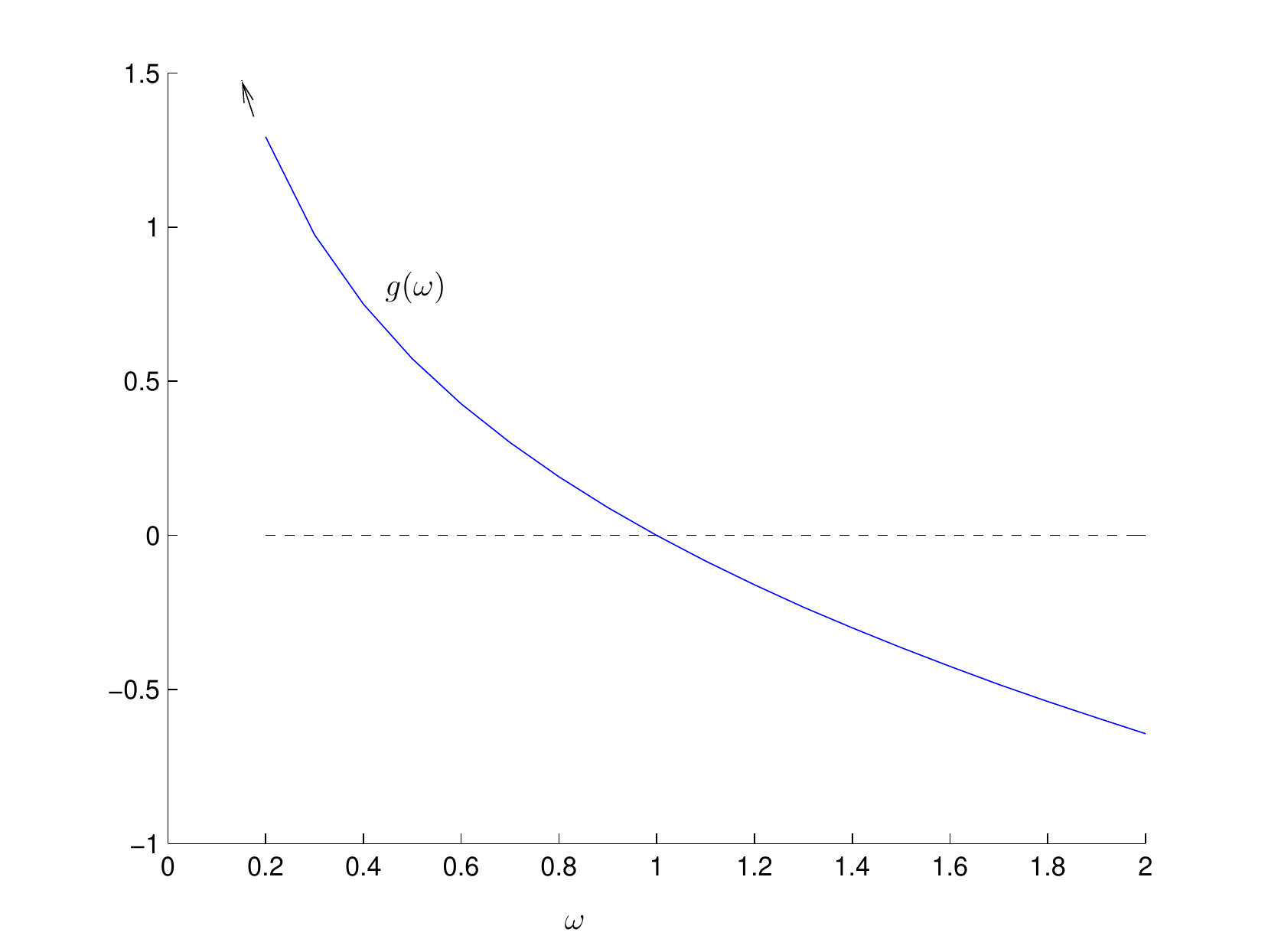}
\caption{A plot of $g(\omega)$ versus $\omega$, where $g$ is defined
in \eqref{E:g-def}.  This function appears in the blow-up conditions
in Theorems \ref{T:Lushnikov} and \ref{T:Lushnikov-adapted}.}
 \label{F:g-graph}
\end{figure}
For an explicit formulation of the condition \eqref{E:L} refer to \S
\ref{S:reformulation}, in particular, when initial datum is
real-valued \eqref{E:L} becomes \eqref{E:Lreal} and when it is
complex-valued the condition rewrites as in
\eqref{E:L+}-\eqref{E:L-}.

Theorem \ref{T:Lushnikov} is based upon use of the uncertainty principle,
\begin{equation}
 \label{E:uncertainty}
\|u\|_{L^2}^4 + \frac49 \left| \Im \int (x\cdot \nabla u) \bar u
\,dx \right|^2 \leq \frac49 \|xu\|_{L^2}^2\|\nabla u\|_{L^2}^2 \,,
\end{equation}
the virial identity, and a ``mechanical analysis'' of the resulting
second-order ODE in $V(t)$.  By replacing \eqref{E:uncertainty} with
\begin{equation}
 \label{E:interpolation}
\|u\|_{L^2} \leq \left( \frac{2^2\cdot 7^5 \cdot \pi^2}{3^5\cdot
5^2}\right)^\frac{1}{14} \|xu\|_{L^2}^\frac37\|u\|_{L^4}^\frac47 ,
\end{equation}
we can obtain a different condition which in some cases improves
upon Theorem \ref{T:Lushnikov}.  The inequality
\eqref{E:interpolation} can be thought of as a variant of the
H\"older interpolation inequality $\|u\|_{L^2} \leq
\|u\|_{L^{6/5}}^{3/7}\|u\|_{L^4}^{4/7}$, since $\|u\|_{L^{6/5}}$ and
$\|xu\|_{L^2}$ scale the same way.  Both inequalities
\eqref{E:uncertainty} and \eqref{E:interpolation} are stated here
with sharp constants and are proved in \S\ref{S:inequalities}.

\begin{theorem}
\label{T:Lushnikov-adapted}
Suppose that $u_0\in H^1$ and $\|xu_0\|_{L^2}<\infty$.
The following is a sufficient condition for blow-up in finite time:
\begin{equation}
 \label{E:LA}
\frac{V_t(0)}{M} < \frac{2\sqrt 2 (ME)^\frac16 }{C^\frac73}  \;
g\left( \frac{4C^\frac{14}{3}E^\frac23}{M^\frac73} V(0)\right)\,,
\qquad C= \left( \frac{2^2\cdot 7^5 \cdot \pi^2}{3^5\cdot
5^2}\right)^\frac{1}{14} ,
\end{equation}
where $g$ is defined in \eqref{E:g-def} and graphed in Figure
\ref{F:g-graph}.
\end{theorem}
For an explicit reformulation of \eqref{E:LA} refer to \S
\ref{S:reformulation}, in particular, for the real-valued initial
datum it becomes \eqref{E:LAreal} and for the complex-valued datum
it is equivalent to \eqref{E:LA+} - \eqref{E:LA-}.

Note that \eqref{E:LA} can be put into the form
$$
\frac{V_t(0)}{M} < 2\sqrt{3} \, \tilde g  \left( \frac{8EV(0)}{3M^2}
\right)\,, \qquad \tilde g(\omega) = \mu \, g(\mu^{-2}\omega) \,,
\qquad \mu = \frac{\sqrt 2 (ME)^\frac16}{\sqrt 3 C^\frac73} ,
$$
which offers a comparison between Theorem \ref{T:Lushnikov} and
\ref{T:Lushnikov-adapted}. These conditions should be compared to
the sufficient condition for $E[u]>0$ in the 2d case, namely
\eqref{E:2dblow-up}.

The approach via the interpolation inequality
\eqref{E:interpolation} also allows us to prove a radial,
infinite-variance version of Theorem \ref{T:Lushnikov-adapted} based
upon a local virial identity, Strauss' radial Gagliardo-Nirenberg
inequality \cite{S}, and a bootstrap argument.  Select a smooth
radial (nonstrictly) increasing function $\psi(x)$ such that
$\psi(x)=|x|^2$ for $0\leq |x|\leq 1$ and $\psi(x)=2$ for $|x|\geq
2$.  Define the \emph{localized variance}
\begin{equation}
\label{E:lv-def}
V_R \defeq \int R^2\psi(x/R)|u(x)|^2 \, dx \,.
\end{equation}
Note that by the dominated convergence theorem, for any $u_0\in
L^2$, we have
$$
\lim_{R\to +\infty} \frac{V_R(0)}{R^2M} =0 \,.
$$
Thus, there always exists $R$ such that \eqref{E:lv-initial} below
holds.

\begin{theorem}
 \label{T:Lushnikov-radial}
Suppose $ME>1$.  Fix $\delta \ll 1$ (the smallness depends only on
$\psi(x)$ in \eqref{E:lv-def}).  Given $u_0\in H^1$ radial, take any $R$
such that
\begin{equation}
 \label{E:lv-initial}
\frac{V_R(0)}{M} \leq \frac12 R^2 \,, \qquad R^2 \gtrsim
\frac{M^2}{\delta}
\end{equation}
(the implicit constant in the second inequality again depends only
on $\psi(x)$ in \eqref{E:lv-def}).  Then the following is a
sufficient condition for blow-up in finite time:
\begin{equation}
 \label{E:Linf}
\begin{aligned}
\frac{(V_R)_t(0)}{M} < \frac{\sqrt 6 (8+\delta)^\frac16
(1-\delta)^{\frac13} (ME)^\frac16}{(C_{\infty})^\frac73} \; g\left(
\frac{(8+\delta)^\frac23}{(1-\delta)^\frac23} \frac{
(C_{\infty})^\frac{14}{3}E^\frac23}{M^\frac73} V_R(0)\right)\,, &\\
C_\infty = \left( \frac{2^{11}  \pi^2}{3^2} \right)^\frac{1}{14} \,, &
\end{aligned}
\end{equation}
where $g$ is defined in \eqref{E:g-def} and graphed in Figure
\ref{F:g-graph}.
\end{theorem}

One way to generate examples of $u_0$ satisfying the hypotheses of
Theorem \ref{T:Lushnikov-radial} but not Theorem
\ref{T:Lushnikov-adapted} is to take any of the examples detailed
below for which Theorem \ref{T:Lushnikov-adapted} applies, but tack
on a slowly decaying tail of infinite variance but at very large
radii.  For example, redefine $u_0$ at very large radii to be
$u_0(x) = |x|^{-2}$. However, the main merit of Theorem
\ref{T:Lushnikov-radial} is the availability, in the case of real
initial data, of a conceptual interpretation in terms of the way in
which mass is initially distributed.

\begin{corollary}
 \label{C:conceptual}
There is $0<\delta\ll 1$ such that the following holds. Suppose that
$ME>1$, $u_0\in H^1$ is radial and real,
\begin{equation}
 \label{E:mass-dist}
\frac{1}{M} \int_{|x|\geq \delta^{1/2} M(ME)^{-1/3}} |u_0|^2 dx \leq
\delta^2 (ME)^{-2/3} \,.
\end{equation}
Then blow-up occurs in finite time.
\end{corollary}

Note that the quantity on the left-side of \eqref{E:mass-dist} is
the fraction of initial mass occurring outside the ball of radius
$\delta M (ME)^{-1/3}$.  Thus, \eqref{E:mass-dist} states that most
mass is \emph{inside} the ball of radius $\delta M (ME)^{-1/3}$, and
we intuitively expect an initially highly concentrated real solution
to blow-up.  By the scaling \eqref{E:scaling}, it is natural that
the radius scales linearly with $M$.

Theorems \ref{T:Lushnikov}, \ref{T:Lushnikov-adapted}, and
\ref{T:Lushnikov-radial}, and Corollary \ref{C:conceptual} are
proved in \S\ref{S:Lushnikov}, and the reformulation of the
conditions of Theorems \ref{T:Lushnikov} and
\ref{T:Lushnikov-adapted} is in \S \ref{S:reformulation}.

In the remainder of the paper,
\S\ref{S:sechprofile}--\ref{S:osc-Gaussian}, we examine several
specific \emph{radial} initial data given as profiles with several
parameters.\footnote{Since we work exclusively with radial data, we
write our functions as functions of $r\in (0,+\infty)$, but keep in
mind that we are studying the 3d NLS equation.}  In each case, we
report the predictions given by Theorems \ref{T:DHR},
\ref{T:Lushnikov}, and \ref{T:Lushnikov-adapted}, and also the
results of numerical simulations. In general, Theorems
\ref{T:Lushnikov}, \ref{T:Lushnikov-adapted} may not give better
results than Theorem \ref{T:DHR} (we have one example where it does
not, see Figure \ref{F:sechp}), however, we show that they give new
information in many other cases, in particular, when initial data
has a negative value of $V_t(0)$.

In \S\ref{S:sechprofile}, we consider initial data
$$
u_0(x) = \lambda^{3/2}Q(\lambda r) \,e^{i\gamma r^2} ,
$$
where $Q$ is the ground state solution to \eqref{E:Q}.  The case
$\gamma=0$ is completely understood by Theorem \ref{T:DHR}.  In
fact, $\lambda^{3/2}Q(\lambda r)$ corresponds to the parabolic-like
boundary curve in Figure \ref{F:dichotomy} which lies below the
$M[u]E[u]=M[Q]E[Q]$ line for all $\lambda \neq 1$, and we have
blow-up for $\lambda>1$ and scattering for $\lambda<1$.  The case
$\gamma \neq 0$ is more interesting; in particular, for negative
phase, $\gamma < 0$, Theorems \ref{T:Lushnikov} and
\ref{T:Lushnikov-adapted} give new range on parameters $(\lambda,
\gamma)$ for which there will be a blow up, see Figure
\ref{F:sechn}. Although for positive phase, $\gamma > 0$, there is
no new information on blow up other than provided by Theorem
\ref{T:DHR}, we show the numerical results for the blow up threshold
in Figure \ref{F:sechp}, in particular, there is no blow up for
$\lambda < 1$ is expected.

In \S\ref{S:Gaussian}, we consider Gaussian initial data
$$
u_0(x) = p \, e^{-\alpha r^2/2} e^{i\gamma r^2} \,.
$$
By scaling, it suffices to consider $\gamma =0, \pm \frac 12$.  In
the real case $\gamma=0$, the behavior will be a function of
$p/\sqrt \alpha$, and the results are depicted in Figure
\ref{F:all-gauss-nophase}.  The case $\gamma=\frac12$ appears in
Figure \ref{F:gaussp}, and the case $\gamma=-\frac12$ appears in
Figure \ref{F:gaussn}. For this initial data Theorem
\ref{T:Lushnikov} gives the best range for blow up, although Theorem
\ref{T:Lushnikov-adapted} gives an improvement over the Theorem
\ref{T:DHR}.

In \S\ref{S:super-Gaussian}, we consider ``super-Gaussian'' initial
data
$$
u_0(x) = p \, e^{-\alpha r^4/2} e^{i\gamma r^2} \,.
$$
By scaling, again it suffices to consider $\gamma =0, \pm \frac 12$.
In the real case $\gamma=0$, the behavior will be a function of
$p/\alpha^{1/4}$, and is depicted in Figure
\ref{F:supergauss-nophase}. The case $\gamma=\frac12$ is presented
in Figure \ref{F:supergaussp}, and the case $\gamma=-\frac12$ is
presented in Figure \ref{F:supergaussn}. For this initial data
Theorem \ref{T:Lushnikov-adapted} gives the best theoretical range
for blow up.

In \S\ref{S:off-Gaussian}, we consider ``off-centered Gaussian''
initial data
$$
u_0(x) = p \, r^2 e^{-\alpha r^2} e^{i\gamma r^2} \,.
$$
By scaling, again it suffices to consider $\gamma =0, \pm \frac 12$.
In the real case $\gamma=0$, the behavior will be a function of
$p/\alpha^{3/4}$, and the results are presented in Figure
\ref{F:all-off-gauss-nophase}. The case $\gamma=+\frac12$ is given
in Figure \ref{F:off-gaussp} and the case $\gamma=-\frac12$ is given
in Figure \ref{F:off-gaussn}. For this initial data Theorem
\ref{T:Lushnikov-adapted} gives as well the best theoretical range
for blow up.

In \S\ref{S:osc-Gaussian}, we consider ``oscillatory Gaussian''
initial data
$$
u_0(x) = p \, \cos (\beta r) \, e^{-r^2} e^{i\gamma r^2} \,.
$$
We restrict our attention to $\gamma =0, \pm \frac 12$, presented in
Figures \ref{F:oscgauss}, \ref{F:oscgaussp}, and \ref{F:oscgaussn},
respectively. For the oscillatory Gaussian the best theoretical
range on blow up threshold is provided by a combination of Theorems
\ref{T:Lushnikov}, \ref{T:Lushnikov-adapted}: for small
oscillations, $\beta \lesssim 1$, Theorem \ref{T:Lushnikov} is
stronger, and for fast oscillations, $\beta \gtrsim 1$, Theorem
\ref{T:Lushnikov-adapted} provides a better range (for exact values
see the above Figures).

The numerics described in
\S\ref{S:sechprofile}--\ref{S:osc-Gaussian} provide evidence to
support the following conjectures.

\smallskip

\noindent\textbf{Conjecture 1}. For each $\epsilon>0$, there exists
radial Schwartz initial data $u_0$ for which $M[u]E[u]>M[Q]E[Q]$,
$\|u_0-Q\|_{H^1}<\epsilon$, $u(t)$ scatters as $t\to -\infty$, and
$u(t)$ blows-up in finite forward time.

\smallskip

That is, there exist initial data arbitrarily close to D in Figure
\ref{F:dichotomy} with the property that the backward time evolution
results in scattering but the forward time evolution results in
finite time blow-up.  The numerical evidence of the existence of
such solutions is a consequence of the study of initial data of the
form $\lambda^{3/2}Q(\lambda r) e^{i\gamma r^2}$ in
\S\ref{S:sechprofile}.  Take $\lambda<1$ but close to $1$.  Then we
find that there exists a curve $\gamma_0(\lambda)$ such that
$\lim_{\lambda \nearrow 1} \gamma_0(\lambda)= 0$ with the following
property:  If $|\gamma|> \gamma_0(\lambda)$, then
$M[u]E[u]>M[Q]E[Q]$, and $u_0$ evolves to a solution $u(t)$ blowing
up in finite positive time if $\gamma<-\gamma_0(\lambda)$ but $u_0$
evolves to a scattering solution in positive time if
$\gamma>\gamma_0(\lambda)$ (see Figures \ref{F:sechn} and
\ref{F:sechp}).  By the time reversal property
$$
u(t) \text{ solves NLS }\implies \bar u(-t) \text{ solves NLS}
$$
we conclude that  if $\gamma<-\gamma_0(\lambda)$, $u(t)$ scatters
\emph{backward} in time.  We can take $\lambda$ as close to $1$ and
$\gamma$ as close to $0$ as we please (while maintaining
$\gamma<-\gamma_0(\lambda)$, establishing the claimed conjecture (as
a result of observed numerical behavior).

Before proceeding, let us remark on some consequences assuming this
conjecture is valid.  We see from Figure \ref{F:dichotomy} that the
smallest admissible value of $\eta(0)^2$ that could lead to a
finite-time blow-up solution is $\frac13+$  (Corner B).

\smallskip

\noindent\textbf{1st Corollary of Conjecture 1}.  For each
$\epsilon>0$, there exist initial data $u_0$ with
$M[u]E[u]>M[Q]E[Q]$ and $\eta^2(0)< \frac13+\epsilon$ for which the
evolution $u(t)$ blows-up in finite time.  For each $N \gg 1$, there
exists initial data $u_0$ with $M[u]E[u]>M[Q]E[Q]$ and $\eta^2(0)
\geq N$ leading to a scattering solution $u(t)$.

\smallskip

More loosely stated, there exist initial data as close to point B in
Figure \ref{F:dichotomy} leading to finite-time blow-up solutions,
and there exist initial data as far to the right in the direction of
E in Figure \ref{F:dichotomy} leading to scattering solutions. This
establishes the irrelevance of the size of $\eta(0)$ in predicting
blow-up or scattering in the case $M[u]E[u]>M[Q]E[Q]$.

This follows from Conjecture 1 as follows. Taking $u(t)$ to be a
solution of the type described in Conjecture 1, note that
$\lim_{t\searrow -\infty} \eta^2(t) = \frac13$.  Hence, for some
large negative time $-T$, we have $\frac13 <\eta^2(-T)<
\frac13+\epsilon$.  Resetting, by time translation, to make time
$-T$ into time $0$ gives the first type of solution described here.
On the other hand, if $T^*$ denotes the blow-up time of $u(t)$, then
$\lim_{t\nearrow T^*} \eta^2(t) = +\infty$.  Therefore, there exists
a time $T'<T^*$ but close to $T^*$ such that $\eta^2(T')>N$.
Applying the time reversal symmetry and time translation gives the
second type of solution described here.

\smallskip

\noindent\textbf{2nd Corollary of Conjecture 1}.  For each
$\epsilon>0$, there exists initial data $u_0$ with
$\|u_0\|_{L^3}<\epsilon$ for which $u(t)$ blows-up in finite time.
For each $N \gg 1$, there exists initial data $u_0$ for which
$\|u_0\|_{L^3} \geq N$ and for which $u(t)$ scatters.

\smallskip

Stated more loosely, the quantity $\|u_0\|_{L^3}$ is irrelevant to
predicting blow-up or scattering.  Thus, the critical Lebesgue norm
gives no prediction of dynamical behavior in the 3d case; note the
contrast with the 1d case discussed above, where the critical
Lebesgue norm $L^1$ determines the threshold for soliton formation.

This follows from Conjecture 1 by the same type of reasoning used to
justify the first corollary, since if $u(t)$ scatters in negative
time we have $\lim_{t\searrow -\infty} \|u(t)\|_{L^3} = 0$ and if
$u(t)$ blows-up in finite forward time $T^*$, we have
$\lim_{t\nearrow T^*} \|u(t)\|_{L^3} = +\infty$.  This latter fact
was proved by Merle-Rapha\"el \cite{MR}.

\smallskip

What about the $\dot H^{1/2}$ norm? The ``small data scattering
theory'' (essentially a consequence of the Strichartz estimates --
see \cite{HR2} for exposition) states that there exists $\delta>0$
such that if $\|u_0\|_{\dot H^{1/2}}<\delta$, then $u(t)$ scatters
in both time directions.  It is then natural to ask whether $\delta$
in the above statement can be improved to $\|Q\|_{\dot H^{1/2}}$.
\begin{theorem}
There exist radial initial data
$u_0$ for which $\|u_0\|_{\dot H^{1/2}}<\|Q\|_{\dot H^{1/2}}$ and
$u(t)$ blows-up in finite forward time.
\end{theorem}
This follows from Theorems \ref{T:Lushnikov} and
\ref{T:Lushnikov-adapted} by considering certain Gaussian initial
data with negative phase (see Figure \ref{F:gaussn}).  It needs to
be remarked, however, that this theorem relies on one piece of
numerical information -- that value $\|Q\|_{\dot H^{1/2}}^2 =
27.72665$.  It is an analytical result in the sense that one need
not numerically solve the NLS equation.

This analytical result is further supported numerically for a
variety of \emph{nonreal} initial data with the inclusion of
\emph{negative} quadratic phase:  $u_0(x)= \phi(x)e^{i\gamma |x|^2}$
with $\phi$ radial and real-valued and $\gamma<0$.   We note,
however, that we did not observe any \emph{real-valued} initial data
$u_0$ with $\|u_0\|_{\dot H^{1/2}}<\|Q\|_{\dot H^{1/2}}$ evolving
toward finite-time blow-up solutions.  Hence, we pose the following
conjecture:

\smallskip

\noindent\textbf{Conjecture 3}. If the initial data $u_0$ is
real-valued and $\|u_0\|_{\dot H^{1/2}}<\|Q\|_{\dot H^{1/2}}$, then
$u(t)$ scatters as $t\to -\infty$ and $t\to +\infty$.

\smallskip

Assuming Conjecture 3 holds, we elaborate further in Conjecture 4
below.  When working with a one-parameter family of profiles, the
$\dot H^{1/2}$ norm can apparently predict both scattering \emph{and
blow-up} if the profiles are monotonic; this is summarized in our
next conjecture.

\smallskip

\noindent\textbf{Conjecture 4}. Consider a real-valued radial
initial data profile $\psi(r)$ that is strictly decreasing as $r\to
+\infty$.  Let $\alpha_0 = \|Q\|_{\dot H^{1/2}}/\|\psi\|_{\dot
H^{1/2}}$.  Then the solution with initial data $u_0(x)=\alpha
\psi(x)$ scatters if $\alpha<\alpha_0$ and blows-up if
$\alpha>\alpha_0$.

\smallskip

However, if the profile is not monotonic, then the $\dot H^{1/2}$
norm appears to only give a sufficient condition for scattering.
This is illustrated in the simulations for oscillatory Gaussian data
-- see Figure \ref{F:oscgauss}.

\subsection{NLS as a model in physics}
\label{S:physics}

\quad

\emph{1. Laser propagation in a Kerr medium \cite{SS, Fib}}. This
model is inherently two dimensional (in $xy$) and the time $t$ in
fact represents the $z$-direction, and is derived via the
\emph{paraxial approximation} for the Helmholtz equation.  The
nonlinearity arises from the dependence of the index of refraction
on the amplitude of the propagating wave.  Blow-up in finite time is
observed in the laboratory as a sharp focusing of the propagating
wave.  Ultimately, the non-backscattering assumption, and hence the
NLS model, breaks down.

\smallskip

\emph{2. Langmuir turbulence in a weakly magnetized plasma
\cite{Zak72, SS}}.  A plasma is modeled as interpenetrating fluids
of highly excited electrons and positive ions.   The Langmuir waves
propagate through the electron medium.  The principle mathematical
model is the Zakharov system \cite{Zak72}, which is a nonlinearly
coupled Schr\"odinger and wave system.  The Schr\"odinger function
is a slowly varying envelope for the electric potential and the wave
function is the deviation of the ion density from its mean value.
The NLS equation arises as the subsonic limit of the Zakharov
system, which is obtained by sending the wave speed $\to +\infty$.
Blow-up in finite time is the central phenomenon of study in
\cite{Zak72}, since it predicts the formation of a \emph{cavern} of
shrinking radius confining fast oscillating electrons whose
collisions dissipate energy (at which point the model breaks down).
The Zakharov model is inherently 3d, although certain experimental
configurations can be modeled with the 1d or 2d equations.

\smallskip

\emph{3. Bose-Einstein condensate (BEC) \cite{DGPS}}. BEC consists
of ultracold (a few nK) dilute atomic gases where $u$ gives the wave
function (e.g. the number of atoms in a region $E$ is $\int_E
|u|^2$) and the coefficient of the nonlinear term is related to the
\emph{scattering length} by $g=8\pi a$.  The scattering length
depends upon the interatomic potential and can be either positive or
negative.  While the model is inherently 3d, the imposition of a
strong confining potential in one or two directions can effectively
reduce the model to two or one dimensions, respectively. Experiments
showing blow-up are reported for ${}^{85}$Rb condensates in
\cite{Rb-exper} and for ${}^7$Li condensates in \cite{Li-exper}.

\smallskip

In each situation, blow-up is physically observed (although of
course the model breaks down at some point prior to the blow-up
time).

\subsection{Acknowledgements}
S.R. thanks Pavel Lushnikov for bringing to her attention his 1995
paper on the dynamic collapse criteria. J.H. and S.R. are grateful
to Gadi Fibich for discussion and remarks on a preliminary version
of this paper.  S.R. is partially supported by NSF grant
DMS-0808081. J.H. is partially supported by a Sloan fellowship
and NSF grant DMS-0901582.

\section{Inequalities}
\label{S:inequalities}

In this section, we prove the two inequalities \eqref{E:uncertainty}
and \eqref{E:interpolation} needed for Theorems \ref{T:Lushnikov}
and \ref{T:Lushnikov-adapted}.

\subsection{Inequality \eqref{E:uncertainty}}

Of course \eqref{E:uncertainty}, the uncertainty principle, is
standard, although we include a proof for completeness. By
integration by parts,
$$
\|u\|_{L^2}^2 = \frac13 \int (\nabla \cdot x) \; |u|^2 \, dx
= -\frac23 \Re \int (x \cdot \nabla u) \, \bar u \, dx \,.
$$
By Cauchy-Schwarz,
$$
\frac94 \|u\|_{L^2}^4 + \left| \Im \int (x \cdot \nabla u) \, \bar u
\,dx \right|^2 = \left| \int (x\cdot \nabla u) \, \bar u \, dx
\right|^2 \leq \|xu\|_{L^2}^2 \|\nabla u\|_{L^2}^2.
$$
This provides the sharp constant since the inequality is achieved by
only the above application of Cauchy-Schwarz, which is saturated
when $xu = C \nabla u$, i.e., when $u$ is a Gaussian.

\subsection{Inequality \eqref{E:interpolation}}

In this section, we prove the following
\begin{proposition}
\label{P:interp}
The inequality
\begin{equation}
 \label{E:interp1}
\|u\|_{L^2} \leq C\|xu\|_{L^2}^\frac37\|u\|_{L^4}^\frac47
\end{equation}
holds with sharp constant $C = \left( \frac{2^2\cdot 7^5 \cdot
\pi^2}{3^5 \cdot 5^2} \right)^{1/14} \approx 1.3983$.  Moreover, all
functions for which equality is achieved are of the form $\beta
\phi(\alpha x)$, where
$$
\phi(x) =
\begin{cases}
(1-|x|^2)^{1/2} & \text{if } 0 \leq |x|\leq 1 \\
0 & \text{if } |x|> 1 .
\end{cases}
$$
\end{proposition}

First we prove there exists \emph{some} constant $C$ for which
\eqref{E:interp1} holds. Given $u$ such that $\|xu\|_{L^2}<\infty$
and $\|u\|_{L^4}<\infty$, define $v$ by $u(x)=\alpha v(\beta x)$
with $\alpha$ and $\beta$ chosen so that $\|xv\|_{L^2}=1$ and
$\|v\|_{L^4}=1$, i.e.,
$$
\alpha = \|u\|_{L^4}^\frac{10}7 \|xu\|_{L^2}^{-\frac37}\,, \qquad
\beta = \|u\|_{L^4}^{\frac47} \|xu\|_{L^2}^{-\frac47} \,.
$$
By H\"older,
\begin{align*}
\|v\|_{L^2}^2
& = \|v\|_{L^2{(|x|\leq 1)}}^2 + \|v\|_{L^2{(|x|\geq 1})}^2\\
& \leq \left(\frac43\,\pi\right)^{1/2} \|v\|_{L^4({|x|\leq 1})}^2
+ \|xv\|_{L^2({|x|\geq 1})}^2 \\
& \leq \left(\frac43\,\pi\right)^{1/2}+1,
\end{align*}
which completes the proof of \eqref{E:interp1} with nonsharp
constant $C= ( (\frac43\pi)^{1/2}+1)^{1/2}\approx 1.7455$.

\begin{remark}
We can optimize the above splitting argument by splitting at radius
$r= \left( \frac{4}{3\pi} \right)^{1/7}$, which gives the constant
$C\approx 1.7265$.  However, this is still off from the sharp
constant $C\approx 1.3983$ stated in Prop. \ref{P:interp}.  The
reason is that the estimates applied after the splitting lead to an
optimizing function which is the characteristic function of the ball
of radius $r$.  However, such a function fails to have both
$\|u\|_{L^4}=1$ and $\|xu\|_{L^2}=1$.
\end{remark}

We now proceed to identify the sharp constant $C$ in
\eqref{E:interp1} and the family of optimizing functions.  Consider
the Lagrangian
$$
L(\phi) = \frac{\|x\phi\|_{L^2}^\frac37
\|\phi\|_{L^4}^\frac47}{\|\phi\|_{L^2}}\,,
$$
defined on $X\defeq L^4\cap L^2(\la x \ra^2 dx)$, $x \in \cR^3$.

\begin{lemma} \quad
 \label{L:minimizer}
\begin{enumerate}
\item
Any minimizer $\phi$ (without loss of generality taken real-valued
and nonnegative) of $L$ in $X$ is radial and (nonstrictly)
decreasing.
\item
There exists a minimizer $\phi$ of $L$ in $X$.
\end{enumerate}
\end{lemma}

\begin{remark}
Lemma \ref{L:minimizer} does not say that a minimizer $\phi$ needs
to be continuous, \emph{strictly} decreasing, or compactly
supported. We only discover this to be the case in the next step of
the proof.
\end{remark}

\begin{proof}
We first argue that any $\phi\in X$ can be replaced by a radial,
monotonically (perhaps not strictly) decreasing function $\tilde
\phi$ such that
\begin{equation}
 \label{E:Lp-norms-same}
\|\tilde \phi \|_{L^2} = \|\phi \|_{L^2} \,, \quad \|\tilde \phi
\|_{L^4} = \|\phi \|_{L^4} \,,
\end{equation}
\begin{equation}
 \label{E:var-dec}
\|x\tilde \phi \|_{L^2} \leq \|x\phi \|_{L^2} \,.
\end{equation}
Moreover, we have equality in \eqref{E:var-dec} if and only if
$\tilde \phi = \phi$, which occurs if and only if $\phi$ itself is
radial, nonnegative, and (nonstrictly) decreasing. Given $\phi$, let
$$
E_{\phi, \lambda} = \{ \, x \, | \, |\phi(x)|> \lambda \, \} \,,
\qquad \mu_\phi (\lambda) = |E_{\phi,\lambda}|.
$$
Recall that (by Fubini's theorem)
\begin{equation}
 \label{E:Lp-dist}
\|\phi \|_{L^p}^p = \int_{\lambda=0}^{+\infty} p\lambda^{p-1}
\mu_\phi(\lambda) \, d\lambda \,.
\end{equation}
Note that $\mu_\phi: [0,+\infty) \to [0,+\infty)$ is a nonstrictly
decreasing right-continuous function.  Thus,
$$
\lim_{\sigma\searrow \lambda} \mu_\phi(\lambda) = \mu_\phi(\lambda)
\leq \lim_{\sigma\nearrow \lambda} \mu_\phi(\sigma)
$$
with equality if and only if $\lambda$ is a point of continuity.
Since $\mu_\phi$ is \emph{nonstrictly} decreasing, there may be
intervals on which $\mu_\phi$ is constant that we need to exclude in
order to achieve invertibility. Let $F_\phi$ be the set of all $r$
such that $\mu_\phi ^{-1}(\{\frac43 \pi r^3\})$ has positive measure
(is a nontrivial interval).  Note that $F_\phi$ is an at most
countable set (because $\mu_\phi$ is nonstrictly decreasing). For
all $r\notin F$, define the radial function $\tilde \phi$ on
$\mathbb{R}^3$ by
$$
\tilde \phi(r) = \lambda \quad \text{ iff } \quad \mu_\phi(\lambda)
\leq \frac43\pi r^3 \leq \lim_{\sigma \nearrow \lambda} \mu_\phi
(\sigma) \,.
$$
Note that $\mu_\phi = \mu_{\tilde \phi}$, i.e., $\phi$ and $\tilde
\phi$ are equidistributed.  By \eqref{E:Lp-dist}, we have
\eqref{E:Lp-norms-same}.  Let
$$
h_\phi(\lambda) \defeq \int_{E_{\phi,\lambda}} |x|^2 \,dx.
$$
By Fubini,
$$
\|x\phi\|_{L^2}^2 = 2\int_{\lambda=0}^{+\infty}   \lambda
h_\phi(\lambda) \, d\lambda \,,
$$
and similarly for $\tilde \phi$. For each $\lambda$, we have
$h_{\tilde \phi}(\lambda) \leq h_\phi(\lambda)$, since $|E_{\tilde
\phi,\lambda}|=|E_{\phi,\lambda}|$ but $E_{\tilde\phi ,\lambda}$ is
uniformly positioned around the origin (it is a ball centered at
$0$).  Consequently, \eqref{E:var-dec} holds.

We note that the above argument establishes (1) in the theorem
statement.  To prove (2), we need to construct a minimizer by a
limiting argument.  Let
$$
m  \defeq \inf_{\phi \in X} L(\phi) \,.
$$
Let $\phi_n$ be a minimizing sequence.  By approximation and the
above argument, we can assume that each $\phi_n$ is continuous,
compactly supported, radial, nonnegative, and nonstrictly
decreasing.  By scaling ($\phi_n(x) \mapsto \alpha\phi_n(\beta x)$)
we can also assume that $\|\phi_n\|_{L^2}=1$ and
$\|\phi_n\|_{L^4}=1$.  We have $\|x\phi_n\|_{L^2} \leq
\|x\phi_1\|_{L^2} \defeq A$.  Since for each $n$, the function
$\phi_n(r)$ is decreasing in $r$, we have
$$
A^2 \geq \|x\phi_n\|_{L^2{(|x|\leq r)}}^2 = \int_{\rho=0}^r 4\pi
\rho^4 |\phi_n(\rho)|^2 \,d \rho \geq |\phi_n(r)|^2 \frac{4\pi
r^5}{5} \,,
$$
i.e., $|\phi_n(r)| \lesssim r^{-5/2}$. Similarly, working with the
fact that $\|\phi_n\|_{L^4} =1$, we have that for all $n$,
$|\phi_n(r)| \lesssim r^{-3/4} \,.$ Combining these two pointwise
bounds, we have
\begin{equation}
 \label{E:ptwise}
|\phi_n(r)| \lesssim r^{-3/4}(1+r)^{-7/4} \,,
\end{equation}
with implicit constant uniform in $n$. Thus, for each $r>0$, the
sequence of nonnegative numbers $\{ \phi_n(r) \}_{n=1}^{+\infty}$ is
bounded and hence has a convergent subsequence.  By a diagonal
argument, we can pass to a subsequence of $\phi_n$ (still labeled
$\phi_n$) such that for each $r\in \mathbb{Q}$, we have that
$\lim_{n\to +\infty} \phi_n(r)$ exists. Denote by $\phi(r)$ the
limiting function (for now only defined on $\mathbb{Q}$).

We claim that $\phi_n$ converges pointwise a.e.\footnote{Note that
$\phi_n$ is a sequence of functions each of which is decreasing, but
it is not the case that for each $r$, the sequence of numbers
$\phi_n(r)$ is decreasing.  Thus, proving that $\phi_n(r)$ converges
for a.e. $r>0$ is a little more subtle.}  Pass to a subsequence
(still labeled $\phi_n$) such that for each $r\in 2^{-n}\mathbb{N}$,
$\frac1n < r < n$, we have $|\phi_n(r)-\phi(r)| \leq 2^{-2n}$.  Let
$$
\phi_n^-(r) \defeq \phi((j+1)2^{-n})-2^{-2n} \text{ for }j2^{-n}< r<
(j+1)2^{-n},
$$
$$
\phi_n^+(r) \defeq \phi(j 2^{-n})+2^{-2n} \text{ for } j2^{-n}< r<
(j+1)2^{-n}.
$$
Then for each $n$,
$$
\phi_n^-(r) \leq \phi_n(r) \leq \phi_n^+(r)
$$
and $\phi_n^-$ is a pointwise nonstrictly increasing sequence of
functions, while $\phi_n^+$ is pointwise nonstrictly decreasing
sequence of functions. Since $\phi$ is a decreasing function,
$$
\sum_{j=2^{n-\log n}}^{n2^n} (\phi(j2^{-n}) - \phi((j+1)2^{-n})) =
\phi(1/n) - \phi(n) \lesssim n^{3/4}
$$
by \eqref{E:ptwise}. Thus, there can exists at most $n^{3/4}2^{n/2}$
indices $j$ for which the jump $\phi(j2^{-n})-\phi((j+1)2^{-n}) \geq
2^{-n/2}$. Thus, the measure of the set $H_n$ on which $\phi_n^+(r)
- \phi_n^-(r) > 2^{-n/2}$ satisfies $|H_n| \leq n^{3/4}2^{n/2}2^{-n}
\leq n^{3/4}2^{-n/2}$. This establishes that for a.e. $r$, $\phi(r)
\defeq \lim_{n\to +\infty}\phi_n(r)$ exists and moreover for a.e.
$r$,
$$
\lim_{n\to +\infty} \phi_n^-(r) = \phi(r)= \lim_{n\to +\infty}
\phi_n^+(r).
$$
We know that $\| x\phi_n^-(x) \|_{L^2{(\frac1n\leq r \leq n)}} \leq
\|x\phi_n\|_{L^2}$.  By monotone convergence, we conclude that
$$
\|x\phi(x)\|_{L^2} = \lim_{n\to +\infty} \| x\phi_n^-(x)
\|_{L^2{(\frac1n\leq r \leq n)}} \leq \lim_{n\to +\infty}
\|x\phi_n\|_{L^2}= m^{7/3}.
$$
Similarly, we have that
$$
\|\phi\|_{L^4} \leq 1 \,.
$$
By \eqref{E:ptwise} and dominated convergence, we conclude that
$$
\|\phi\|_{L^2} = \lim_{n\to +\infty} \|\phi_n\|_{L^2} = 1 \,.
$$
Hence, $L(\phi) \leq m$ (and thus, $L(\phi)=m$), i.e., $\phi\in X$
is a minimizer.
\end{proof}

\begin{lemma}
 \label{L:Euler-Lagrange}
If $\phi\in X$, $\phi$ is radial, nonnegative, (nonstrictly)
decreasing, and solves $L'(\phi)=0$, then there exist $\alpha > 0$,
$\beta> 0$, $0 < \gamma \leq 1$ such that $\phi(x)=\beta
\phi_\gamma(\alpha x)$, where
\begin{equation}
 \label{E:phi-form}
\phi_\gamma (x) =
\begin{cases}
(1-|x|^2)^{1/2} & \text{for }|x| \leq \gamma \\
0 &  \text{for }|x|> \gamma.
\end{cases}
\end{equation}
\end{lemma}
\begin{proof}
Writing
$$
\log L(\phi) = \frac3{14}\log \|x\phi\|_{L^2}^2 + \frac17 \log
\|\phi\|_{L^4}^4 - \frac12\|\phi\|_{L^2}^2,
$$
it follows that for each $\psi$,
$$
0 = L'(\phi)(\psi) = L(\phi) \int \left(
\frac{3|x|^2\phi}{7\|x\phi\|_{L^2}^2} + \frac{4}{7}
\frac{\phi^3}{\|\phi\|_{L^4}^4} - \frac{\phi}{\|\phi\|_{L^2}^2}
\right) \psi \, dx \,.
$$
From this, we see that it must have the form
\begin{equation}
\label{E:phi-form-prelim}
\phi(x) = \beta(1-\alpha^2 |x|^2)^{1/2}\,.
\end{equation}
on the set where $\phi(x)\neq 0$.  Since $\phi(x)$ is decreasing, we
see that \eqref{E:phi-form-prelim} holds on $0\leq |x|\leq
\gamma/\alpha$ for some $0<\gamma\leq 1$, and $\phi(x)=0$ for
$|x|>\gamma/\alpha$.
\end{proof}

Let $\phi \in X$ be a minimizer for $L(\phi)$ as in Lemma
\ref{L:minimizer}.  We know that $\phi$ must solve the
Euler-Lagrange equation $L'(\phi)=0$, and hence, Lemma
\ref{L:Euler-Lagrange} is applicable, which establishes that
$\phi(x) = \beta\phi_\gamma(\alpha x)$ for some $\gamma$,  $0<
\gamma \leq 1$, where $\phi_\gamma$ is given by \eqref{E:phi-form}.
We now plug $\phi(x)= \beta \phi_\gamma(\alpha x)$ into $L(\phi)$ to
determine $\gamma$.  However, since $L(\phi)$ is invariant under the
rescaling $\phi(x) \mapsto \beta\phi(\alpha x)$, it suffices to take
$\alpha=1$, $\beta=1$ in this computation.  We compute
$$
\|\phi \|_{L^4}^4 = 4\pi  \left( \frac13 - \frac{2\gamma^2}{5} +
\frac{\gamma^4}{7} \right) \,, \quad \|x\phi \|_{L^2}^2 = 4\pi
\left( \frac15 - \frac{\gamma^2}{7} \right) \,,
$$
$$
\|\phi\|_{L^2}^2 =4\pi \left( \frac13 - \frac{\gamma^2}{5}\right)\,,
$$
and we thus obtain
$$
L(\phi)^{14}= \frac{ \left( \frac13 - \frac{2\gamma^2}{5} +
\frac{\gamma^4}{7} \right)^2 \left( \frac15 - \frac{\gamma^2}{7}
\right)^3 }{ (4\pi)^2 \left( \frac13 - \frac{\gamma^2}{5}
\right)^{7}} \,.
$$
A tedious computation shows that $\gamma =1$
produces the minimum value, which is
$$
L(\phi)^{14} = \frac{ 3^5 \cdot 5^2}{2^2\cdot 7^5 \cdot \pi^2} \,.
$$
This completes the proof of Prop. \ref{P:interp}.

\section{New blow-up criteria}
\label{S:Lushnikov}

In this section, we prove Theorem \ref{T:Lushnikov},
\ref{T:Lushnikov-adapted}, and \ref{T:Lushnikov-radial}.

\subsection{The blow up criteria of Lushnikov}
\label{S:L}

Here, we prove Theorem \ref{T:Lushnikov}.   It is an
adaptation and further investigation of the dynamic criterion for
collapse proposed by P. Lushnikov in \cite{Lu95}.

We write $M=M[u]$, $E=E[u]$, etc., for simplicity. By the virial
identity \eqref{E:virial} in the case $n=3$,
\begin{equation}
 \label{E:4}
V_{tt}(t) = 24 E - 4 \|\grad u(t)\|^2_2,
\end{equation}
and the bound \eqref{E:uncertainty}, we obtain
\begin{equation}
 \label{E:5}
V_{tt}(t) \leq 24 E - 9 \frac{M^2}{V(t)} - \frac14
\frac{|V_t(t)|^2}{V(t)}.
\end{equation}

Now, rewritting the equation \eqref{E:5} to remove the last term
with $V_t^2$ by making a substitution $\ds  V=B^{4/5}$, we get
\begin{equation}
 \label{E:6}
B_{tt} \leq 30 E B^{1/5} - \frac{45}4 \frac{M^2}{B^{3/5}}.
\end{equation}
This differential inequality is an equality with some unknown
non-negative quantity:
\begin{equation}
 \label{E:7}
B_{tt} \leq 30 E B^{1/5} - \frac{45}4 \frac{M^2}{B^{3/5}}
-g^2(t).
\end{equation}

The equation \eqref{E:7} is the key in further analysis and in
\cite{Lu95} is called the {\it dynamic criterion for collapse}.
To analyze this equation, Lushnikov proposes to use a mechanical
analogy of a particle moving in a field with a potential barrier.
Let $B=B(t)$ be the position of a particle (with mass 1) in motion
under 2 forces:
$$
B_{tt} = F_1 + F_2,
$$
where
\begin{equation}
 \label{E:potential}
F_1 = -\frac{\partial U}{\partial B} \quad \text{with the potential}
\quad U = - 25 E B^{6/5} + \frac{225}{8} {M^2}{B^{2/5}},
\qquad \qquad
\end{equation}
and
$$
F_2 = -g^2(t), \quad \text{some unknown force which pulls the
particle towards zero}.
$$

If this particle reaches the origin in a finite period of time,
$$
B(t^*)=0 \quad \text{for some} \quad 0<t^*<\infty,
$$
then collapse necessarily occurs at some time $t \leq t^*$.

Several observations are due:
\begin{itemize}
\item
If the particle reaches the origin without the force $-g^2(t)$
(i.e., if $B(t_1)=0$ %
some $0 < t_1 < \infty$), then it also reaches the origin in the
situation when this force is applied ($B(t_2)=0$ %
for some $0 < t_2 \leq t_1$).

\item
If $E<0$, then $B$ always reaches
the origin and collapse always happens. Thus, the more interesting
case to consider is $E>0$.

\item
The potential $U$ in \eqref{E:potential} as a function of $B$ is
a convex function (at least for positive $B$) with the maximum
attained at
$$
B_{max} = \left(\frac38 \frac{M^2}{E} \right)^{5/4} \quad \text{and}
\quad U_{max} = U(B_{max}) = \frac{75 \sqrt 3}{8 \sqrt 2} \,
\frac{M^3}{E^{1/2}}.
$$
\end{itemize}

Define the ``energy'' of the particle $B$:
\begin{equation}
\label{E:mech-energy}
\mathcal{E}(t) = \frac{B_t(t)^2}{2} + U(B(t)),
\end{equation}
which is time dependent due to the term $g(t)^2$ in \eqref{E:7}.
However, recall that for our purposes it is
sufficient for $B$ to reach the origin if $B$ satisfies only
$B_{tt}=F_1$ (see the first observation above), for which the energy
$\mathcal{E}(t)$ is conserved.

The analysis is facilitated if we introduce the rescaled variables
$\tilde B(s)$ and $\tilde{\mathcal{E}}(s)$, where
$$
B(t) = B_{max}\, \tilde B\left( \frac{16 E t}{\sqrt 3 M}\right) \,,
\quad \mathcal{E}(t) = U_{max} \, \tilde{\mathcal{E}}\left( \frac{16
E \, t}{\sqrt 3 M} \right) \,, \quad s=\frac{16E}{\sqrt 3 M} \,t.
$$
We obtain
$$
\tilde B_{ss} \leq \frac{15}{16}\left(\tilde B^{1/5} - \frac1{\tilde
B^{3/5}} \right).
$$
If we set
$$
\tilde U(\tilde B) \defeq  - \frac12 \tilde B^{6/5} + \frac32 \tilde
B^{2/5} \,,
$$
then \eqref{E:mech-energy} converts to
$$
\tilde{\mathcal{E}}(s) = \frac{8}{25} \tilde B_s(s)^2 + \tilde
U(\tilde B(s)) \,.$$

\begin{figure}
\includegraphics[scale=0.7]{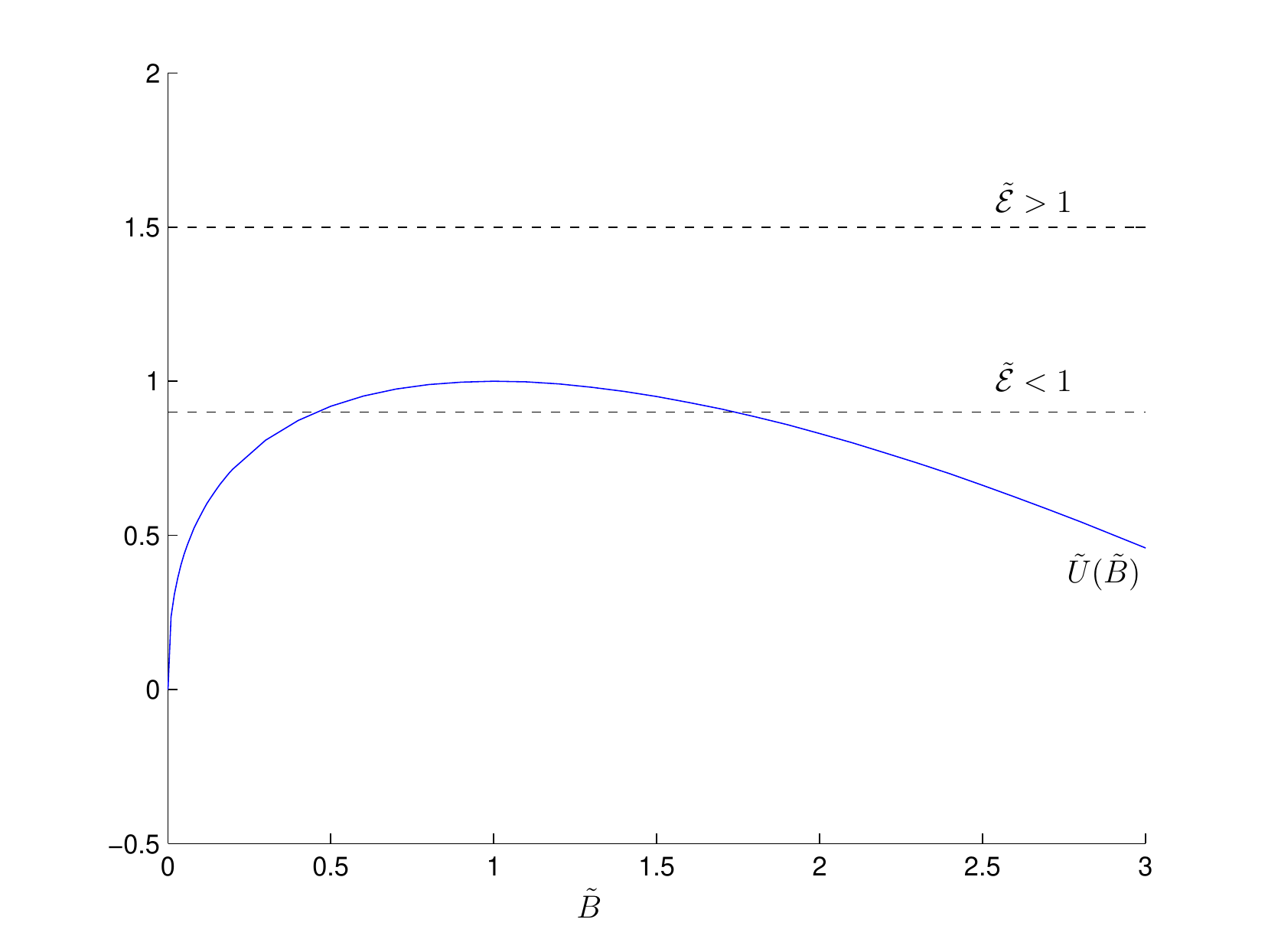}
\caption{A depiction of the rescaled potential $\tilde U$ as a
function of position $\tilde B$. In the case $\tilde{\mathcal{E}} <
1$, a particle starting at $\tilde B(0)<1$ is trapped in that region
and moves to the origin in finite time.  In the case
$\tilde{\mathcal{E}} > 1$, a particle with initial velocity $\tilde
B_s(0)<1$ has sufficient energy to reach the origin in finite time,
regardless of its initial position $\tilde B(0)$.}
 \label{F:mechanical}
\end{figure}

The potential $\tilde U(\tilde B)$ is depicted in Figure
\ref{F:mechanical}.  From this energy diagram, we can identify two
sufficient conditions under which $\tilde B(s)$ necessarily reaches
$0$ in finite time:
\begin{enumerate}
\item
$\tilde{\mathcal{E}}(0) < 1$ and $\tilde B(0)<1$. In this case, the
value of $\tilde B_s(0)$ does not matter.

\item
$\tilde{\mathcal{E}}(0) \geq 1$ and $\tilde B_s(0)<0$. In this case,
the value of $\tilde B(0)$ does not matter.
\end{enumerate}
Now define $\tilde V = \tilde B^{4/5}$.  Then
$$
\tilde{\mathcal{E}} = \frac12 \tilde V^{1/2}( \tilde V_s^2 - \tilde
V + 3) \,.
$$
Introduce the function
\begin{equation}
\label{E:f-def}
f(\tilde V) = \sqrt{ \frac{2}{\tilde V^{\frac12}} + \tilde V - 3} \,.
\end{equation}
We obtain
$$
\tilde{\mathcal{E}} <1  \quad \Leftrightarrow \quad |\tilde V_s| <
f(\tilde V)
$$
$$
\tilde{\mathcal{E}} \geq 1  \quad \Leftrightarrow \quad |\tilde V_s|
\geq f(\tilde V)
$$
Thus, we see that sufficient condition (1) for blow-up above equates
to
$$
\tilde V(0)<1 \text{ and } -f(\tilde V(0))<\tilde V_s(0) < f(\tilde V(0))
$$
and condition (2) equates to
$$
\tilde V_s(0) \geq - f(\tilde V(0)).
$$
This is graphed in Figure \ref{F:Vs-versus-V}.

\begin{figure}
\includegraphics[scale=0.7]{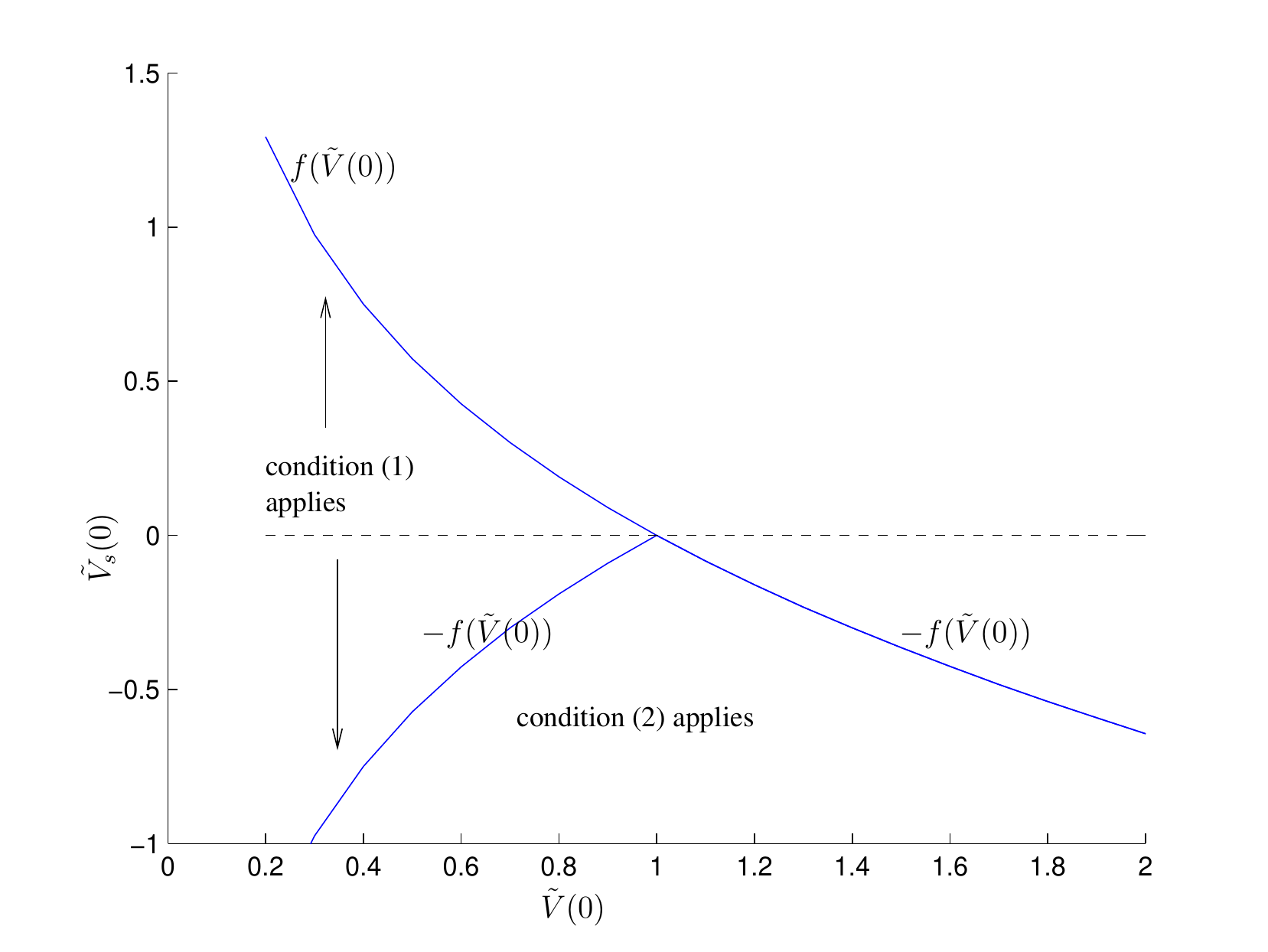}
\caption{The two sufficient conditions (1) and (2) described in the
text \S \ref{S:L} for blow-up in finite time convert to the
conditions on $\tilde V_s(0)$ in terms of $V(0)$ depicted in this
figure. }
 \label{F:Vs-versus-V}
\end{figure}

The two separate conditions can be merged into one:  the solution
blows-up in finite time if
$$
\tilde V_s(0)<
\begin{cases}
+ f(\tilde V(0)) & \text{if } \tilde V(0)\leq 1\\
-f(\tilde V(0)) & \text{if } \tilde V(0)\geq 1
\end{cases}
$$
Tracing back through the rescalings, we see that the relationship
with $V(t)$ and $\tilde V(s)$ is
$$
V(t) = B_{max}^{4/5} \, \tilde V \left( \frac{16 E \, t}{\sqrt 3 M}
\right),
$$
which completes the proof of Theorem \ref{T:Lushnikov}.

\subsection{An adaptation}
\label{S:Lushnikov-adapted}

In this subsection, we prove Theorem \ref{T:Lushnikov-adapted}. The
analysis here is similar to that used above in the proof of Theorem
\ref{T:Lushnikov}, except that we use \eqref{E:interpolation} in
place of \eqref{E:uncertainty}.

By energy conservation, we can rewrite the virial identity as
\begin{equation}
\label{E:supp1}
\begin{aligned}
V_{tt} &= 24 E - 4\|\nabla u\|_{L^2}^2 \\
&= 16 E - 2\|u\|_{L^4}^4.
\end{aligned}
\end{equation}
By \eqref{E:interpolation},
\begin{equation}
 \label{E:anotherVtt}
V_{tt} \leq 16 E - \frac{2M^\frac72}{C^7 V^\frac32}.
\end{equation}
Let
$$
U(V) = -16EV - \frac{4M^\frac72}{C^7 V^\frac12}.
$$
Then, as in the previous subsection, we can interpret $V(t)$ as
giving the position of a particle subject to a conservative force
$-\partial_V U(V)$ plus another unknown nonconservative force
pulling $V(t)$ toward $0$.  The corresponding mechanical energy is
\begin{equation}
\label{E:me} \mathcal{E}(t) = \frac12 V_t^2 + U(V).
\end{equation}
Restricting to the case $E>0$, we compute that $U(V)$ achieves its
maximum $U_{max}$ at $V_{max}$, with
$$
V_{max} = \frac{M^\frac73}{4 C^\frac{14}3 E^\frac23}, \qquad U_{max}
= \frac{-12 M^\frac73 E^\frac13}{C^\frac{14}{3}} .
$$
To facilitate the rest of the analysis, we introduce a rescaling.
Define $\tilde V(s)$ and $\tilde{\mathcal{E}}(s)$ by the relations
$$
V(t) = V_{max}\tilde V(\alpha t) \,, \quad \mathcal{E}(t) =
|U_{max}| \, \tilde{\mathcal{E}}(\alpha t) \,, \quad \alpha =
\frac{8\sqrt 3 C^\frac{7}{3}E^\frac56}{M^\frac76}, \quad s=\alpha t.
$$
Then
$$
\tilde V_{ss}(s) \leq \frac13 \left( 1- {\tilde V^{-3/2}(s)} \right)
$$
and \eqref{E:me} equates to
$$
\tilde{\mathcal{E}} = \frac12 \tilde V_s^2 -\frac13 \tilde V -
\frac23 \tilde V^{-\frac12} \,.
$$
From this, we identify two sufficient conditions for blow-up in finite time.
\begin{enumerate}
\item
$\tilde{\mathcal{E}}(0)<-1$ and $\tilde V(0)<1$.
\item
$\tilde{\mathcal{E}}(0)\geq -1$ and $\tilde V_s(0)<0$.
\end{enumerate}
Let $f(\tilde V)$ be defined by \eqref{E:f-def}.
Then
$$
\tilde{\mathcal{E}}< -1 \quad \Leftrightarrow \quad |\tilde V_s| <
\sqrt{\frac23} f(\tilde V),
$$
$$
\tilde{\mathcal{E}} \geq -1 \quad \Leftrightarrow \quad |\tilde V_s|
\geq \sqrt{\frac23}f(\tilde V).
$$
Thus, condition (1) above holds if and only if
$$
\tilde V(0)<1 \quad  \text{and} \quad -\sqrt{\frac23}\, f(\tilde
V(0))<\tilde V_s(0)<\sqrt{\frac23} \, f(\tilde V(0))
$$
and condition (2) holds if and only if
$$
\tilde V_s(0) \leq -\sqrt{\frac23}\, f(\tilde V(0)).
$$
Clearly, we can merge the two conditions into one: we have blow-up
in finite time provided
$$
\tilde V_s(0) < \sqrt{\frac23}
\begin{cases}
f( \tilde V(0)) & \text{if }\tilde V(0)\leq 1, \\
-f(\tilde V(0)) & \text{if }\tilde V(0)\geq 1.
\end{cases}
$$
Substituting back $V(t)$, we obtain with $\omega = 4 C^{14/3}
E^{2/3} M^{-7/3}V(0)$
$$
\frac{V_t(0)}{M}\, \frac{C^{7/3}}{2 \sqrt 3 (ME)^{1/6}} <
\sqrt{\frac23}
\begin{cases}
f(\omega) & \text{if } \omega \leq 1 \\
-f(\omega) & \text{if } \omega \geq 1,
\end{cases}
$$
which gives \eqref{E:LA} and finishes the proof of Theorem
\ref{T:Lushnikov-adapted}.

\subsection{Infinite variance radial case}

Let $\psi(x)$ be a smooth, radial, (nonstrictly) increasing function
such that
$$
\psi(x)=
\begin{cases}
|x|^2 & \text{if }|x|\leq 1 \\
2 & \text{if }|x|\geq 2
\end{cases}
$$
Let $V_R$ denote the \emph{localized variance}, defined in
\eqref{E:lv-def}. Note that $V_R/M\leq 2R^2$.  We need to replace
inequality \eqref{E:interpolation} with a localized version.

\begin{lemma}
\label{L:interp-radial}
Suppose that $V_R/M \leq \frac12 R^2$.  Then
$$
\|u\|_{L^2} \leq \left( \frac{2^\frac{11}{2} \pi}{3}
\right)^\frac{1}{7} \|u\|_{L_x^4}^{\frac47}V_R^\frac3{14} \,.$$
\end{lemma}
\begin{proof}
Let $r^2 = 2V_R/M$ so that, by assumption, we have $r^2\leq R^2$.  Then
\begin{align*}
M = \|u\|_{L^2}^2
& = \|u\|_{L^2({|x|\leq r})}^2 + \|u\|_{L^2({|x|\geq r})}^2 \\
& \leq \left(\frac43 \,\pi r^3\right)^{1/2} \|u\|_{L^4({|x|\leq
r})}^2 + \frac1{r^2}V_R \\
& \leq \left(\frac43 \, \pi r^3\right)^{1/2} \|u\|_{L^4({|x|\leq
r})}^2 + \frac{M}2 ,
\end{align*}
where we note that the inequality $\|u\|_{L^2({|x|\geq r})}^2 \leq
\frac1{r^2}V_R$ requires $r^2\leq R^2$.  The result follows.
\end{proof}
We have not made any effort to identify the sharp constant here.

Now we turn to the proof of Theorem \ref{T:Lushnikov-radial}. By
direct calculation, we have the local virial identity:
$$
V_R''(t) = 4 \int \partial_j \partial_k \psi\left( \frac{x}{R}
\right) \, \partial_ju \, \partial_k \bar u  - \int \Delta
\psi\left( \frac{x}{R} \right) |u|^4 - \frac{1}{R^2} \int \Delta^2
\psi\left( \frac{x}{R} \right) |u|^2 \,.
$$
Note that
$$
V_R''(t) = 16E  - 2\|u(t)\|_{L_x^4}^4 + A_R(u(t)) \,,$$
where
$$
A_R(u(t)) \lesssim \frac{1}{R^2} \|u\|_{L^2({|x|\geq R})}^2  +
\|u\|_{L^4({|x|\geq R})}^4\,.
$$
Recall that we are assuming $ME> 1$.  Then, if we take
\begin{equation}
 \label{E:R-def}
R^2 \gtrsim  \delta^{-1}M^2 \,,
\end{equation}
we have
$$
\frac{1}{R^2} \|u\|_{L^2({|x|\geq R})}^2 \leq  \frac{\delta}{M} \leq
\delta E \,.
$$
Also, by the radial Gagliardo-Nirenberg inequality,
\begin{align*}
\|u\|_{L^4({|x|\geq R})}^4
&\leq \frac1{2\pi R^2} \|u\|_{L_x^2}^3 \|\nabla u\|_{L_x^2} \\
&\leq \frac{\delta}{2\pi} \frac{\|\nabla u\|_{L^2}}{\|u\|_{L^2}}\\
&\leq \frac{\delta}{2\pi} \|\nabla u\|_{L^2} E^{1/2}, \qquad \text{since }ME> 1\\
&\leq \frac14\,\delta E + \frac14\,\delta \|\nabla u\|_{L^2}^2\\
&\leq \delta E + \delta \|u\|_{L^4}^4.
\end{align*}
Hence,
$$
V_R''(t) \leq (16+2\delta)E-(2-\delta)\|u\|_{L^4}^4 \,.
$$
Now we follow the analysis in \S\ref{S:Lushnikov-adapted} but we
must ensure that for all $t$,
\begin{equation}
 \label{E:bootstrap}
\frac{V_R(t)}{M} \leq \frac12 R^2 .
\end{equation}
In fact, \eqref{E:bootstrap} will act as a bootstrap assumption that
will be reinforced by the mechanical analysis.  Suppose that
\eqref{E:bootstrap} holds, and thus, Lemma \ref{L:interp-radial} is
applicable.  Then
$$
V_R'' \leq (16+2\delta)E - (2-\delta)M^{7/2}(C_\infty)^{-7}
V_R^{-3/2}, \quad \text{where} \quad C_\infty = \left(\frac{2^{11/2}
\pi}{3}\right)^{1/7}.
$$
Let
$$
U(V_R) = -(16+2\delta)EV_R - 2(2-2\delta)(C_\infty)^{-7} M^{7/2}
V_R^{-1/2} \,,
$$
and define the mechanical energy as
$$
\mathcal{E} = \frac12 (V_R')^2 + U(V_R) \,.
$$
The maximum of $U(V_R)$ occurs at $V_{max}$ and is equal to
$U_{max}$, where
$$
V_{max} = \frac{(1-\delta)^\frac23 M^\frac73}{(8+\delta)^\frac23
(C_\infty)^\frac{14}{3}E^\frac23} \,, \qquad U_{max}= - \frac{ 6
(8+\delta)^\frac13 (1-\delta)^\frac23 M^\frac73
E^\frac13}{(C_\infty)^\frac{14}{3}} \,.
$$
We introduce a rescaling: Define $\tilde V(s)$ and
$\tilde{\mathcal{E}}(s)$ by the relations
$$
V(t) = V_{max}\tilde V(\alpha t) \,, \qquad \mathcal{E}(t) =
|U_{max}| \,\tilde{\mathcal{E}}(\alpha t) \,, \qquad \alpha =
\frac{6^\frac12 (8+\delta)^\frac56 (C_\infty)^\frac{7}{3}
E^\frac56}{(1-\delta)^\frac13 M^\frac76}.
$$
Then
$$
\tilde{\mathcal{E}} = \frac12 \tilde V_s^2 +\tilde U(\tilde V) \,,
\qquad \tilde U(\tilde V) \defeq  -\frac13 \tilde V - \frac23 \tilde
V^{-\frac12} \,.
$$
The maximum of $\tilde U(\tilde V)$ now occurs at $\tilde V =1$.
The bootstrap assumption \eqref{E:bootstrap} equates to
\begin{equation}
 \label{E:bootstrap-rescaled}
\tilde V_R(s) \leq \frac{R^2 (8+\delta)^\frac23
(C_\infty)^\frac{14}3 E^\frac23}{2 (1-\delta)^\frac23 M^\frac43} =:
\tilde V_{b} \,,
\end{equation}
where we have defined the right-hand side as $\tilde V_b$, the
``bootstrap threshold''. But by \eqref{E:R-def} and the assumption
$ME>1$, we have
$$
\tilde V_b \geq \frac{
(8+\delta)^\frac23}{(1-\delta)^{\frac23}\delta} \,.
$$
We thus have $\tilde V_b \gg 1$, provided $\delta$ is taken
sufficiently small.\footnote{The smallness on $\delta$ here does not
depend on $M$, $E$, etc.  It depends only on $\psi(x)$, the weight
appearing in the local virial identity.}

\begin{figure}
\includegraphics[scale=0.7]{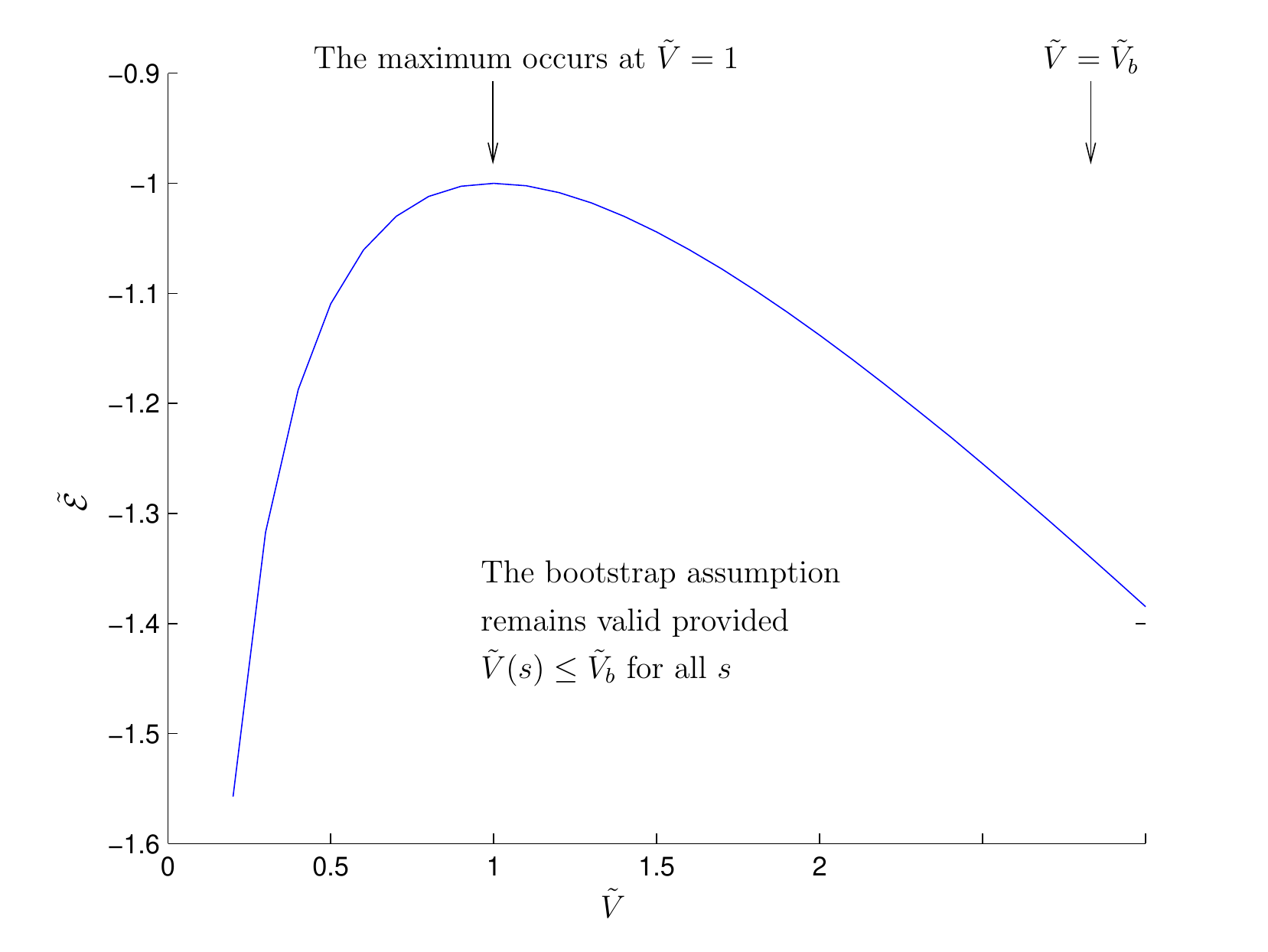}
\caption{A depiction of $\tilde U(\tilde V)$, with $\tilde V_b$
indicated. Note that $\tilde V(0) \ll \tilde V_b$ and always $\tilde
V_b \gg 1$.} \label{F:radial-mech}
\end{figure}

The (rescaled) potential $\tilde U(\tilde V)$, and $\tilde V_b$ are
depicted in Figure \ref{F:radial-mech}.  From this, we identify two
sufficient conditions for blow-up in finite time, noting that in
each case, if \eqref{E:bootstrap-rescaled} holds initially, then it
will hold for all times:
\begin{enumerate}
\item $\tilde{\mathcal{E}}(0)<-1$ and $\tilde V(0)<1$.
\item $\tilde{\mathcal{E}}(0)\geq -1$ and $\tilde V_s(0)<0$.
\end{enumerate}
The remainder of the analysis is the same as in the previous
section, and thus, Theorem \ref{T:Lushnikov-radial} is established.

Finally, we present the proof of Corollary \ref{C:conceptual}.

\begin{proof}[Proof of Cor. \ref{C:conceptual} assuming Theorem \ref{T:Lushnikov-radial}]
Take $R= \mu M\delta^{-1/2}$, where $\mu$ is the constant in the
second equation in \eqref{E:lv-initial}.  Since $u_0$ is real,
\eqref{E:Linf} converts to the statement
\begin{equation}
\label{E:concep}
(ME)^\frac23 \frac{V_R(0)}{M^3} = \frac{E^\frac23V_R(0)}{M^\frac73} \ll 1
\end{equation}
(see the graph of $g$ in Fig. \ref{F:g-graph}.)  Decompose
\begin{align*}
\frac{V_R(0)}{MR^2} 
&= 
\begin{aligned}[t]
&\frac1M \int_{|x|\leq \delta^{1/2}
M(ME)^{-1/3}} \psi(x/R) |u_0|^2 \, dx \\
&+ \frac1M\int_{|x|\geq
\delta^{1/2} M(ME)^{-1/3}} \psi(x/R) |u_0|^2 \, dx \\
\end{aligned}\\
&=\text{I}+\text{II} \,.
\end{align*}
For I, note that, when $|x|\leq \delta^{1/2} M(ME)^{-1/3} \leq R$,
$$
\psi(x/R) = |x|^2/R^2 = |x|^2\delta/\mu^2 M^2 \leq
\delta^2(ME)^{-2/3}\mu^{-2},
$$
and thus, $|\text{I}| \leq \delta^2(ME)^{-2/3}\mu^{-2}$.  For II,
just use $|\psi(x/R)|\lesssim 1$ and the assumption
\eqref{E:mass-dist} to obtain $|\text{II}| \leq
\delta^2(ME)^{-2/3}$. Hence,
$$
\frac{V_R(0)}{MR^2} \leq \delta^2(ME)^{-2/3} \,.
$$
Thus, the first condition in \eqref{E:lv-initial} is satisfied, and
moreover,
$$
\frac{V_R(0)}{M^3} = \frac{V_R(0) \mu^2}{MR^2\delta} \leq \mu^2 \delta (ME)^{-2/3} \,.
$$
Therefore, \eqref{E:concep} holds provided $\delta$ is sufficiently
small.
\end{proof}

\section{Preliminaries for the profile analyses}
\label{S:prelim}

In the sections that follow
(\S\ref{S:sechprofile}--\ref{S:osc-Gaussian}), we consider several
different initial data families.  Here we record some facts needed.

\subsection{$\dot H^{1/2}$ norm}
For radially symmetric data $u_0$, the Fourier transform can be expressed as
\begin{align}
\notag
\hat{u}_0(R)
& = 2 \, R^{-1} \int_0^\infty u_0(r) \, \sin(2\pi R r) \, r \, dr\\
\label{E:Fourier-radial-Bessel}
&= 2 \pi \, R^{-1/2} \int_0^{\infty}
u_0(r) \, J_{1/2}(2 \pi R r) \, r^{3/2} \, dr \,,
\end{align}
where $J_{1/2}$ is the Bessel function of index $\frac12$ given by
$$
J_{1/2}(2\pi Rr) = \frac{\sin (2\pi Rr)}{\pi \sqrt{Rr}} \,.
$$
We can then obtain
\begin{equation}
 \label{E:H12norm-radial}
\|u_0\|^2_{\dot{H}^{1/2}(\cR^3)} = 8 \,\pi^2 \, \int_0^{\infty}
 R^3 \, |\hat{u}_0(R)|^2 \, dR.
\end{equation}

\subsection{Properties of $Q$}
Recall that the Pohozhaev identities \eqref{E:Pohozhaev} hold, which
in the case $n=3$ take the form
\begin{equation}
 \label{E:Pohozhaev3d}
\|\nabla Q\|_{L^2}^2 = 3\|Q\|_{L^2}^2 \quad \text{and} \quad
\|Q\|_{L^4}^4 = 4\|Q\|_{L^2}^2.
\end{equation}
It then follows that $E[Q] = \frac12 M[Q]$ or $\frac16 \|\grad
Q\|^2_2$.

We computed numerically:
\begin{align*}
\|Q\|^2_{L^2(\cR^3)} & = 4 \pi \, \int_0^{\infty} Q^2(r) \,  r^2 \,
dr \approx 18.94,\\
\| y Q\|^2_{L^2(\cR^3)} & = 4 \pi \, \int_0^{\infty} Q^2(r) \, r^4
\, dr \approx 20.32,\\
\|Q\|_{\dot H^{1/2}(\cR^3)}^2 & = 8 \,\pi^2 \, \int_0^{\infty} \! \,
|\hat{Q}(R)|^{\,2} R^3\, dR \approx 27.72665 \,.
\end{align*}

\subsection{Reformulation of blow up conditions}
\label{S:reformulation}

For real valued initial data, Theorem \ref{T:Lushnikov} condition
\eqref{E:L} can be simplified to
\begin{equation}
 \label{E:Lreal}
V(0) < \frac38 \frac{M^2}{E}. \hspace{3cm}
\end{equation}
Similarly, the condition \eqref{E:LA} from Theorem
\ref{T:Lushnikov-adapted} can be simplified to
\begin{equation}
 \label{E:LAreal}
\hspace{3cm} V(0) < c \, \frac{M^{7/3}}{E^{2/3}} \quad
\text{with}\quad c=\frac1{4\,C^{14/3}} \quad \text{and} \quad C
\quad \text{from} ~~ \eqref{E:LA}.
\end{equation}
Observe that the second condition \eqref{E:Lreal} is an improvement
over \eqref{E:LAreal} for real valued initial data when
$$
M[u]E[u] > \frac{7^5 \pi^2}{450}%
\, M[Q]E[Q] \approx 2.06 \, M[Q]E[Q].
$$
Thus, when discussing the real valued initial data, we will refer to
\eqref{E:Lreal} and \eqref{E:LAreal} instead of Theorems
\ref{T:Lushnikov} and \ref{T:Lushnikov-adapted}, correspondingly;
the blow up conditions look simple in this case.
\smallskip

For complex valued initial data, Theorem \ref{T:Lushnikov} condition
\eqref{E:L} has to be considered
separately for positive and negative values of $V_t(0)$. %
Define (a scale invariant quantity)
$$ \ds \omega = \frac83 \frac{E \,V(0)}{M^2},
$$
then for positive valued $V_t(0)$ blow up happens
\begin{itemize}
\item
in the intersection of the regions:
\begin{equation}
 \label{E:L+}
0 < \omega \leq 1 \quad \text{and} \quad \frac{V_t(0)}{M} < 2 \sqrt
3 \left(\frac{3^{1/2}M}{2^{1/2} E^{1/2} V(0)^{1/2}} + \frac{8 \,E
V(0)}{3 \,M^2} - 3\right)^{1/2},
\end{equation}
\end{itemize}
and for the negative valued $V_t(0)$ blow up happens
\begin{itemize}
\item
in all of the region $0 < \omega \leq 1$
\item
and in the intersection of the regions:
\begin{equation}
 \label{E:L-}
\omega \geq 1 \quad \text{and} \quad \frac{|V_t(0)|}{M} > 2 \sqrt 3
\left(\frac{3^{1/2}M}{2^{1/2} E^{1/2} V(0)^{1/2}} + \frac{8 \,E
V(0)}{3 \,M^2} - 3 \right)^{1/2}.
\end{equation}
\end{itemize}

In this paper we consider the complex valued initial data with the
quadratic phase %
\begin{equation}
 \label{E:quadraticphase}
u_0(r) = f(r)\, e^{i \gamma \,r^2},
\end{equation}
with $f(r)$ - real valued radial function and $\gamma \in \cR$. For
such initial condition we have
$$
V(0) = 4\pi F, \quad V_t(0) = 32 \pi \gamma F \quad \text{with}
\quad F = \int_0^{\infty} r^4 \, |f(r)|^2 \, dr.
$$
Also
$$
E =  \left( 2\pi \int_0^\infty r^2 |\grad f |^2 \, dr - \pi
\int_0^\infty r^4 |f|^2 \, dr \right) + 8 \pi \gamma^2 F = E^0 +
E^\gamma.
$$
Since
\begin{equation}
 \label{E:phasenulling}
{V_t(0)^2} - 32 V(0) E^\gamma = 0,
\end{equation}
the second condition in \eqref{E:L+} is simplified further to
\begin{equation}
 \label{E:L+simple}
\sqrt{\frac32} \frac{M}{[E\, V(0)]^{1/2}}+ \frac83
\frac{E^0\,V(0)}{M^2}-3>0,
\end{equation}
and analogously, the second condition in \eqref{E:L-} will be as
above inequality with the reversed sign (we write it out for future
reference)
\begin{equation}
 \label{E:L-simple}
\sqrt{\frac32} \frac{M}{[E\, V(0)]^{1/2}}+ \frac83
\frac{E^0\,V(0)}{M^2}-3 < 0.
\end{equation}

Similarly, Theorem \ref{T:Lushnikov-adapted} condition \eqref{E:LA}
has to be studied separately for positive and negative values of
$V_t(0)$. Denoting (also a scale invariant quantity)
$$
\kappa = 4\,C^{14/3} \, \frac{E^{2/3}V(0)}{M^{7/3}},
$$
we obtain that for the positive valued $V_t(0)$ blow up happens
\begin{itemize}
\item
in the intersection of the regions:
\begin{equation}
 \label{E:LA+}
\qquad 0 < \kappa \leq 1 \quad \text{and} \quad \frac{V_t(0)}{M} <
\frac{2\sqrt 2}{C^{7/3}}
\left(\frac{M^{3/2}}{C^{7/3}V(0)^{1/2}}+\frac{4C^{14/3}E V(0)}{M^2}
- 3(ME)^{1/3} \right)^{1/2},
\end{equation}
\end{itemize}
and for the negative valued $V_t(0)$ blow up happens
\begin{itemize}
\item
in all of the region $0 < \kappa \leq 1$
\item
and in the intersection of the regions:
\begin{equation}
 \label{E:LA-}
\qquad \kappa \geq 1 \quad \text{and} \quad \frac{|V_t(0)|}{M}
> \frac{2\sqrt 2}{C^{7/3}}
\left(\frac{M^{3/2}}{C^{7/3}V(0)^{1/2}}+\frac{4C^{14/3}E V(0)}{M^2}
- 3(ME)^{1/3} \right)^{1/2}.
\end{equation}
\end{itemize}
For the initial data with the quadratic phase as in
\eqref{E:quadraticphase}, the second condition in \eqref{E:LA+}
reduces to
\begin{equation}
 \label{E:LA+simple}
\frac{M^{3/2}}{C^{7} V(0)^{1/2}} + 4 \, \frac{E^0 V(0)}{M^2} -
3\left(\frac{ME}{C^{14}}\right)^{1/3} > 0,
\end{equation}
due to \eqref{E:phasenulling}. The second inequality in
\eqref{E:LA-} reduces to the same inequality as above except with
the reversed sign (again, we write it out for the convenience of
future reference):
\begin{equation}
 \label{E:LA-simple}
\frac{M^{3/2}}{C^{7} V(0)^{1/2}} + 4 \, \frac{E^0 V(0)}{M^2} -
3\left(\frac{ME}{C^{14}}\right)^{1/3} < 0.
\end{equation}

In computations below we study the conditions
\eqref{E:L+}-\eqref{E:L-}, \eqref{E:L+simple}-\eqref{E:L-simple} and
\eqref{E:LA+}-\eqref{E:LA-simple} instead of \eqref{E:L} and
\eqref{E:LA}, respectively. In graphical presentation we refer to
the set of conditions \eqref{E:L+}-\eqref{E:L-},
\eqref{E:L+simple}-\eqref{E:L-simple} as {\it ``Condition from Thm.
\ref{T:Lushnikov}"} and \eqref{E:LA+} - \eqref{E:LA-simple} as {\it
``Condition from Thm. \ref{T:Lushnikov-adapted}"}.

\section{$Q$ profile}
\label{S:sechprofile}

In this section we study initial data of the form
\begin{equation}
 \label{E:Qphase}
u_0(r) = \lambda^{3/2}Q(\lambda r) \, e^{i\gamma r^2},
\end{equation}
where $Q$ is the ground state defined by \eqref{E:Q}. Note that
$u_0$ has been scaled so that $M[u]=M[Q]$ for all $\lambda>0$.

We compute
$$
\frac{E[u]}{E[Q]} = 3\lambda^2 - 2\lambda^3 + \frac{3\tilde
\gamma^2}{\lambda^2} \,, \qquad \frac{ \|\nabla
u_0\|_{L^2}^2}{\|\nabla Q\|_{L^2}^2} = \lambda^2 + \frac{\tilde
\gamma^2}{\lambda^2} \,,
$$
where, by a Pohozhaev identity \eqref{E:Pohozhaev3d},
\begin{equation}
 \label{E:tilde-gamma}
\ds \tilde\gamma^2 = 4 \, \gamma^2
\frac{\|yQ\|_{L^2(\cR^3)}^2}{\|\nabla Q\|_{L^2(\cR^3)}^2} = \frac43
\, \gamma^2 \frac{\|yQ\|_{L^2(\cR^3)}^2}{\|Q\|_{L^2(\cR^3)}^2}.
\end{equation}
The blow up when energy is negative occurs when
\begin{equation}
 \label{E0:negative-energy}
\tilde \gamma^2 < \frac23 \lambda^5 - \lambda^4.
\end{equation}
For positive energy blow-up is analytically proven in the case $\ds
\frac{E[u]}{E[Q]} \leq 1$ and $\ds \frac{\|\nabla
u_0\|_{L^2}^2}{\|\nabla Q\|_{L^2}^2}>1$ by Theorem \ref{T:DHR}.
These conditions equate to
$$
\lambda^2(1-\lambda^2) < \tilde \gamma^2 < \frac13 \lambda^2
(1-3\lambda^2+2\lambda^3) = \frac13 (1-\lambda)^2(1+2\lambda) \,,
$$
which is not possible for any $0<\lambda<1$. On the other hand,
scattering is proved in the case  $\frac{E[u]}{E[Q]}<1$ and $\frac{
\|\nabla u_0\|_{L^2}^2}{\|\nabla Q\|_{L^2}^2}<1$, which just reduces
to the condition $0<\lambda<1$ and
$$
\tilde \gamma^2 < \frac13\lambda^2 (1-3\lambda^2+2\lambda^3).
$$
This is depicted in Figure \ref{F:MG-MEcurves}.
\begin{figure}[h]
\begin{center}
\includegraphics[width=14cm]{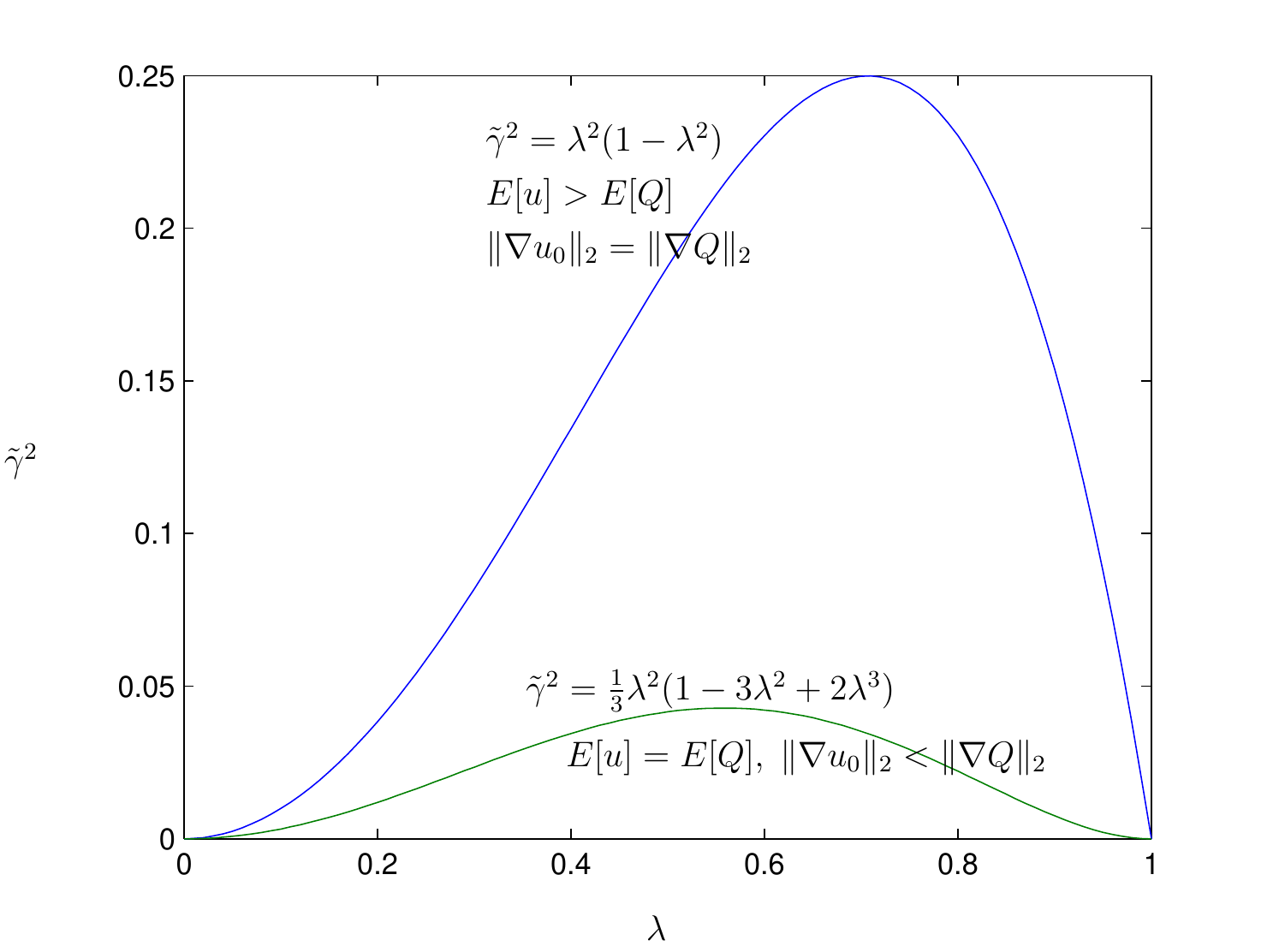}
\end{center}
\caption{The curves for the mass-gradient and mass-energy thresholds
for the $Q$ profile in \eqref{E:Qphase}.}
 \label{F:MG-MEcurves}
\end{figure}

We now investigate the conditions for blow up from Theorems
\ref{T:Lushnikov} and \ref{T:Lushnikov-adapted}. Note that
$$
V(0) = \frac1{\lambda^2} \, \|y Q\|^2_2 \quad \text{and} \quad
V_t(0) = \gamma \, \frac{8}{\lambda^2} \,   \|y Q\|_2^2.
$$
Hence, depending on the sign of $\gamma$, Theorems
\ref{T:Lushnikov}, \ref{T:Lushnikov-adapted} provide different
ranges of values $(\lambda, \gamma)$ for which blow up occurs. We
start with $\gamma < 0$. In what follows we use $\tilde\gamma$
instead of $\gamma$, recall their simple relation
\eqref{E:tilde-gamma}.

$\bullet$ Theorem \ref{T:Lushnikov}: we investigate the region given
by $0 < \omega \leq 1$, see notation in \eqref{E:L+}, and then its
complement with an additional restriction \eqref{E:L-simple}. The
condition $\omega \leq 1$ is
\begin{equation}
 \label{E0:Lgreendash}
\tilde \gamma^2 \leq \frac23 \lambda^5 + \left(\frac14
\frac{\|Q\|^2_2}{\|y Q\|_2^2} -1 \right) \lambda^4,
\end{equation}
note that the right side is nonnegative when $\lambda \geq
\frac32(1-\frac{\|Q\|^2_2}{4\,\|y Q\|_2^2}) \approx 1.15$.
The %
condition \eqref{E:L-simple} is
\begin{multline}
 \label{E0:Lgreen}
\qquad \tilde \gamma^2 > \frac23 \lambda^5 -
\left(\frac{\|Q\|^2_2}{\|y Q\|^2_2} \frac1{\left(4(\frac23
\lambda-1)\frac{\|y Q\|^2_2}{\|Q\|^2_2}+3\right)^2} - 1
\right)\lambda^4,\\
\text{provided} \quad \lambda \geq \frac32
\left(1-\frac34\frac{\|Q\|^2_2}{\|y Q\|^2_2}\right) \approx .45.
\qquad
\end{multline}
Thus, for $\gamma < 0$, or negative $V_t(0)$, union of
\eqref{E0:Lgreendash} and \eqref{E0:Lgreen} will give a region where
solution will blow up in finite time. It turns out that the first
condition \eqref{E0:Lgreendash} is a part of the region covered by
\eqref{E0:Lgreen}, and the last one is depicted in Figure
\ref{F:sechn} under the name {\it ``Blow up by Thm.
\ref{T:Lushnikov}"}.

For $\gamma > 0$, the blow up region by Theorem \ref{T:Lushnikov} is
an intersection of \eqref{E0:Lgreendash} and the inequality
\eqref{E0:Lgreen} with the reversed sign, which is depicted in
Figure \ref{F:sechp} under the name {\it ``Blow up by Thm.
\ref{T:Lushnikov}"}. It turns out that this region has previously
been covered by our Theorem \ref{T:DHR}.

Summarizing, Theorem \ref{T:Lushnikov} provides a nontrivial blow up
condition for the pair $(\lambda, \tilde \gamma^2)$ only if $0.45
\leq \lambda \leq 1.15$ and $\gamma < 0$ (see Figure \ref{F:sechn}).

$\bullet$ Theorem \ref{T:Lushnikov-adapted}: similarly to the above,
we start with $\gamma < 0$ and investigate the region given by $0 <
\kappa \leq 1$, see notation in \eqref{E:LA+}, and then its
complement with an additional restriction \eqref{E:LA-simple}. The
condition $\kappa \leq 1$ is
\begin{equation}
 \label{E0:LAorangedash}
\tilde \gamma^2 < \left(\frac{5 \cdot 3^{3/2}}{8\pi \cdot 7^{5/2}}
\frac{\|Q\|_2^5}{\|yQ\|_2^3} +\frac23\right) \lambda^5 - \lambda^4,
\end{equation}
with the right side being nonnegative when $\lambda \geq
\left(\frac{5 \cdot 3^{3/2}}{8\pi \cdot 7^{5/2}}
\frac{\|Q\|_2^5}{\|yQ\|_2^3} +\frac23\right)^{-1} \approx 1.25$. The
condition \eqref{E:LA-simple} is
\begin{equation}
 \label{E0:LAorange}
\tilde \gamma^2 > \frac23 \lambda^5 - \lambda^4+ \frac23
\frac{C^{14}}{\|Q\|_2^4}\lambda^2 \left(
\left(\frac1{C^7}\frac{\|Q\|^3_2}{\|y Q\|_2} -
4\frac{\|yQ\|_2^2}{\|Q\|^2_2}\right)\lambda+ 2
\frac{\|yQ\|^2_2}{\|Q\|^2_2} \right)^3,
\end{equation}
where $C$ is from \eqref{E:LA}, and is valid for any $\lambda > 0$.
Again, it turns out that the first condition \eqref{E0:LAorangedash}
is a part of the region covered by \eqref{E0:LAorange}, and the last
one is depicted in Figure \ref{F:sechn} under the name {\it ``Blow
up by Thm. \ref{T:Lushnikov-adapted}"}.

For $\gamma > 0$, the blow up region by Theorem
\ref{T:Lushnikov-adapted} is an intersection of
\eqref{E0:LAorangedash} and the inequality \eqref{E0:LAorange} but
with the reversed sign, which is depicted in Figure \ref{F:sechp}
under the name {\it ``Blow up by Thm. \ref{T:Lushnikov-adapted}"}.
This region has also been previously covered by our Theorem
\ref{T:DHR}.

In summary, Theorem \ref{T:Lushnikov-adapted} provides a nontrivial
blow up condition for the pair $(\lambda, \tilde \gamma^2)$ for any
$0 \leq \lambda \leq 1.25$ and $\gamma < 0$ (see Figure
\ref{F:sechn}). In comparison with Theorem \ref{T:Lushnikov}, it
provides a wider range for $0 < \lambda < 0.762$.

Note that Theorems \ref{T:Lushnikov} and \ref{T:Lushnikov-adapted}
provide new information on the blow up behavior for the $Q$ profile
of the form \eqref{E:Qphase} with negative phase. In particular,
Theorem \ref{T:Lushnikov} gives a nontrivial range of $\gamma$ when
$0.45 < \lambda < 1.15$. Further extension is given by Theorem
\ref{T:Lushnikov-adapted} for any $\lambda < 0.762$, see Figure
\ref{F:sechn}. Both Theorems provide a blow up range ``under" the
mass-gradient condition, thus, showing that the last condition is
irrelevant for determining long time behavior in the region when
$M[u]E[u] < M[Q]E[Q]$.

Lastly, we compute $\|u_0\|_{\dot H^{1/2}}$ norm using
\eqref{E:H12norm-radial} numerically and then compare all conditions
about the global behavior for this initial data together with
numerical data in Figures \ref{F:sechn} and \ref{F:sechp} with
negative and positive signs in the initial phase, correspondingly.
For the positive phase Theorems \ref{T:Lushnikov} and
\ref{T:Lushnikov-adapted} do not provide any new information,
however, we have numerical range on blow up threshold and include
the plot for illustration and completeness.

\begin{figure}[h]
\includegraphics[scale=.72]{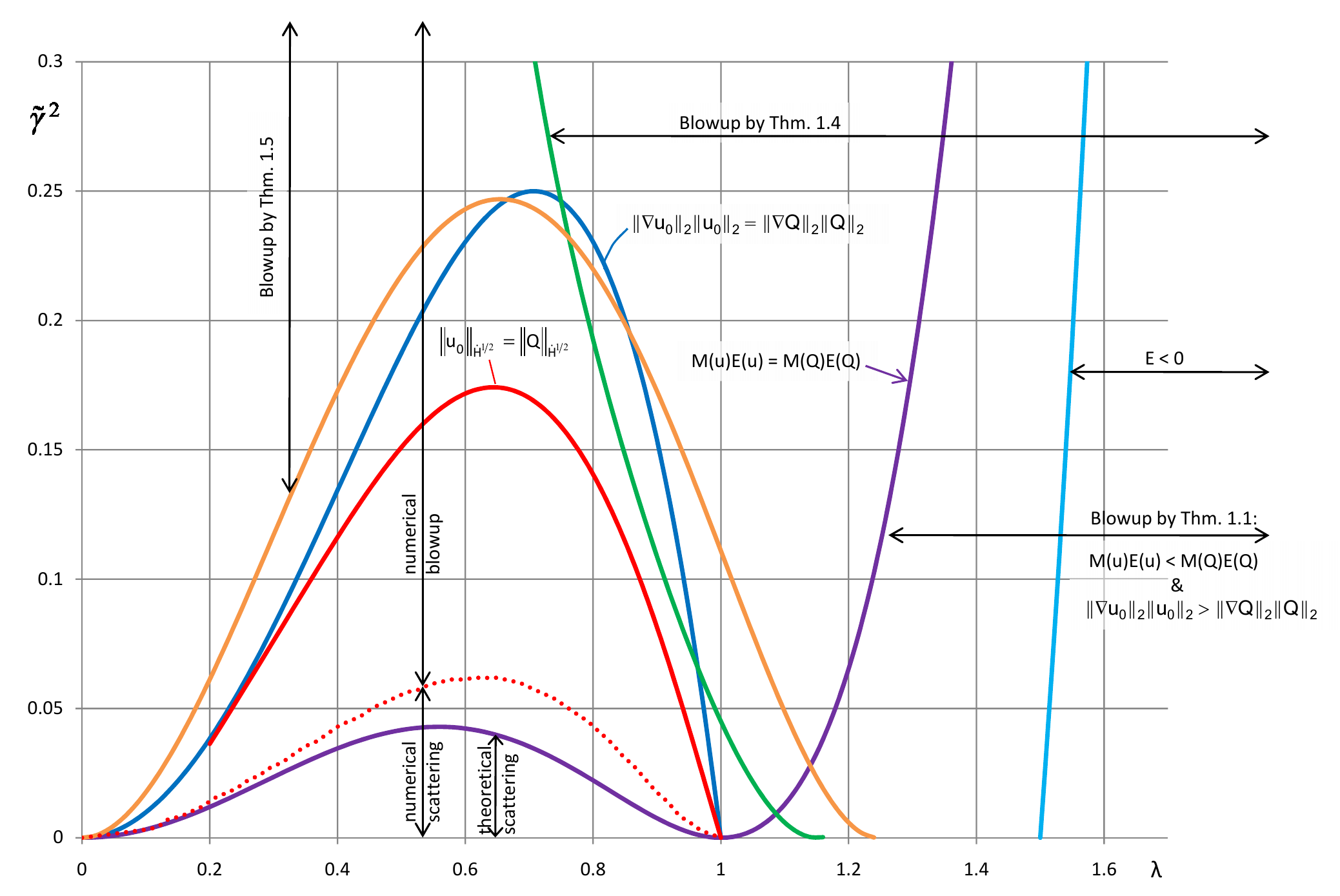}
\caption{Global behavior of the solutions to the $Q$ profile initial
data with the negative quadratic phase: $u_0(r) = \lambda^{3/2}
Q(\lambda r) \, e^{- i |\gamma| r^2}$. Here, $\tilde \gamma$ is the
renormalized $\gamma$, see \eqref{E:tilde-gamma}, namely, $\tilde
\gamma^2 \approx 1.43 \gamma^2$. The region of ``theoretical
scattering" is provided by Thm. \ref{T:DHR}, see also Figure
\ref{F:MG-MEcurves}. ``Blowup by Thm. \ref{T:Lushnikov}" is given by
\eqref{E0:Lgreen} and ``Blowup by Thm. \ref{T:Lushnikov-adapted}" is
given by \eqref{E0:LAorange}. The intersection of these two
conditions occurs at $\lambda \approx 0.762$.}
 \label{F:sechn}
\end{figure}

\begin{figure}[h]
\includegraphics[scale=.72]{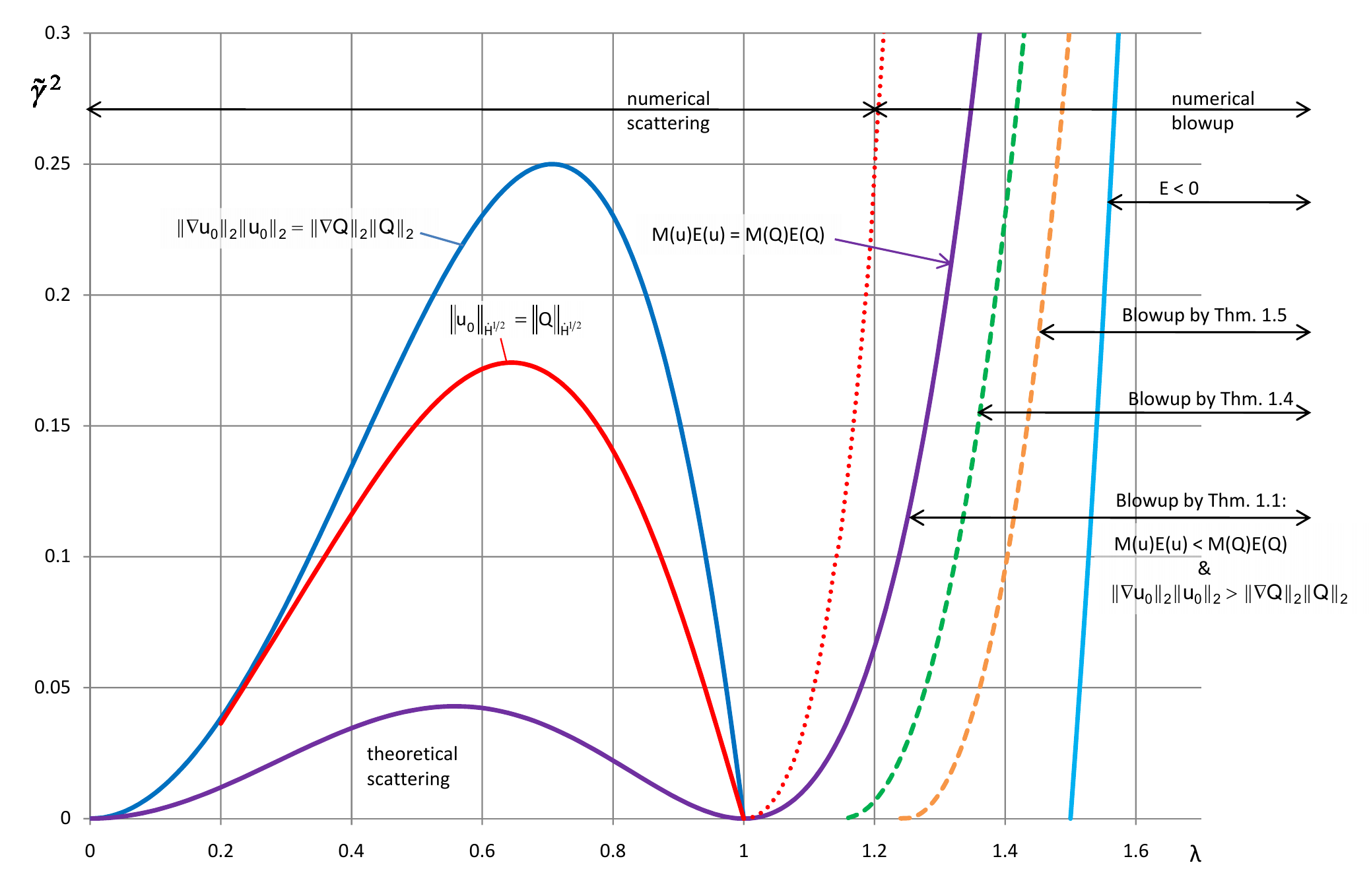}
\caption{Global behavior of the solutions to the $Q$ profile initial
data with the positive quadratic phase: $u_0(r) = \lambda^{3/2}
Q(\lambda r) \, e^{+ i |\gamma| r^2}$. As before, $\tilde \gamma$ is
the renormalized $\gamma$, see \eqref{E:tilde-gamma}, namely,
$\tilde \gamma^2 \approx 1.43 \gamma^2$. The region of ``theoretical
scattering" is the same as in Fig. \ref{F:sechn} and is provided by
Thm. \ref{T:DHR}. ``Blowup by Thm. \ref{T:Lushnikov}" is given by
the complement of \eqref{E0:Lgreen} intersected with
\eqref{E0:Lgreendash} and ``Blowup by Thm.
\ref{T:Lushnikov-adapted}" is given by the complement of
\eqref{E0:LAorange} intersected with \eqref{E0:LAorangedash}. }
 \label{F:sechp}
\end{figure}

\begin{example}
Consider $u_0 = Q(r)\, e^{i \gamma r^2}$ (i.e., take $\lambda=1$ in
\eqref{E:Qphase}). Note that
$$
M[u_0] E[u_0] = \left(1+4\gamma^2 \frac{\|yQ\|^2_2}{\|Q\|^2_2}
\right) M[Q] E[Q] > M[Q] E[Q],
$$
and thus, Theorems \ref{T:DHR} and \ref{T:DR} can not be applied.
However, the solution with such initial condition will blow up in
finite time if
$$
\gamma < -\left(\frac34 \frac{\|Q\|^4_2}{\|yQ\|^4_2 \left(3-\frac43
\frac{\|Q\|^2_2}{\|yQ\|^2_2} \right)^2} -
\frac{\|Q\|^2_2}{4\|yQ\|^2_2} \right)^{1/2} \approx - 0.177,
$$
by Theorem \ref{T:Lushnikov}, or when
$$
\gamma < -\left( \frac{\|Q\|^4_2}{54 C^7 \|yQ\|^2_2}
\left(\frac{\|Q\|_2}{\|yQ\|_2} + \frac{2 C^7 \,
\|yQ\|^2_2}{\|Q\|_2^4} \right)^3 - \frac{\|Q\|^2_2}{4
\, \|yQ\|^2_2} \right)^{1/2} \approx %
- 0.279,
$$
by Theorem \ref{T:Lushnikov-adapted}. In this example, Theorem
\ref{T:Lushnikov} is more powerful than Theorem
\ref{T:Lushnikov-adapted}, however, this is not always the case, as
can be seen from Figure \ref{F:sechn} (for example, for $\lambda <
.762$ Theorem \ref{T:Lushnikov-adapted} gives a larger range for
blow up).
\end{example}

\subsection{Conclusions}

\begin{enumerate}
\item
The condition ``$\|u_0\|_{\dot H^{1/2}}<\|Q\|_{\dot H^{1/2}}$ {\it
implies scattering}'' is not valid; the numerical blow-up curve is
\emph{below} the $\|u_0\|^2_{\dot H^{1/2}}=\|Q\|^2_{\dot H^{1/2}}$
curve in Figure \ref{F:sechn}.  This supports Conjecture 2.

\item
The condition ``$\|u_0\|_{L^2}\|\nabla
u_0\|_{L^2}<\|Q\|_{L^2}\|\nabla Q\|_{L^2}$ {\it implies
scattering}'' is not valid (unless $M[u]E[u]<M[Q]E[Q]$ as in Theorem
\ref{T:DHR}); not only the numerical blow-up curve is \emph{below}
the $\|u_0\|_{L^2}\|\nabla u_0\|_{L^2}=\|Q\|_{L^2}\|\nabla
Q\|_{L^2}$ curve in Figure \ref{F:sechn}, but also both Theorems
\ref{T:Lushnikov} and \ref{T:Lushnikov-adapted} provide range of
$(\lambda, \gamma)$ for which blow up from the initial data
\ref{E:Qphase} occurs and this range is {\it below} the
$\|u_0\|_{L^2}\|\nabla u_0\|_{L^2}=\|Q\|_{L^2}\|\nabla Q\|_{L^2}$
curve in Figure \ref{F:sechn}.

\item
Previously, no theoretical blow-up result for the profile
\eqref{E:Qphase} with $0<\lambda \leq 1$ could be obtained from
Theorem \ref{T:DHR}. The new blow-up criteria in Theorems
\ref{T:Lushnikov} and \ref{T:Lushnikov-adapted} give a nonempty set
of $(\lambda, \gamma)$ with $\gamma<0$ for which blow up occurs, see
conditions \eqref{E0:Lgreendash}-\eqref{E0:Lgreen} and
\eqref{E0:LAorangedash}-\eqref{E0:LAorange} as well as the
illustration in Figure \ref{F:sechn}.
\end{enumerate}

\section{Gaussian profile}
\label{S:Gaussian}

In this section, we study initial data $u_0$ of the form
\begin{equation}
 \label{E:Gaussian}
u_0(x) = p \, e^{-\alpha r^2/2} \, e^{i\gamma r^2} \,,  \quad r =
|x|\,, \quad x \in \cR^3.
\end{equation}
By scaling, it suffices to consider the cases $\gamma=0$ (real data)
and $\gamma=\pm \frac12$. The main parameters are %
\begin{align*}
& M[u] = %
\frac{\pi^{3/2} p^2}{\alpha^{3/2}}, %
\qquad \Vert \nabla u_0 \Vert^2_{L^2} =  \frac{3\pi^{3/2}\,p^2}{2\,
\alpha^{1/2}} \,\left(1+4\frac{\gamma^2}{\alpha^2} \right),\\
& E[u] = \frac{\pi^{3/2} p^2}{4 \, \alpha^{1/2}}
\left(3\left(1+4\frac{\gamma^2}{\alpha^2}\right) -
\frac{p^2}{2\sqrt2 \alpha} \right), \\
& V(0) %
=  \frac{3 \,\pi^{3/2} p^2}{2 \,\alpha^{5/2}}, %
\qquad  V_t(0) =  \frac{12 \gamma \pi^{3/2} p^2}{\alpha^{5/2}}, \\
& \Vert u_0 \Vert^2_{\dot{H}^{1/2}}
= \frac{2 \, \pi \, p^2}{\alpha} \left(1+ 4
\frac{\gamma^2}{\alpha^2}\right)^{1/2}.
\end{align*}
To compute the last expression $\|u_0\|_{\dot H^{1/2}}$, consider
the Fourier transform of $u_0^\gamma$ (here, $R^2 = |\xi|^2$, $\xi
\in \cR^3$)
\begin{align*}
\widehat{u}_{0} (R)
& = p \, \left(\frac{2 \pi}{\alpha - 2 i \gamma} \right)^{3/2} \,
e^{- \frac{2 \pi^2 \, \alpha}{\alpha^2 +4 \gamma^2} R^2} \, e^{- i
\frac{4 \pi^2 \, \gamma}{\alpha^2 +4 \gamma^2} R^2} ,
\end{align*}
where we used $ \int_{-\infty}^{\infty} e^{i(az^2+2 b z)} \, dz =
\sqrt{\frac{\pi \, i}{a}} \, e^{-i \, b^2/a}$,  $a$, $b \in \cC$. By
\eqref{E:H12norm-radial} we have
\begin{equation}
\Vert u_{0} \Vert^2_{\dot{H}^{1/2}(\cR^3)}
= \frac{64 \pi^5 \, p^2}{(\alpha^2 + 4 \gamma^2)^{3/2}}
\int_0^\infty e^{-\frac{4 \pi^2 \, \alpha}{\alpha^2+4 \gamma^2}R^2}
\, R^3 \, dR  = \frac{2 \, \pi \, p^2}{\alpha^2} \, (\alpha^2+4
\gamma^2)^{1/2}.
\end{equation}

\subsection{Real Gaussian}
 \label{S:GaussianReal}
Take $\gamma=0$ in \eqref{E:Gaussian}. Then by scaling, the behavior
of solutions is a function of $p/\sqrt \alpha$. We have
\begin{itemize}
\item
$E[u]>0$ if
\begin{equation}
 \label{E1:posE}
p < \left(6 \sqrt 2 \right)^{1/2} \sqrt \alpha \approx 2.91 \, \sqrt
\alpha;
\end{equation}

\item
the condition on the mass and gradient $\ds \|u_{0} \|^2_{L^2}
\|\nabla u_{0} \|^2_{L^2} < \|Q\|^2_{L^2}\|\nabla Q\|^2_{L^2}$
implies
\begin{equation}
 \label{E1:MG}
p < 2^{1/4} \, \pi^{-3/4}\, \|Q\|_{L^2} \, \sqrt \alpha \approx 2.19
\, \sqrt \alpha;
\end{equation}

\item
the mass-energy condition $ M[u]E[u] < M[Q]E[Q]$ is
$$
\frac{\pi^3 p^4}{4 \alpha^2} \left(3 - \frac{p^2}{2\sqrt2 \alpha}
\right) < \frac12 \, \|Q\|^4_{L^2},
$$
which gives
\begin{equation}
 \label{E1:ME}
p < 1.92 \sqrt \alpha \quad \text{and} \quad p > 2.69 \sqrt \alpha.
\end{equation}

\item
the invariant norm condition $\Vert u_0 \Vert^2_{\dot{H}^{1/2}} <
\Vert Q \Vert^2_{\dot{H}^{1/2}}$ is
\begin{equation}
 \label{E1:H12}
p < (2 \, \pi)^{-1/2} \, \Vert Q \Vert_{\dot{H}^{1/2}} \, \sqrt
\alpha \approx 2.10 \sqrt \alpha.
\end{equation}

\item
(Theorem \ref{T:Lushnikov})
the condition \eqref{E:Lreal} is
\begin{equation}
 \label{E1:green}
p > \left(4 \sqrt 2 \right)^{1/2} \, \sqrt \alpha \approx 2.38 \sqrt
\alpha,
\end{equation}

\item
(Theorem \ref{T:Lushnikov-adapted})
the condition \eqref{E:LAreal} is
\begin{equation}
 \label{E1:orange}
p > \left(\frac{3 \cdot 2^{3/2} \cdot 7^{5/2}}{30 \pi^{1/2}+7^{5/2}}
\right)^{1/2} \sqrt \alpha \approx 2.45 \sqrt \alpha,
\end{equation}

\item
Numerical simulations: the results for the real Gaussian initial
data \eqref{E:Gaussian} are in Table
\ref{T:numerics-gaussian-nophase-alpha}. For $p \geq p_b$ the blow
up was observed, for $p \leq p_s$ the solution dispersed over time.
For example, for $\alpha = 1$ the threshold is between $2.07$ and
$2.08$. This is consistent with the previously reported threshold by
Vlasov et al. in \cite{VPT89} ($p = 2.0764$). From this table it
also follows that ${p_s}/{\sqrt \alpha} \in (2.07, 2.075)$ and
${p_b}/\sqrt \alpha \in (2.077, 2.08)$.

\begin{table}[h]
\begin{tabular} [c]{|l||l|l|l|l|l|l|l|}
\hline $\alpha$ & $0.5$ & $1.0$ & $2.0$  & $4.0$ & $6.0$ & $8.0$ & $10.0$\\
\hline
\hline $p_s$ & $1.46$ & $2.07$ & $2.93$ & $4.15$ & $5.08$ & $5.87$ & $6.56$\\
\hline $p_b$ & $1.47$ & $2.08$ & $2.94$ & $4.16$ & $5.09$ & $5.88$ & $6.57$ \\
\hline
\end{tabular}
\bigskip
\caption{Thresholds for blow up/scattering from numerical
simulations for the Gaussian initial data: for $p \leq p_s$
scattering was observed and for $p \geq p_b$ blow up in finite time
was observed. For comparison the values of $p$ from \eqref{E1:H12}
are also listed.}
  \label{T:numerics-gaussian-nophase-alpha}
\end{table}
In Kuznetsov et. al. \cite{KRRT} it is reported that for the
Gaussian initial data \eqref{E:Gaussian} with $\gamma = 0$, the
numerical condition for collapse
is when two conditions hold:
\begin{equation}
 \label{E:Kuznetsov1}
\frac{\|\nabla u_0\|_2^2 \|u_0\|_2^2}{\|\nabla Q\|_2^2 \|Q\|_2^2} >
0.80255 \quad \text{or} \quad p > \frac{(2 \cdot
0.80255)^{1/4}}{\pi^{3/4}} \|Q\|_{L^2} \, \sqrt \alpha \approx
2.0759 \, \sqrt \alpha,
\end{equation}
and
\begin{equation}
 \label{E:Kuznetsov2}
\frac{E[u_0]M[u_0]}{E[Q]M[Q]}>1.1855 \quad \text{or} \quad 2.0764
\sqrt \alpha < p < 2.6105 \sqrt \alpha. \hspace{3.5cm}
\end{equation}
Thus, the numerical threshold $p/\sqrt \alpha \approx 2.0764$ is
reported in \cite{KRRT}. Our data is consistent with this report,
see also Figure \ref{F:all-gauss-nophase}.
\end{itemize}

To compare the conditions \eqref{E1:MG} - \eqref{E1:orange} with the
numerics, we graph them in Figure \ref{F:all-gauss-nophase}. For
clarity of presentation, and also for comparison with the
$\gamma\neq 0$ case considered next, we plot $\ds \frac{p}{\sqrt
\alpha}$ on the vertical axis vs $\alpha$ on the horizonal.

\begin{figure}[p]
\includegraphics[scale=.81]{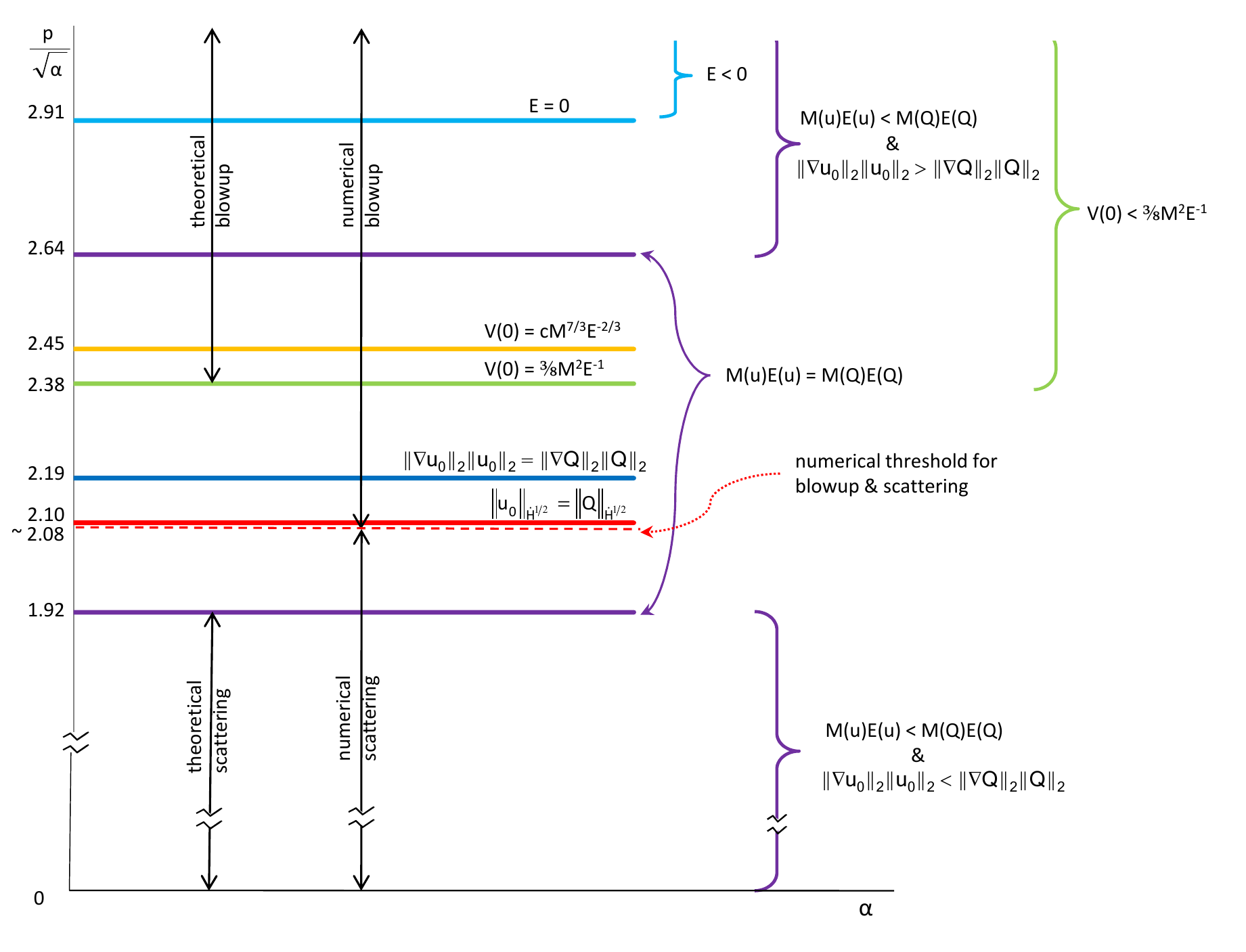}  %
\caption{Global behavior of the solutions with the (real) Gaussian
initial data $u_0(r) = p \, e^{\alpha r^2 / 2}$. The line denoted by
$V(0)=3/8 M^2 E^{-1}$ is the threshold for blow up from Theorem
\ref{T:Lushnikov}, see \eqref{E1:green}; similarly, the line denoted
by $V(0)=c M^{7/3} E^{-2/3}$ is the threshold for blow up from
Theorem \ref{T:Lushnikov-adapted}, see \eqref{E1:orange}. The line
``theoretical scattering" is given by Thm. \ref{T:DHR}, see
\eqref{E1:ME}. The numerical threshold (dashed line) comes from
Table \ref{T:numerics-gaussian-nophase-alpha} normalized by $\sqrt
\alpha $. For all other values refer to text in \S
\ref{S:GaussianReal}.}
 \label{F:all-gauss-nophase}
\end{figure}

\afterpage{\clearpage}

\subsection{Gaussian with a quadratic phase}
Now we consider \eqref{E:Gaussian} with $\gamma \neq 0$. (By scaling
it suffices to consider $\gamma = \pm \frac12$.)  We compute
\begin{itemize}
\item
$E[u]>0$ if
\begin{equation}
 \label{E1:posEphase}
p < \left(6 \sqrt 2 \right)^{1/2} \sqrt \alpha \, \left(1+4
\frac{\gamma^2}{\alpha^2} \right)^{1/2} \approx 2.91 \, \sqrt \alpha
\, \left(1+4 \frac{\gamma^2}{\alpha^2} \right)^{1/2} ;
\end{equation}

\item
the condition on the mass and gradient $\ds \|u_{0} \|^2_{L^2}
\|\nabla u_{0} \|^2_{L^2} < \|Q\|^2_{L^2}\|\nabla Q\|^2_{L^2}$ is
\begin{equation}
\label{E1:MG+phase}
p < \frac{2^{1/4} \, \|Q\|_{L^2}}{\pi^{3/4}} \, \frac{\sqrt
\alpha}{(1+4\frac{\gamma^2}{\alpha^2})^{1/4}} \approx 2.19 \, {\sqrt
\alpha} \, {\left(1 + 4\frac{\gamma^2}{\alpha^2}\right)^{-1/4}};
\end{equation}

\item
the mass-energy condition $ M[u]E[u] < M[Q]E[Q]$ is
\begin{equation}
 \label{E1:ME+phase}
\frac{\pi^3 p^4}{4 \,\alpha^2}
\left(3\left(1+4\frac{\gamma^2}{\alpha^2}\right) -
\frac{p^2}{2\sqrt2 \alpha} \right) < \frac12 \|Q\|^4_{L^2}. %
\end{equation}
The left hand side is a cubic polynomial in $y^2=p^2/\alpha$ and can
be solved explicitly to obtain $\ds \frac{p}{\alpha^{1/2}} \text{ as
a function of }\alpha $, though with a very complicated expression.
We list a few values for $\gamma = \pm \frac12$ in Table
\ref{T1:ME}.
\begin{table}[h]
\begin{tabular}{| c || c | c | c | c | c | c | c | c |}
\hline $\alpha$ & $0.25$ & $0.5$ & $1$ & $2$ & $4$ & $6$ & $8$ & $10$ \\
\hline $p_1$ & $0.41$ & $0.79$ & $1.45$ & $2.42$ & $3.70$ & $4.62$ & $5.38$ & $6.04$\\
\hline $p_2$ & $6.01$ & $4.60$ & $4.09$ & $4.46$ & $5.65$ & $6.75$ & $7.71$ & $8.58$\\
\hline
\end{tabular}
\bigskip
\caption{The positive real roots of the equation in
\eqref{E1:ME+phase}.}
 \label{T1:ME}
\end{table}
The inequality in \eqref{E1:ME+phase} holds for $p < p_1$ and
$p>p_2$.

\item
the invariant norm condition $\Vert u_0 \Vert^2_{\dot{H}^{1/2}} <
\Vert Q \Vert^2_{\dot{H}^{1/2}}$ is
\begin{equation}
 \label{E1:p12-phase}
\quad p < \alpha \left( \frac{27.72665}{2\, \pi \,
(\alpha^2+4\gamma^2)^{1/2}}\right)^{1/2} \approx 2.10\, {\sqrt
\alpha} \, {\left(1+ 4\frac{\gamma^2}{\alpha^2}\right)^{-1/4}} ;
\end{equation}

\item
(Theorem \ref{T:Lushnikov})
the condition $\omega \leq 1$ from \eqref{E:L+} amounts to
\begin{equation}
 \label{E1:Lphase}
p \geq \left(4 \sqrt 2 \right)^{1/2} \, \sqrt \alpha \, \left(1+6
\frac{\gamma^2}{\alpha^2} \right)^{1/2} \approx 2.38 \sqrt \alpha \,
\left(1+6 \frac{\gamma^2}{\alpha^2} \right)^{1/2},
\end{equation}
similarly, $\omega \geq 1$ from \eqref{E:L-} will be the above with
the reversed sign. The condition \eqref{E:L+simple} for positive
$\gamma$ with $y=p/\sqrt \alpha$ is
\begin{equation}
y^6 - 6 \sqrt 2 \left(1+4\frac{\gamma^2}{\alpha^2}\right) y^4 + 64
\sqrt 2 > 0
\end{equation}
and with the reversed inequality sign for negative $\gamma$. The
last inequality is cubic in $y^2$ producing two positive roots,
which are listed in Table \ref{T1:Lgauss} for $\gamma = \pm\frac12$.
\begin{table}[h]
\begin{tabular}{| c || c | c | c | c | c | c | c | c |}
\hline $\alpha$ & $0.25$ & $0.5$ & $1$ & $2$ & $4$ & $6$ & $8$ & $10$ \\
\hline $p_b$ & $0.45$ & $0.86$ & $1.58$ & $2.68$ & $4.22$ & $5.37$ & $6.33$ & $7.16$\\
\hline $p_t$ & $6.01$ & $4.60$ & $4.08$ & $4.39$ & $5.42$ & $6.35$ & $7.17$ & $7.91$\\
\hline
\end{tabular}
\bigskip
\caption{The positive real roots, $p_t$ and $p_b$, of the function
in \eqref{E1:Lphase} for the condition \eqref{E:L+simple} (Theorem
\ref{T:Lushnikov}). }
 \label{T1:Lgauss}
\end{table}

\item
(Theorem \ref{T:Lushnikov-adapted})
the condition $\kappa \leq 1$ from \eqref{E:LA+} amounts to
\begin{equation}
 \label{E1:LAorange-phase}
p \geq \left(\frac{3 \cdot 2^{3/2} \cdot 7^{5/2}}{30
\pi^{1/2}+7^{5/2}} \right)^{1/2} \sqrt \alpha \, \left(1+ 4
\frac{\gamma^2}{\alpha^2} \right)^{1/2} \approx 2.45 \sqrt \alpha \,
\left(1+ 4 \frac{\gamma^2}{\alpha^2} \right)^{1/2},
\end{equation}
similarly, $\kappa \geq 1$ from \eqref{E:LA-} is the above
inequality with the reversed sign. The condition \eqref{E:LA+simple}
with $y=p/\sqrt \alpha$ for positive $\gamma$ is
\begin{equation}
 \label{E1:LAphase}
\left(\frac{15 \sqrt{2\pi}}{7^{5/2}} - \frac1{2^{3/2}} \right) y^2
+3 - %
\left(\frac{3^5\cdot 5^2\pi}{2\cdot 7^5} \right)^{1/3}
\left(3\left(1+4 \frac{\gamma^2}{\alpha^2}\right)y^4 -
\frac{y^6}{2^{3/2}} \right)^{1/3}> 0,
\end{equation}
and with the reversed sign for negative $\gamma$. The positive real
zeros of the function in \eqref{E1:LAphase} for $\gamma =
\pm\frac12$ are listed in Table \ref{T1:LAgauss}.
\begin{table}[h]
\begin{tabular}{| c || c | c | c | c | c | c | c | c |}
\hline $\alpha$ & $0.25$ & $0.5$ & $1$ & $2$ & $4$ & $6$ & $8$ & $10$ \\
\hline $p_b$ & $0.48$ & $0.93$ & $1.68$ & $2.81$ & $4.39$ & $5.57$ & $6.55$ & $7.41$\\
\hline $p_t$ & $6.01$ & $4.61$ & $4.10$ & $4.46$ & $5.54$ & $6.51$ & $7.36$ & $8.13$\\
\hline
\end{tabular}
\bigskip
\caption{The positive real roots, $p_t$ and $p_b$, of the function
in \eqref{E1:LAphase} for the condition \eqref{E:LA+simple} (Theorem
\ref{T:Lushnikov-adapted}). }
 \label{T1:LAgauss}
\end{table}

\item
Numerical simulations: we fix the quadratic phase $\gamma = \pm .5$
and vary the parameter $\alpha$ in our numerical simulations in
order to obtain the blow up threshold, see data in Table
\ref{T1:phase-num}. Note that we obtain different thresholds
depending on the sign of $\gamma$.
\begin{table}[h]
\begin{tabular} [c]{|c||l|l|l|l|l|l|l|}
\hline $\alpha$ & $0.5$ & $1.0$ & $2.0$  & $4.0$ & $6.0$ & $8.0$ & $10.0$\\
\hline
\hline $p_s^+$ & $2.8$ & $3.00$ & $3.56$ & $4.58$ & $5.43$ & $6.17$ & $6.83$\\
\hline $p_b^+$ & $2.8$ & $3.01$ & $3.57$ & $4.59$ & $5.44$ & $6.18$ & $6.835$\\
\hline $p_s^-$ & $0.79$ & $1.44$ & $2.42$ & $3.76$ & $4.76$ & $5.59$ & $6.313$\\
\hline $p_b^-$ & $0.80$ & $1.45$ & $2.43$ & $3.77$ & $4.77$ & $5.60$ & $6.315$\\
\hline
\end{tabular}
\bigskip
\caption{Thresholds from numerical simulations for the Gaussian
initial data with phase $\gamma = \pm .5$: for $p \geq p_b^\pm$ blow
up in finite time was observed and for $p \leq p_s^\pm$ scattering
was observed; $+$ superscript indicates the threshold for the
positive phase $\gamma=.5$ and $-$ superscript indicates the
negative phase $\gamma=-.5$.} %
\end{table}
 \label{T1:phase-num}
\end{itemize}

To compare conditions \eqref{E1:posEphase} -
\eqref{E1:LAphase} with the numerics, %
we graph the dependence of $p$ on $\alpha$ in Figures \ref{F:gaussp}
and \ref{F:gaussn} separately for positive and negative values of
the phase $\gamma = \pm \frac12$. We plot $\ds \frac{p}{\sqrt
\alpha}$ on the vertical axis to observe the asymptotics as $\alpha
\to \infty$, and thus, approaches the case of the real Gaussian
initial data.

\begin{figure}[p]
\includegraphics[scale=.73]{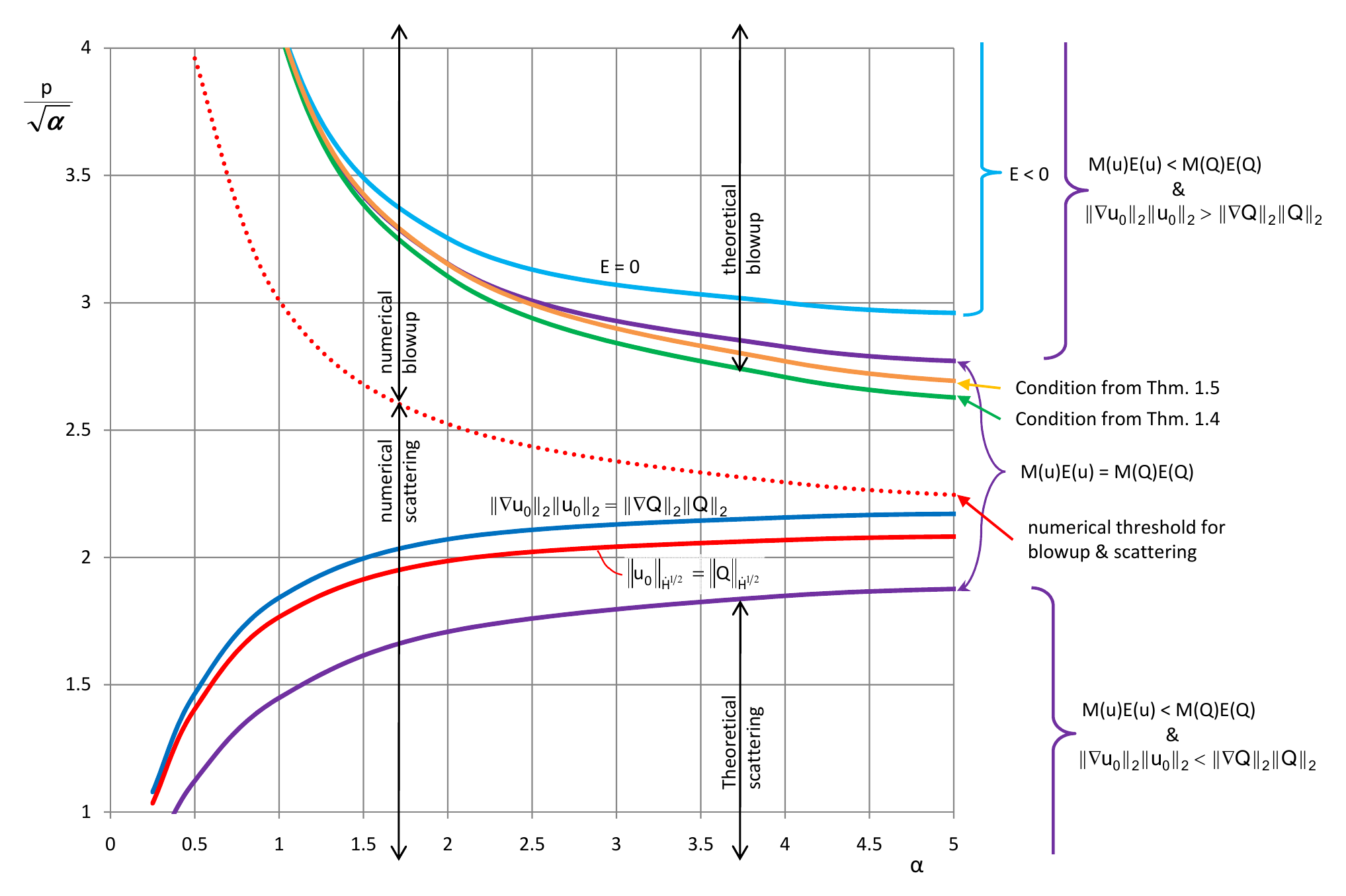}  %
\caption{Global behavior of the solutions to the Gaussian initial
data with the positive quadratic phase $u_0(r) = p \, e^{-\alpha
r^2/2} \, e^{+ i \frac12 r^2}$. The curve denoted by ``Condition
from Thm. \ref{T:Lushnikov}" comes from the largest positive root of
the equation in \eqref{E1:Lphase}, namely, values $p_t/\sqrt \alpha$
in Table \ref{T1:Lgauss}. Similarly, the curve denoted by
``Condition from Thm. \ref{T:Lushnikov-adapted}" comes from the
largest positive root of the equation in \eqref{E1:LAphase}, namely,
values $p_t/\sqrt \alpha$ in Table \ref{T1:LAgauss}. The curve
``theoretical scattering" is provided by Thm. \ref{T:DHR}, see
\eqref{E1:ME+phase} and values $p_1$ in Table \ref{T1:ME}. The
numerical threshold (dotted curve) comes from values $p_s^+$,
$p_b^+$ in Table \ref{T1:phase-num} normalized by $\sqrt \alpha$. }
 \label{F:gaussp}
\end{figure}

\begin{figure}[p]
\includegraphics[scale=.74]{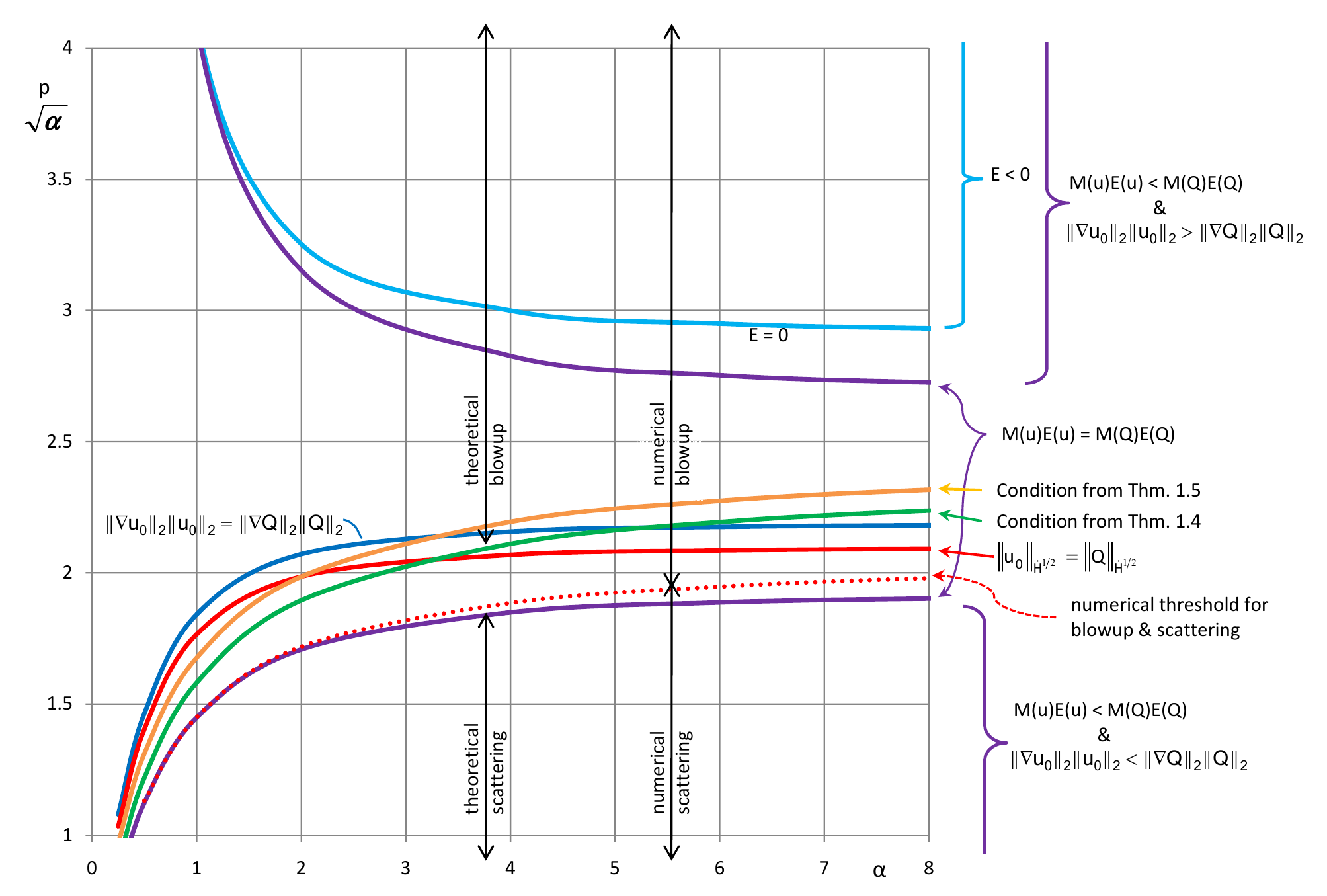}  %
\caption{Global behavior of the solutions to the Gaussian initial
data with the negative quadratic phase $u_0(r) = p \, e^{-\alpha
r^2/2} \, e^{- i \frac12 r^2}$. The curve denoted by ``Condition
from Thm. \ref{T:Lushnikov}" comes from the smallest positive root
of the equation in \eqref{E1:Lphase}, namely, values $p_b/\sqrt
\alpha$ in Table \ref{T1:Lgauss}. Similarly, the curve denoted by
``Condition from Thm. \ref{T:Lushnikov-adapted}" comes from the
smallest positive root of the equation in \eqref{E1:LAphase},
namely, values $p_b/\sqrt \alpha$ in Table \ref{T1:LAgauss}. The
curve ``theoretical scattering" is the same as in Fig.
\ref{F:gaussp} and is provided by Thm. \ref{T:DHR}, see
\eqref{E1:ME+phase}  and values $p_1$ in Table \ref{T1:ME}. The
numerical threshold (dotted curve) comes from values $p_s^-$,
$p_b^-$ in Table \ref{T1:phase-num} normalized by $\sqrt \alpha$.
This plot illustrates that both Theorems \ref{T:Lushnikov} and
\ref{T:Lushnikov-adapted} provide ranges of the Gaussian initial
data $u_0$ such that $\|u_0\|_{\dot{H}^{1/2}} <
\|Q\|_{\dot{H}^{1/2}}$ and $u(t)$ blows up in finite time.}
 \label{F:gaussn}
\end{figure}

\subsection{Conclusions}
The above computations show
\begin{enumerate}
\item
Consistency with Conjecture 3: if $u_0$ is \emph{real}, then
$\|u_0\|_{\dot H^{1/2}}< \|Q\|_{\dot H^{1/2}}$ implies $u(t)$
scatters (numerical data).

\item
Consistency with Conjecture 4 (since $u_0$ is based upon a radial
profile that is \emph{monotonically decreasing}): if $\|u_0\|_{\dot
H^{1/2}}> \|Q\|_{\dot H^{1/2}}$, then $u(t)$ blows-up in finite time
(numerical data) .

\item
Theoretical proof of the Conjecture 2, i.e.,
``$\|u_0\|_{\dot{H}^{1/2}}< \|Q\|_{\dot{H}^{1/2}}$ {\it implies
scattering}'' is false for an arbitrary radial data. We show that it
is possible to produce a radial Gaussian initial data with negative
phase (e.g., for $0 < \lambda < 3.3$ by Theorem \ref{T:Lushnikov})
such that $\|u_0\|_{\dot{H}^{1/2}}< \|Q\|_{\dot{H}^{1/2}}$ and the
solution $u(t)$ blows up in finite time.

\item
The condition ``$\|u_0\|_{L^2}\|\nabla u_0\|_{L^2}<
\|Q\|_{L^2}\|\nabla Q\|_{L^2}$ {\it implies scattering}'' is not
valid (unless $M[u]E[u]<M[Q]E[Q]$ as in Theorem \ref{T:DHR}).

\item
In all three cases (real data, data with positive phase and with
negative phase) Theorems \ref{T:Lushnikov} and
\ref{T:Lushnikov-adapted} provide new range on blow up than was
previously known from our Theorem \ref{T:DHR}.

\item
For the Gaussian initial data with negative phase and small values
of $\alpha$ ($0 < \alpha  \lesssim 2$, see Figure \ref{F:gaussn})
the numerical threshold for scattering and blow up coincides with
the scattering threshold provided by Theorem \ref{T:DHR}. As $\alpha
\to \infty$ the numerical threshold (for both positive and negative
phases) approaches the values $\|u_0\|_{\dot{H}^{1/2}} =
\|Q\|_{\dot{H}^{1/2}}$.
\end{enumerate}

\afterpage{\clearpage}

\section{Super-Gaussian profile}
\label{S:super-Gaussian}

Now we modify the initial Gaussian data to the ``super Gaussian"
profile:
\begin{equation}
 \label{E:supergauss}
u_0(r) = p \, e^{-\alpha \, r^4/2} \, e^{i \gamma \, r^2}, \quad r =
|x|, \quad x \in \cR^3.
\end{equation}
For this initial data we calculate
\begin{align*}
& M[u] = %
\frac{\pi \, p^2}{\alpha^{3/4}} \, \Gamma\left(\frac34\right), %
\qquad \Vert \nabla u_0 \Vert^2_{L^2} =  \frac{\pi^2 \,
p^2}{2\sqrt{2}\,
\alpha^{1/4} \, \Gamma(\frac34)}\left(5+4\frac{\gamma^2}{\alpha} \right),\\
& E[u] = \frac{\pi p^2}{4 \sqrt{2} \alpha^{1/4} \, {\Gamma
\left(\frac34\right)}} \left(\pi\left(5+
4\frac{\gamma^2}{\alpha}\right) - \frac{[\Gamma
\left(\frac34\right)]^2}{2^{1/4}} \frac{p^2}{ \sqrt{\alpha}}
\right), \\
& V(0) %
= \frac{\pi^{2} p^2}{2 \sqrt{2} \, \alpha^{5/4}\, \Gamma
\left(\frac34\right) }, \qquad  V_t(0) = \frac{2 \sqrt{2} \pi^{2}
p^2 \gamma}{\alpha^{5/4}\, \Gamma \left(\frac34\right) }.
\end{align*}
Here, $\ds \Gamma(s+1) = \int_0^\infty e^{-t} \, t^s \, dt. $

\subsection{Real super Gaussian}
When $\gamma=0$, we have
\begin{itemize}
\item
$E[u]>0$ if
\begin{equation}
 \label{E2:posE}
p < \frac{2^{1/8} \sqrt{5 \pi}}{\Gamma \left(\frac34\right)} \,
\alpha^{1/4} \approx 3.53 \, \alpha^{1/4};
\end{equation}

\item
the condition on the mass and gradient $\ds \|u_{0} \|^2_{L^2}
\|\nabla u_{0} \|^2_{L^2} < \|Q\|^2_{L^2}\|\nabla Q\|^2_{L^2}$
implies
\begin{equation}
 \label{E2:MG}
p < \frac{2^{3/8}\, 3^{1/4}}{5^{1/4} \pi^{3/4}} \, \|Q\|_{L^2} \,
\alpha^{1/4} \approx 2.10 \, \alpha^{1/4};
\end{equation}

\item
the mass-energy condition $ M[u]E[u] < M[Q]E[Q]$ is
$$
\frac{\pi^2 p^4}{4 \sqrt{2}\alpha} \left({5 \pi} -
\frac{\left[\Gamma \left(\frac34\right)\right]^2}{2^{1/4}}
\frac{p^2}{ \sqrt{\alpha}} \right) < \frac12 \|Q\|^4_{L^2},
$$
and thus, we obtain
\begin{equation}
 \label{E2:ME}
p < 1.71 \,\alpha^{1/4} \quad \text{and} \quad p > 3.44
\,\alpha^{1/4}.
\end{equation}

\item
(Theorem \ref{T:Lushnikov}) the condition \eqref{E:Lreal} is
\begin{equation}
 \label{E2:green}
p > \frac{2^{1/8} \pi^{1/2}}{\Gamma(\frac34)} \left(5 -
\frac{6}{\pi^2}\left[\Gamma \left(\frac34 \right)\right]^4
\right)^{1/2}
\, \alpha^{1/4} \approx 3.00 \, \alpha^{1/4}.
\end{equation}

\item
(Theorem \ref{T:Lushnikov-adapted}) the condition \eqref{E:LAreal}
is
\begin{equation}
 \label{E2:orange}
p > \frac{2^{1/8} \cdot 5^{1/2} \cdot \pi^{3/4}
C^{7/2}}{\Gamma(\frac34)(C^7 \pi^{1/2} + 4
[\Gamma(\frac34)]^4)^{1/2}} \, \alpha^{1/4} \approx 2.89 \,
\alpha^{1/4}.
\end{equation}

\item
the invariant norm condition $\Vert u_0 \Vert^2_{\dot{H}^{1/2}} <
\Vert Q \Vert^2_{\dot{H}^{1/2}}$ is given in Table \ref{T2:num+H12}.
Here, we first compute the Fourier transform of $u_0$ using
\eqref{E:Fourier-radial-Bessel} and then $\dot{H}^{1/2}$ norm by
\eqref{E:H12norm-radial} which is listed in the second row of the
Table \ref{T2:num+H12} for various $\alpha$. The third row indicates
the values of $p$, denoted by $p_{1/2}$, for the threshold in the
invariant norm condition. We observe that $p_{1/2}/ \sqrt[4]\alpha
\approx 2.02$.

\item
Numerical simulations: the results with the super Gaussian initial
data \eqref{E:supergauss} are in Table \ref{T2:num+H12}. For $p \geq
p_b$ the blow up was observed, for $p \leq p_s$ the solution
dispersed over time. We observe that $p_s/\sqrt[4]\alpha \approx
2.01$ and $p_b\sqrt[4]\alpha \approx 2.02$. For convenience Table
\ref{T2:num+H12} contains the $\dot{H}^{1/2}$ norm calculations.
\end{itemize}
\begin{table}[h]
\begin{tabular}{| c || c | c | c | c | c | c | c | c |}
\hline $\alpha$ & $.25$ & $.5$ & $1$ & $2$ & $4$ & $6$ & $8$ & $10$ \\
\hline \hline $\frac1{p^2}\|u_0\|_{\dot{H}^{1/2}}$
& $13.54$ & $9.58$ & $6.77$ & $4.79$ & $3.39$ & $2.76$ & $2.39$ & $2.14$\\
\hline
\hline $p_{1/2}$ & $1.43$ & $1.70$ & $2.02$ & $2.41$ & $2.86$ & $3.17$ & $3.41$ & $3.60$\\
\hline
\hline $p_s$ & $1.42$ & $1.69$ & $2.01$ & $2.39$ & $2.85$ & $3.15$ & $3.39$ & $3.58$\\
\hline $p_b$ & $1.43$ & $1.70$ & $2.02$ & $2.40$ & $2.86$ & $3.16$ & $3.40$ & $3.59$\\
\hline
\end{tabular}
\bigskip
\caption{The $\dot{H}^{1/2}$ norm of the super Gaussian initial data
depending on $\alpha$ and values of $p$ for the $\dot{H}^{1/2}$
condition as well as the numerical results for blow up threshold and
global existence - $p \geq p_b$ and $p \leq p_s$, correspondingly. }
  \label{T2:num+H12}
\end{table}

To compare the conditions \eqref{E2:posE} - \eqref{E2:orange} with
numerical data, we graph the dependence of $p$ on $\alpha$ in Figure
\ref{F:supergauss-nophase}. For clarity of presentation we plot $\ds
\frac{p}{\sqrt[4] \alpha}$ on the vertical axis.

\begin{figure}[p]
\includegraphics[scale=.81]{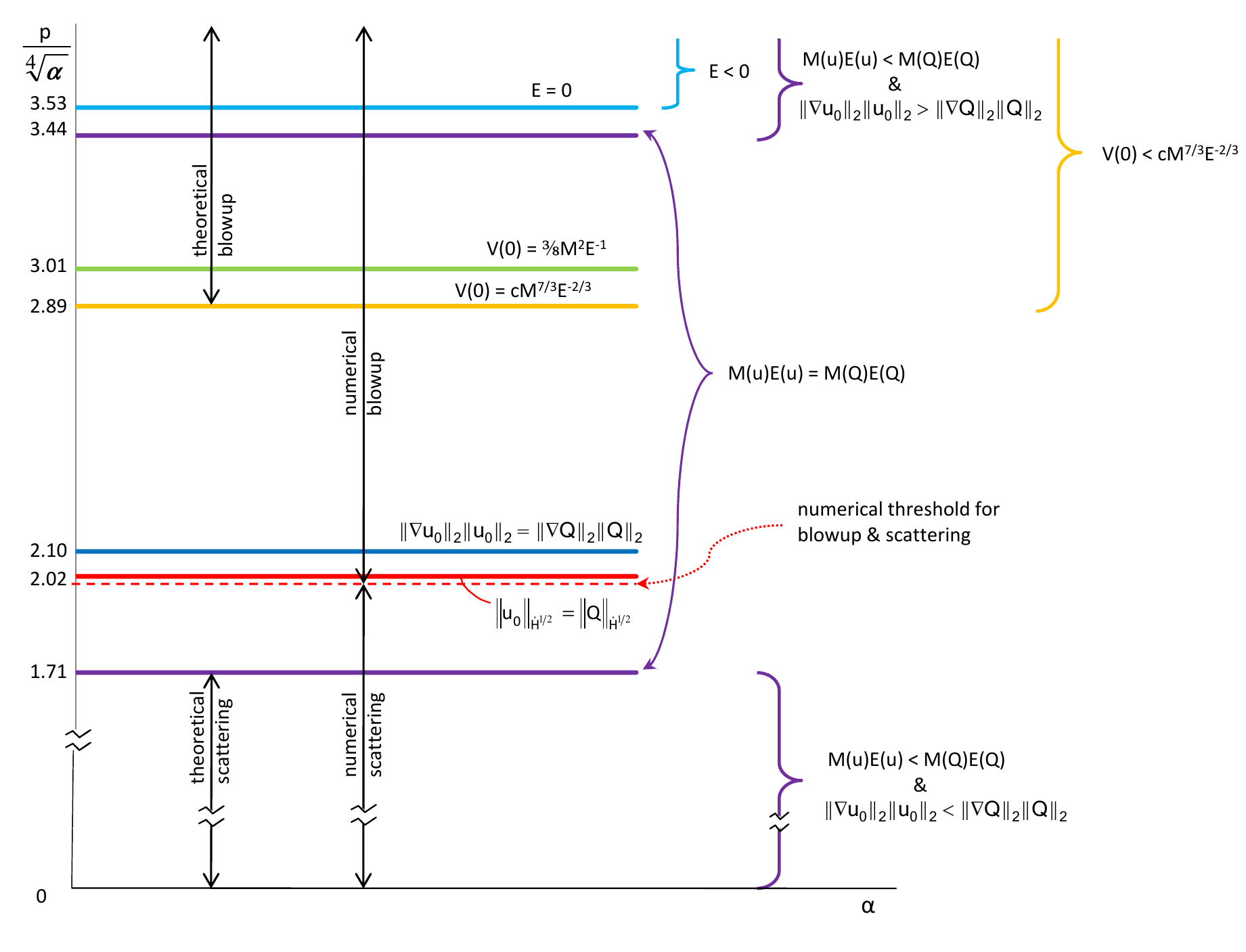}  %
\caption{Global behavior of the solutions with the super Gaussian
initial data $u_0(r) = p \, e^{-\alpha r^4/2}$. The line denoted by
$V(0)=3/8 M^2 E^{-1}$ is the threshold for blow up from Theorem
\ref{T:Lushnikov}, see \eqref{E2:green}; similarly, the line denoted
by $V(0)=c M^{7/3} E^{-2/3}$ is the threshold for blow up from
Theorem \ref{T:Lushnikov-adapted}, see \eqref{E2:orange}. The line
``theoretical scattering" is given by Thm. \ref{T:DHR}, see
\eqref{E2:ME}. The numerical threshold (dashed line) comes from
Table \ref{T2:num+H12} normalized by $\sqrt[4] \alpha $. For all
other values refer to text.}
 \label{F:supergauss-nophase}
\end{figure}

\subsection{Super Gaussian with a quadratic phase}
When $\gamma \neq 0$, %
we have
\begin{itemize}
\item
$E[u]>0$ if
\begin{equation}
 \label{E2:posEphase}
p < \frac{2^{1/8} \pi^{1/2}}{\Gamma \left(\frac34\right)}
\left(5+4\frac{\gamma^2}{\alpha} \right)^{1/2} \, \alpha^{1/4}
\approx 1.58 \, \alpha^{1/4} \,
\left(5+4\frac{\gamma^2}{\alpha}\right)^{1/2};
\end{equation}

\item
the condition on the mass and gradient $\ds \|u_{0} \|^2_{L^2}
\|\nabla u_{0} \|^2_{L^2} < \|Q\|^2_{L^2}\|\nabla Q\|^2_{L^2}$
implies
\begin{equation}
 \label{E2:MGphase}
p < \frac{2^{3/8} \,3^{1/2}}{\pi^{3/4}
(5+4\frac{\gamma^2}{\alpha})^{1/4}\,} \, \|Q\|_{L^2} \, \alpha^{1/4}
\approx 3.15 \, \alpha^{1/4} \, \left(5+4\frac{\gamma^2}{\alpha}
\right)^{-1/4};
\end{equation}

\item
the mass-energy threshold $ M[u]E[u] = M[Q]E[Q]$ is
\begin{equation}
 \label{E2:MEphase}
\frac{\pi^2 p^4}{4 \sqrt{2}\alpha} \left(\pi\left(5 + 4
\frac{\gamma^2}{\alpha} \right) - \frac{\left[\Gamma
\left(\frac34\right)\right]^2}{2^{1/4}} \frac{p^2}{\sqrt{\alpha}}
\right) = \frac12 \, \|Q\|^4_{L^2},
\end{equation}
the real positive zeros of the above expression when $\gamma = \pm
\frac12$ are in Table \ref{T:1-2super}.
\begin{table}[h]
\begin{tabular}{| c || c | c | c | c | c | c | c | c |}
\hline $\alpha$ & $.25$ & $.5$ & $1$ & $2$ & $4$ & $6$ & $8$ & $10$ \\
\hline $p_1$ & $1.00$ & $1.28$ & $1.60$ & $1.96$ & $2.37$ & $2.64$ & $2.85$ & $3.02$\\
\hline $p_2$ & $3.34$ & $3.48$ & $3.81$ & $4.32$ & $5.01$ & $5.49$ & $5.88$ & $6.20$\\
\hline
\end{tabular}
\bigskip
\caption{ The positive real roots of the equation in
\eqref{E2:MEphase}. }
 \label{T:1-2super}
\end{table}

\item
for the invariant norm condition $\Vert u_0 \Vert^2_{\dot{H}^{1/2}}
< \Vert Q \Vert^2_{\dot{H}^{1/2}}$, we compute the Fourier Transform
of $\ds u_{0}(r) = p \, e^{- \alpha r^4/2} \, e^{i \gamma \, r^2}$
again by \eqref{E:Fourier-radial-Bessel} and then $\dot{H}^{1/2}$
norm by \eqref{E:H12norm-radial}. The values of $p$ when the
condition $\|u_0\|_{\dot{H}^{1/2}} = \|Q\|_{\dot{H}^{1/2}}$ are in
Table \ref{T:H12-super-phase} as well.

\begin{table}[h]
\begin{tabular}{| c || c | c | c | c | c | c | c | c |}
\hline $\alpha$ & $.25$ & $.5$ & $1$ & $2$ & $4$ & $6$ & $8$ & $10$ \\
\hline $\frac1{p^2}\|u_0\|_{\dot{H}^{1/2}}$
& $17.53$ & $11.03$ & $7.29$ & $4.97$ & $3.45$ & $2.80$ & $2.42$ & $2.16$\\
\hline $p_{1/2}$ & $1.26$ & $1.59$ & $1.95$ & $2.36$ & $2.83$ & $3.15$ & $3.39$ & $3.59$\\
\hline
\end{tabular}
\bigskip
\caption{ The $\dot{H}^{1/2}$ norm of the super Gaussian initial
data with the phase $\gamma = \pm \frac12$ depending on $\alpha$ is
listed in the second row and values of $p$ for this condition are in
the last. }
 \label{T:H12-super-phase}
\end{table}

\item
(Theorem \ref{T:Lushnikov}) the condition $\omega \leq 1$ from
\eqref{E:L+} is
\begin{equation}
 \label{E2:greenphase}
p > \frac{2^{1/8} \pi^{1/2}}{\Gamma \left(\frac34 \right)} \left(
\left(5+4\frac{\gamma^2}{\alpha}\right) -\frac{6}{\pi^2}
\left[\Gamma \left(\frac34 \right)\right]^4 \right)^{1/2}
\alpha^{1/4} \approx 1.58 \left(3.63 + 4 \frac{\gamma^2}{\alpha}
\right)^{1/2} \alpha^{1/4}, %
\end{equation}
and the condition \eqref{E:L+simple} (with $y = p/\alpha^{1/4}$) is
\begin{equation}
 \label{E2:Lsuperphase}
\frac{5\pi^2}{6 [\Gamma(\frac34)]^4} - 3 -\frac{\pi}{2^{5/4} 3
[\Gamma(\frac34)]^2} \, y^2 +
\frac{2^{3/2}3^{1/2}[\Gamma(\frac34)]^2}{\pi\left(5
+\frac{4\gamma^2}{\alpha}
-\frac{[\Gamma(\frac34)]^2}{2^{1/4}\, \pi}
\, y^2 \right)^{1/2}} > 0,
\end{equation}
the positive real zeros of the left hand side with $\gamma =
\pm\frac12$ are in Table \ref{T2:Lsuper}.
\begin{table}[h]
\begin{tabular}{| c || c | c | c | c | c | c | c | c |}
\hline $\alpha$ & $.25$ & $.5$ & $1$ & $2$ & $4$ & $6$ & $8$ & $10$ \\
\hline $p_b$ & $1.62$ & $2.04$ & $2.56$ & $3.18$ & $3.90$ & $4.38$ & $4.75$ & $5.05$\\
\hline $p_t$ & $3.32$ & $3.43$ & $3.71$ & $4.13$ & $4.69$ & $5.10$ & $5.42$ & $5.68$\\
\hline
\end{tabular}
\bigskip
\caption{The positive real roots, $p_t$ and $p_b$, of the function
in \eqref{E2:Lsuperphase} for the condition \eqref{E:L+simple}
(Theorem \ref{T:Lushnikov}). }
 \label{T2:Lsuper}
\end{table}

\item
(Theorem \ref{T:Lushnikov-adapted}) the condition \eqref{E:LA+} is
\begin{equation}
 \label{E2:orangedash}
\begin{aligned}
p &> \frac{2^{1/8} \pi^{3/4} \, C^{7/2}}{\Gamma(\frac34) \left(C^7
\pi^{1/2} + 4[\Gamma(\frac34)]^4 \right)^{1/2}} \, \left(1+4
\frac{\gamma^2}{\alpha} \right)^{1/2} \alpha^{1/4} \\
&\approx 1.29 \, \left(1+4 \frac{\gamma^2}{\alpha} \right)^{1/2} \,
\alpha^{1/4},
\end{aligned}
\end{equation}
and the condition \eqref{E:LA+simple} with $y = p/\alpha^{1/4}$ is
\begin{equation}
\label{E2:LAsuperphase}
\begin{aligned}
\left(\frac{2^{7/4}[\Gamma(\frac34)]^2}{C^7 \pi^{1/2}} -
\frac1{2^{5/4}[\Gamma(\frac34)]^2} \right)y^2 + \frac{5 \pi}{2
[\Gamma(\frac34)]^4} \qquad & \\
-\frac{3 \cdot 2^{1/6}}{C^{14/3}} \left(\left(5 +4
\frac{\gamma^2}{\alpha}\right)y^4 - \frac{[\Gamma(\frac34
)]^2}{2^{1/4} \pi} y^6\right)^{1/3} & > 0, 
\end{aligned}
\end{equation}
the positive real zeros of the left hand side with $\gamma =
\pm\frac12$ are in Table \ref{T2:LAsuperphase}.
\begin{table}[h]
\begin{tabular}{| c || c | c | c | c | c | c | c | c |}
\hline $\alpha$ & $.25$ & $.5$ & $1$ & $2$ & $4$ & $6$ & $8$ & $10$ \\
\hline $p_b$ & $1.41$ & $1.84$ & $2.36$ & $2.50$ & $3.68$ & $4.15$ & $4.51$ & $4.80$\\
\hline $p_t$ & $3.30$ & $3.41$ & $3.66$ & $4.05$ & $4.59$ & $4.97$ & $5.27$ & $5.53$\\
\hline
\end{tabular}
\bigskip
\caption{The positive real roots, $p_t$ and $p_b$, of the function
in \eqref{E2:LAsuperphase} for the condition \eqref{E:LA+simple}
(Theorem \ref{T:Lushnikov-adapted}). }
 \label{T2:LAsuperphase}
\end{table}
\item
Numerical simulations: the results for the super Gaussian with the
quadratic phase $\gamma = \pm\frac12$ are given in Table
\ref{T2:ph+super}. The rows denoted by $p_b^+$, $p_s^+$, $p_b^-$ and
$p_s^-$ are thresholds for blow up and global existence/scattering
for the positive phase $\gamma = +\frac12$ and negative phase
$\gamma = - \frac12$, correspondingly.
\end{itemize}
To compare the conditions \eqref{E2:posEphase} -
\eqref{E2:LAsuperphase} with the numerics, %
we graph the dependence of $p$ on $\alpha$ in Figures
\ref{F:supergaussp} and \ref{F:supergaussn} separately for positive
and negative values of the phase $\gamma = \pm \frac12$. We plot
$\ds \frac{p}{\sqrt[4] \alpha}$ on the vertical axis to observe the
asymptotics as $\alpha \to \infty$ .
\begin{table}[h]
\begin{tabular}{| c || c | c | c | c | c | c | c | c |}
\hline $\alpha$ & $.25$ & $.5$ & $1$ & $2$ & $4$ & $6$ & $8$ & $10$ \\
\hline
\hline $p_s^+$ & $2.18$ & $2.27$ & $2.47$ & $2.76$ & $3.14$ & $3.41$ & $3.63$ & $3.81$\\
\hline $p_b^+$ & $2.19$ & $2.28$ & $2.48$ & $2.77$ & $3.15$ & $3.42$ & $3.64$ & $3.82$\\
\hline
\hline $p_s^-$ & $1.02$ & $1.33$ & $1.69$ & $2.12$ & $2.60$ & $2.92$ & $3.17$ & $3.38$\\
\hline $p_b^-$ & $1.03$ & $1.34$ & $1.70$ & $2.11$ & $2.61$ & $2.93$ & $3.18$ & $3.39$\\
\hline
\end{tabular}
\bigskip
\caption{Numerical results for blow up threshold ($p \geq
p_b^{\pm}$) and global existence ($p \leq p_s^{\pm}$). The sign in
the superscript indicates the positive or negative sign of $\gamma$,
correspondingly.}
 \label{T2:ph+super}
\end{table}

\begin{figure}[p]
\includegraphics[scale=.90]{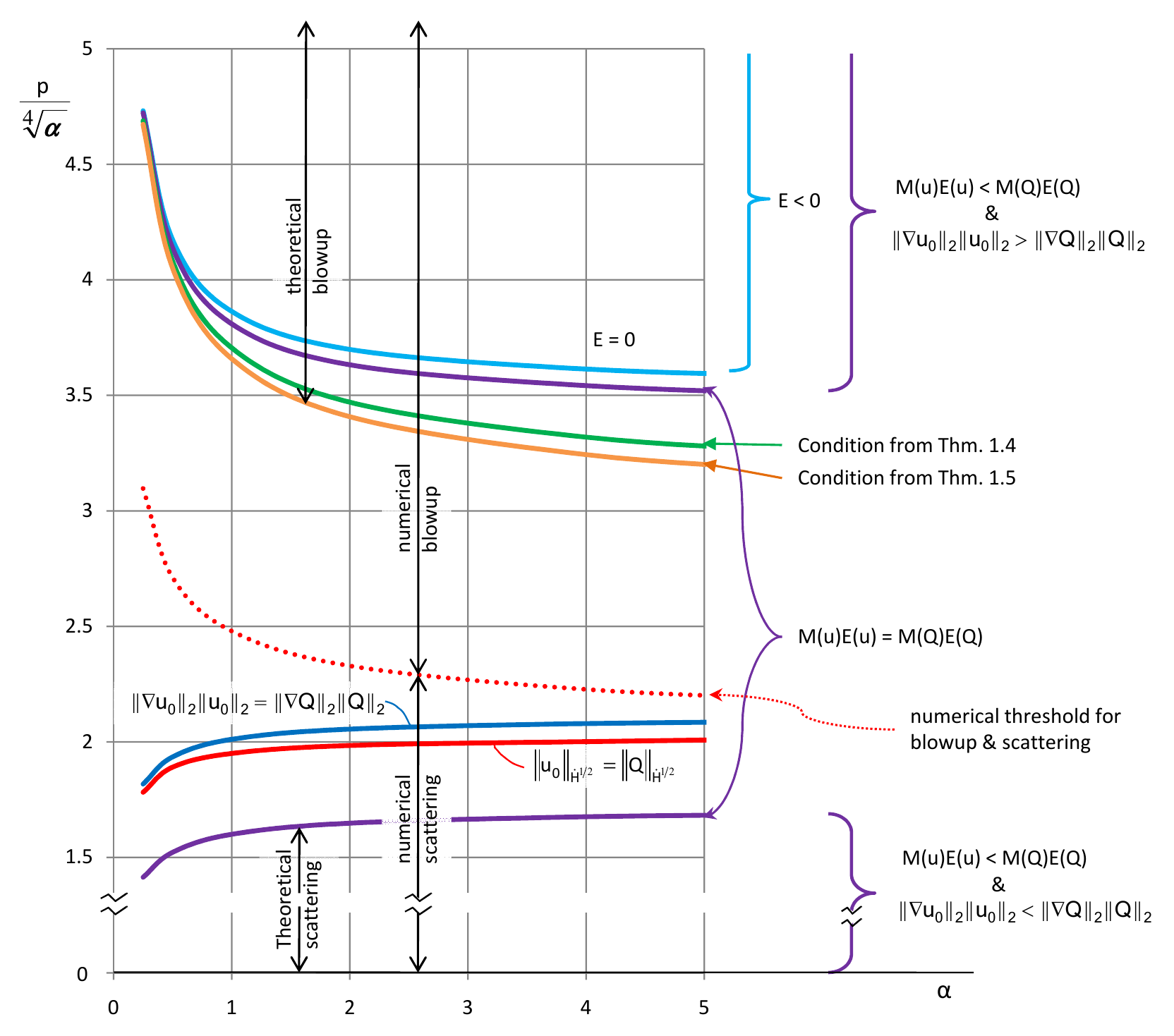}  %
\caption{Global behavior of solutions to \eqref{E:NLSa} with the
super Gaussian initial data with the positive quadratic phase
$u_0(r) = p \, e^{-\alpha r^4/2} \, e^{+ i \frac12 r^2}$. The curve
denoted by ``Condition from Thm. \ref{T:Lushnikov}" comes from the
largest positive root of the equation in \eqref{E2:Lsuperphase},
namely, values $p_t/\sqrt[4] \alpha$ in Table \ref{T2:Lsuper}.
Similarly, the curve denoted by ``Condition from Thm.
\ref{T:Lushnikov-adapted}" comes from the largest positive root of
the equation in \eqref{E2:LAsuperphase}, namely, values
$p_t/\sqrt[4] \alpha$ in Table \ref{T2:LAsuperphase}. The dotted
line ``numerical threshold" is plotted from the values
$p_s^+/\sqrt[4] \alpha$ and $p_b^+/\sqrt[4] \alpha$
(indistinguishable) from Table \ref{T2:ph+super}. The curve
``theoretical scattering" is provided by Thm. \ref{T:DHR}, see
\eqref{E2:MEphase} and values $p_1$ in Table \ref{T:1-2super}.}
 \label{F:supergaussn}
\end{figure}

\begin{figure}[p]
\includegraphics[scale=.90]{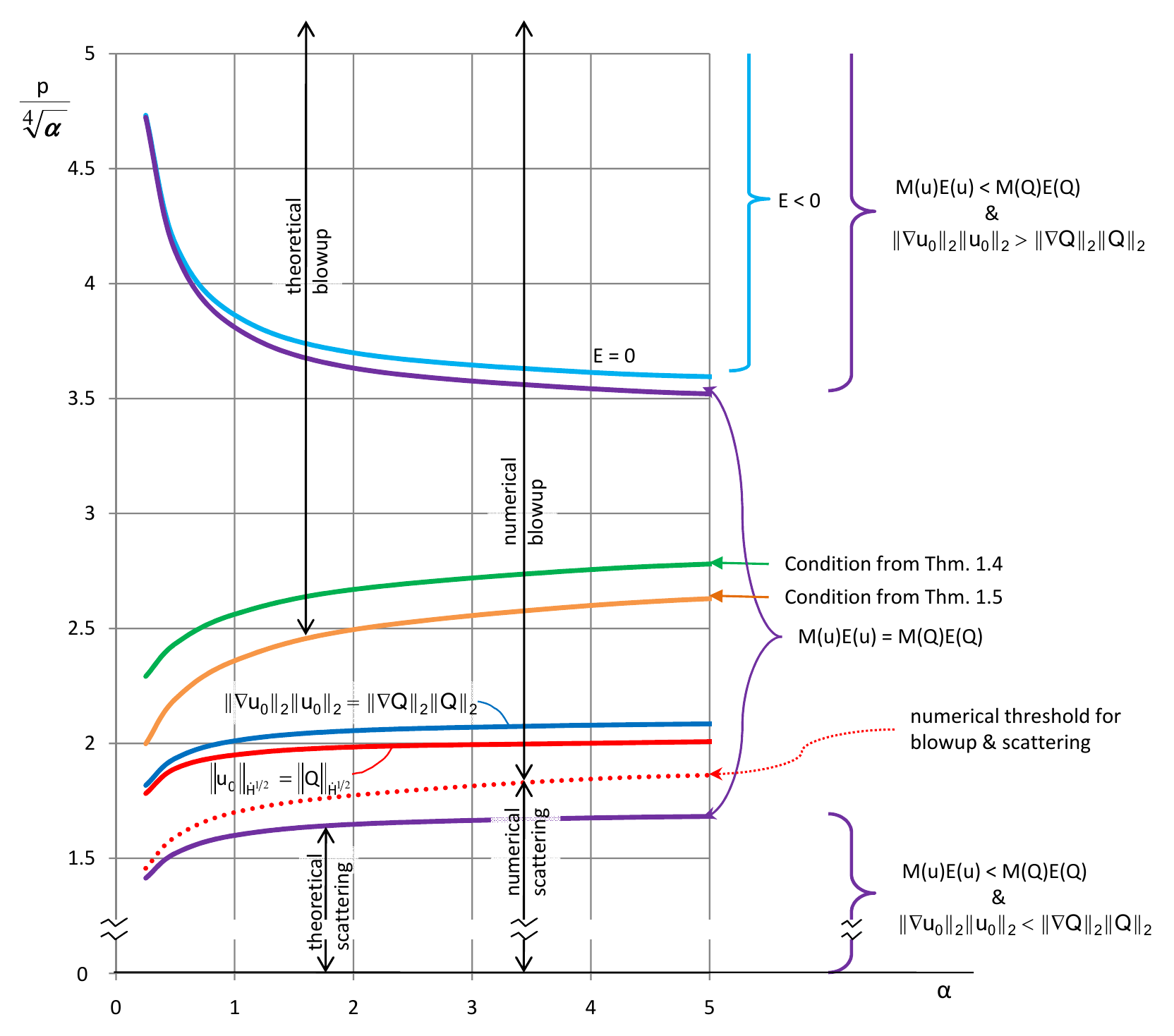}  %
\caption{Global behavior of solutions to \eqref{E:NLSa} with the
super Gaussian initial data with the negative quadratic phase
$u_0(r) = p \, e^{-\alpha r^4/2} \, e^{- i \frac12 r^2}$. The curve
denoted by ``Condition from Thm. \ref{T:Lushnikov}" comes from the
smallest positive root of the equation in \eqref{E2:Lsuperphase},
namely, values $p_b/\sqrt[4] \alpha$ in Table \ref{T2:Lsuper}.
Similarly, the curve denoted by ``Condition from Thm.
\ref{T:Lushnikov-adapted}" comes from the smallest positive root of
the equation in \eqref{E2:LAsuperphase}, namely, values
$p_b/\sqrt[4] \alpha$ in Table \ref{T2:LAsuperphase}. The dotted
line ``numerical threshold" is plotted from the values
$p_s^-/\sqrt[4] \alpha$ and $p_b^-/\sqrt[4] \alpha$
(indistinguishable) from Table \ref{T2:ph+super}. The curve
``theoretical scattering" is the same as in Fig. \ref{F:supergaussp}
and is provided by Thm. \ref{T:DHR}, see \eqref{E2:MEphase} and
values $p_1$ in Table \ref{T:1-2super}.}
 \label{F:supergaussp}
\end{figure}

\begin{remark}
Numerical simulations showed that, for example, when $\alpha=2$ for
$2.3995 \leq p \leq 6$ the blow up occurs over the origin. When $p
\geq 10$ the blow up happens on a {\it contracting sphere}. This
phenomena we originally discussed in our heuristic analysis in
\cite{HR1}, it was also numerically obtained in
\cite{F07}.
\end{remark}

\afterpage{\clearpage}

\subsection{Conclusions}
The above computations show
\begin{enumerate}
\item
Consistency with Conjecture 3: if $u_0$ is \emph{real}, then
$\|u_0\|_{\dot H^{1/2}}< \|Q\|_{\dot H^{1/2}}$ implies $u(t)$
scatters.

\item
Consistency with Conjecture 4 (since $u_0$ is based upon a radial
profile that is \emph{monotonically decreasing}): if $\|u_0\|_{\dot
H^{1/2}}> \|Q\|_{\dot H^{1/2}}$, then $u(t)$ blows-up in finite
time.

\item
The condition ``$\|u_0\|_{L^2}\|\nabla u_0\|_{L^2}<
\|Q\|_{L^2}\|\nabla Q\|_{L^2}$ {\it implies scattering}'' is not
valid (unless $M[u]E[u]<M[Q]E[Q]$ as in Theorem \ref{T:DHR}).

\item
In all three cases (real data, data with positive phase and with
negative phase) Theorems \ref{T:Lushnikov} and
\ref{T:Lushnikov-adapted} provide new range on blow up than was
previously known from our Theorem \ref{T:DHR}. Furthermore, Theorem
\ref{T:Lushnikov-adapted} provides the best results.
\end{enumerate}

\section{Off-centered Gaussian profile}
\label{S:off-Gaussian}

Next we consider the off-centered Gaussian profile:
\begin{equation}
 \label{E:offgauss}
u_0(x) = p \, r^2 \, e^{-\alpha \,r^2} \, e^{i \gamma \,r^2}, \quad
r = |x|, \, x \in \cR^3.
\end{equation}
For this initial data we calculate
\begin{align*}
& M[u] = \frac{15 \pi^{3/2} p^2}{32 \sqrt 2 \alpha^{7/2}}, \quad
\|\nabla u_0\|^2_{L^2} = \frac{3 \pi^{3/2} p^2}{32 \sqrt 2
\alpha^{5/2}} \left(11+35 \frac{\gamma^2}{\alpha^2} \right),\\
& E[u] = \frac{3 \pi^{3/2} p^2}{64 \sqrt 2 \alpha^{5/2}} \left(11+35
\frac{\gamma^2}{\alpha^2}
- \frac{315 p^2}{2^{10} \sqrt 2 \alpha^3} \right), \\
& V(0) = \frac{105 \,\pi^{3/2} p^2}{128 \sqrt 2 \,\alpha^{9/2}},
\quad
V_t(0) = \frac{105 \,\gamma \, \pi^{3/2} p^2}{16 \sqrt 2 \,\alpha^{9/2}}\\
& \|u_0\|^2_{\dot{H}^{1/2}} = 32 \pi^2 \int_0^{\infty} R \left|
\int_0^\infty r^3 \sin (2\pi R r) e^{-\alpha r^2} \, e^{i \gamma
r^2} \, dr \right|^2 \, dR = \frac{3\pi p^2}{4 \alpha^3} \frac{(1
+2\frac{\gamma^2}{\alpha^2})}{(1+ \frac{\gamma^2}{\alpha^2})^{1/2}}.
\end{align*}
\subsection{Real off-centered Gaussian}
When $\gamma=0$, we have
\begin{itemize}
\item
$E[u]>0$ if
\begin{equation}
 \label{E3:posE}
p < \frac{32}3 \,\left(\frac{11\sqrt 2}{35}\right)^{1/2} \,
\alpha^{3/2} \approx 7.11 \, \alpha^{3/2};
\end{equation}

\item
the condition on the mass and gradient $\ds \|u_{0} \|^2_{L^2}
\|\nabla u_{0} \|^2_{L^2} < \|Q\|^2_{L^2}\|\nabla Q\|^2_{L^2}$
implies
\begin{equation}
 \label{E3:MG}
p < \left(\frac{2^{11}}{165 \pi^3} \right)^{1/4} \|Q\|_2 \,
\alpha^{3/2} \approx 3.46 \, %
\alpha^{3/2};
\end{equation}

\item
the mass-energy condition $ M[u]E[u] < M[Q]E[Q]$ is
$$
\frac{45 \pi^3 p^4}{2^{12} \alpha^6} \left(11 - \frac{315
p^2}{2^{10} \sqrt 2 \alpha^3} \right) < \frac12 \|Q\|^4_{L^2},
$$
and thus, we obtain
\begin{equation}
 \label{E3:ME}
p < 2.73 \,\alpha^{3/2} \quad \text{and} \quad p > 7.04
\,\alpha^{3/2}.
\end{equation}

\item
the invariant norm condition $\Vert u_0 \Vert^2_{\dot{H}^{1/2}} <
\Vert Q \Vert^2_{\dot{H}^{1/2}}$ is
\begin{equation}
 \label{E3:H12}
p < \frac{2}{(3 \, \pi)^{1/2}} \, \Vert Q \Vert_{\dot{H}^{1/2}} \,
\alpha^{3/2} \approx 3.43 \, \alpha^{3/2}.
\end{equation}

\item
(Theorems \ref{T:Lushnikov}) the condition \eqref{E:Lreal} is %
\begin{equation}
 \label{E3:green}
p > \frac{32}{21} \left( \frac{62 \sqrt 2}{5} \right)^{1/2} \,
\alpha^{3/2} \approx 6.38 \, \alpha^{3/2}
\end{equation}

\item
(Theorems \ref{T:Lushnikov-adapted}) the condition \eqref{E:LAreal}
is
\begin{equation}
 \label{E3:orange}
p > \frac{2^5 \cdot 2^{1/4}\cdot  7^{2}\cdot  11^{1/2}}{3 \cdot
5^{1/2}(7^5 + 2^5 \cdot 3^{3/2} \cdot 5^2 \pi^{1/2})^{1/2}} \,
\alpha^{3/2} \approx 5.93 \, \alpha^{3/2},
\end{equation}

\item
Numerical simulations: the results for the off-centered Gaussian
initial data \eqref{E:offgauss} are in Table \ref{T:off-nophase}.
For $p \geq p_b$ the blow up was observed, for $p \leq p_s$ the
solution dispersed over time. We obtain that $p_s/\alpha^{3/2}
\approx 3.57$ and $p_b/\alpha^{3/2} \approx 3.58$.
\end{itemize}
\begin{table}[h]
\begin{tabular}{| c || c | c | c | c | c | c | c | c |}
\hline $\alpha$ & $.25$ & $.5$ & $1$ & $1.5$ & $2$ & $3$ & $4$ & $5$ \\
\hline
\hline $p_s$ & $0.44$ & $1.26$ & $3.57$ & $6.57$ & $10.11$ & $18.5$ & $28.4$ & $40.0$\\
\hline $p_b$ & $0.45$ & $1.27$ & $3.58$ & $6.58$ & $10.12$ & $18.6$ & $28.5$ & $40.1$\\
\hline
\end{tabular}
\bigskip
\caption{ The numerical results for blowing up ($p \geq p_b$) and
global existence ($p \leq p_s$) for the off-centered Gaussian
initial data depending on the parameter $\alpha$. }
 \label{T:off-nophase}
\end{table}

To compare all the above conditions \eqref{E3:posE} -
\eqref{E3:orange} with the numerics, we graph the dependence of $p$
on $\alpha$ in Figure \ref{F:all-gauss-nophase}. For clarity of
presentation we plot $\ds \frac{p}{\alpha^{3/2}}$ on the vertical
axis.

\begin{figure}[p]
\includegraphics[scale=.81]{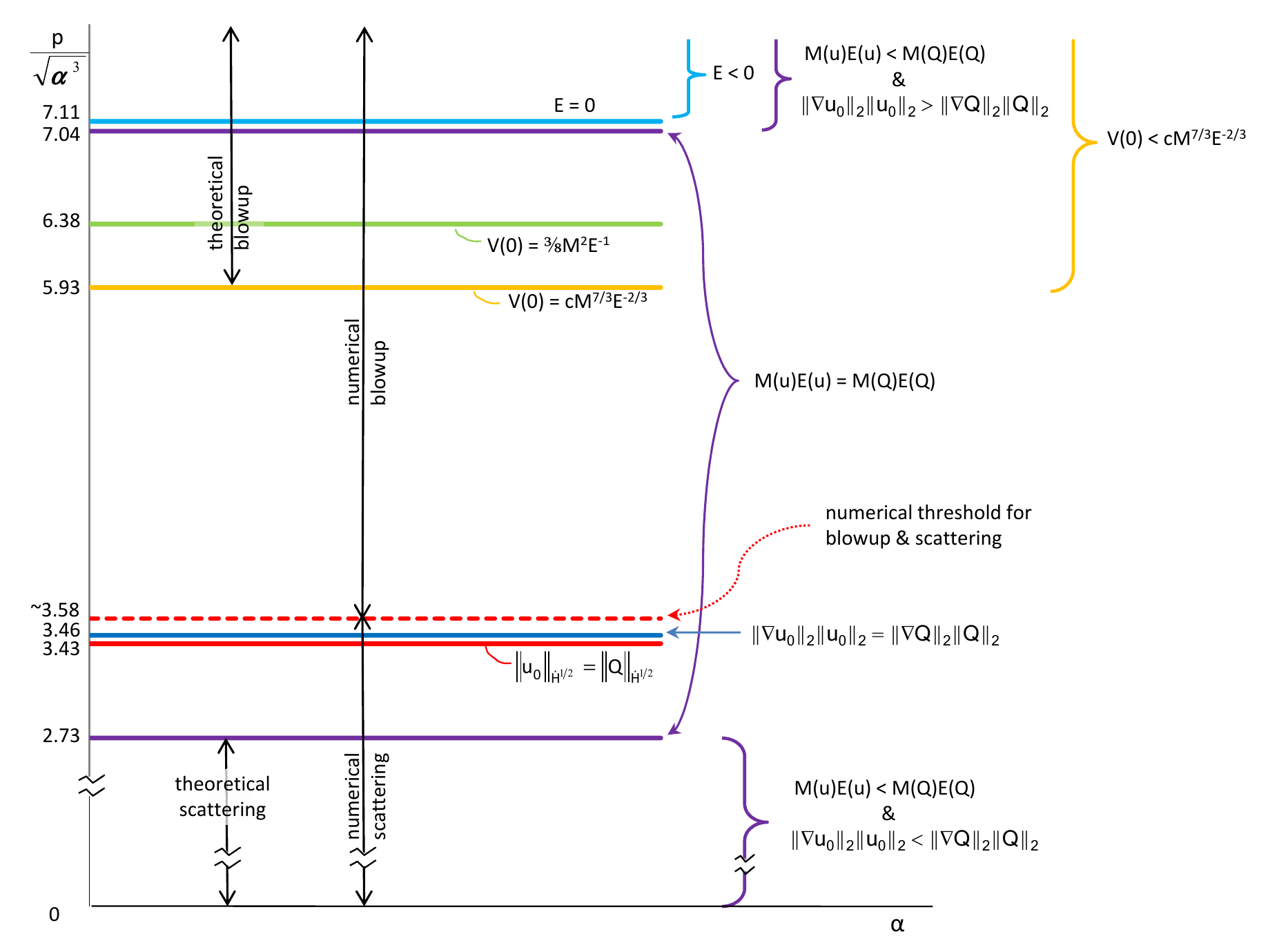}  %
\caption{Global behavior of the solutions to \eqref{E:NLSa} with the
real off-centered Gaussian initial data \eqref{E:offgauss}. The line
denoted by $V(0)=3/8 M^2 E^{-1}$ is the threshold for blow up from
Theorem \ref{T:Lushnikov}, see \eqref{E3:green}; similarly, the line
denoted by $V(0)=c M^{7/3} E^{-2/3}$ is the threshold for blow up
from Theorem \ref{T:Lushnikov-adapted}, see \eqref{E3:orange}. The
line ``theoretical scattering" is given by Thm. \ref{T:DHR}, see
\eqref{E3:ME}. Observe that the numerical threshold (dashed line,
values are given in Table \ref{T:off-nophase}) is away from the line
$\|u_0\|_{\dot{H}^{1/2}} = \|Q\|_{\dot{H}^{1/2}}$, see
\eqref{E3:H12}, which is different compared to real Gaussian and
real super Gaussian initial data, see Fig. \ref{F:all-gauss-nophase}
and \ref{F:supergauss-nophase}. }
 \label{F:all-off-gauss-nophase}
\end{figure}

\afterpage{\clearpage}

\subsection{Off-centered Gaussian with quadratic phase}
When $\gamma \neq 0$, we have
\begin{itemize}
\item
$E[u]>0$ if
\begin{equation}
 \label{E3:posEphase}
p < \frac{32}3 \left(\frac{\sqrt 2}{35}\right)^{1/2} \,\left(11 +
{35} \frac{\gamma^2}{\alpha^2}\right)^{1/2} \, \alpha^{3/2} \approx
2.14 \, \left(11 + {35} \frac{\gamma^2}{\alpha^2}\right)^{1/2} \,
\alpha^{3/2};
\end{equation}

\item
the condition on the mass and gradient $\ds \|u_{0} \|^2_{L^2}
\|\nabla u_{0} \|^2_{L^2} < \|Q\|^2_{L^2}\|\nabla Q\|^2_{L^2}$
implies
\begin{equation}
 \label{E3:MGphase}
p < \left(\frac{2^{11}}{15 \pi^3}\right)^{1/4} \|Q\|_2 {\left(11+35
\frac{\gamma^2}{\alpha^2}\right)^{-1/4}}\, {\alpha^{3/2}} \approx
6.29 \,{\left(11+35 \frac{\gamma^2}{\alpha^2}\right)^{-1/4}}
\alpha^{3/2};
\end{equation}

\item
the mass-energy condition $ M[u]E[u] < M[Q]E[Q]$ is
\begin{equation}
 \label{E3:MEphase}
\frac{45 \pi^3 p^4}{2^{12} \alpha^6} \left(11 + 35
\frac{\gamma^2}{\alpha^2}- \frac{315 p^2}{2^{10} \sqrt 2 \alpha^3}
\right) - \frac12 \, \|Q\|^4_{2} < 0,
\end{equation}
the real positive zeros of the left hand side when $\gamma = \pm
\frac12$ given in Table \ref{T3:MEphase}.
\begin{table}[h]
\begin{tabular}{| c || c | c | c | c | c | c | c | c |}
\hline $\alpha$ & $.25$ & $.5$ & $1$ & $1.5$ & $2$ & $3$ & $4$ & $5$\\ %
\hline
\hline $p_1$ & $0.17$ & $0.65$ & $2.3$ & $4.58$ & $7.31$ & $13.84$ & $21.54$ & $30.27$\\
\hline $p_2$ & $3.29$ & $5.14$ & $9.51$ & $15.14$ & $21.9$ & $38.26$ & $57.8$ & $80.05$\\
\hline
\end{tabular}
\bigskip
\caption{Real positive zeros of the function in \eqref{E3:MEphase}
when $\gamma = \pm \frac12$.}
 \label{T3:MEphase}
\end{table}

\item
the invariant norm condition $\Vert u_0 \Vert^2_{\dot{H}^{1/2}} <
\Vert Q \Vert^2_{\dot{H}^{1/2}}$ is
\begin{equation}
 \label{E3:H12phase}
p < \frac{2}{(3 \, \pi)^{1/2}} \, \Vert Q \Vert_{\dot{H}^{1/2}}
\frac{(1+\frac{\gamma^2}{\alpha^2})^{1/4}}{(1+2
\frac{\gamma^2}{\alpha^2})^{1/2}}\, \alpha^{3/2} \approx 3.43
\frac{(1+\frac{\gamma^2}{\alpha^2})^{1/4}}{(1+2
\frac{\gamma^2}{\alpha^2})^{1/2}}\, \alpha^{3/2}.
\end{equation}

\item
(Theorems \ref{T:Lushnikov}) the condition \eqref{E:Lreal} is
\begin{equation}
 \label{E3:LAgreenphase}
p > \frac{32 \cdot 2^{1/4}}{3} \left( \frac{62}{5\cdot7^2} +
\frac{\gamma^2}{\alpha^2} \right)^{1/2} \, \alpha^{3/2} \approx
12.68 \,  \left( 0.253 + \frac{\gamma^2}{\alpha^2} \right)\,
\alpha^{3/2}
\end{equation}
and the condition \eqref{E:L+simple} with $y=p/\alpha^{3/2}$ is
\begin{equation}
 \label{E3:Lsimple}
\frac{2^6 \sqrt{105}}{7 \sqrt{2^{11}(11 + \frac{35 \gamma^2
}{\alpha^2})-315 \sqrt 2 y^2}} -\frac{3\cdot7^2}{2^{11}} y^2
+\frac{32}{15 \sqrt 2} > 0,
\end{equation}
the positive real roots of which when $\gamma = \pm \frac12$ are in
Table \ref{T3:Lsimple}.
\begin{table}[h]
\begin{tabular}{| c || c | c | c | c | c | c | c | c |}
\hline $\alpha$ & $.25$ & $.5$ & $1$ & $1.5$ & $2$ & $3$ & $4$ & $5$\\ %
\hline
\hline $p_b$ & $0.60$ & $1.79$ & $5.37$ & $10.25$ & $16.18$ & $30.63$ & $47.97$ & $67.80$\\ %
\hline $p_t$ & $3.29$ & $5.14$ & $9.49$ & $14.99$ & $21.50$ & $37.00$ & $55.27$ & $75.91$\\ %
\hline
\end{tabular}
\bigskip
\caption{The positive real roots of the function in
\eqref{E3:Lsimple} when $\gamma = \pm\frac12$ for the off-centered
Gaussian initial data. }
 \label{T3:Lsimple}
\end{table}

\item
(Theorems \ref{T:Lushnikov-adapted}) the condition \eqref{E:LAreal}
is
\begin{equation}
 \label{E3:LAphase}
\begin{aligned}
p &> \frac{2^5 \cdot 2^{1/4} \cdot  7^{2}}{3 \cdot 5^{1/2}(7^5 + 2^5
\cdot 3^{3/2} \cdot 5^2 \pi^{1/2})^{1/2}} \left(11+35
\frac{\gamma^2}{\alpha^2} \right)^{1/2} \, \alpha^{3/2} \\
&\approx 1.79
\left(11+35 \frac{\gamma^2}{\alpha^2} \right)^{1/2} \alpha^{3/2}
\end{aligned}
\end{equation}
and the condition \eqref{E:LA+simple} with $y = p/\alpha^{3/2}$ is
\begin{equation}
 \label{E3:LAorangephase}
\begin{aligned}
\left(\frac{(2^6 \cdot 3^{3/2} \cdot 5^{2} \pi^{1/2}-7^5)}{2^{5/2}
\cdot 105}y^2+ \frac{2^8\cdot 7^3\cdot 11}{3^3\cdot 5^2} \right)^3 &
\\ -84 \pi \left( 2^{11} (11+35\frac{\gamma^2}{\alpha^2}) - 315 \sqrt 2
y^2 \right) y^4 & > 0 \,,
\end{aligned}
\end{equation}
the positive real zeros of which with $\gamma=\pm\frac12$ is in
Table \ref{T3:LAsimple}.
\begin{table}[h]
\begin{tabular}{| c || c | c | c | c | c | c | c | c |}
\hline $\alpha$ & $.25$ & $.5$ & $1$ & $1.5$ & $2$ & $3$ & $4$ & $5$\\ %
\hline
\hline $p_b$ & $0.30$ & $1.16$ & $4.15$ & $8.46$ & $13.79$ & $26.94$ & $42.92$ & $61.04$\\ %
\hline $p_t$ & $3.29$ & $5.14$ & $9.45$ & $14.80$ & $21.05$ & $35.75$ & $52.97$ & $72.38$\\ %
\hline
\end{tabular}
\bigskip
\caption{The positive real roots of the function in
\eqref{E3:LAorangephase} when $\gamma = \pm\frac12$ for the
off-centered Gaussian initial data. }
 \label{T3:LAsimple}
\end{table}

\item
Numerical simulations: the results for the off-centered Gaussian
initial data with nonzero phase \eqref{E:offgauss} are in Table
\ref{T3:num-phase}. For $p \geq p_b$ the blow up was observed, for
$p \leq p_s$ the solution dispersed over time.
\end{itemize}

\begin{table}[h]
\begin{tabular}{| c || c | c | c | c | c | c | c | c |}
\hline $\alpha$ & $.25$ & $.5$ & $1$ & $1.5$ & $2$ & $3$ & $4$ & $5$ \\
\hline
\hline $p_s^+$ & n/a & n/a & $5.78$ & $8.93$ & $12.68$ & $21.5$ & $31.9$ & $43.6$\\
\hline $p_b^+$ & n/a & n/a & $5.79$ & $8.94$ & $12.69$ & $21.6$ & $32.0$ & $43.7$\\
\hline
\hline $p_s^-$ & $.18$ & $.69$ & $2.50$ & $5.11$ & $8.337$ & $16.2$ & $25.8$ & $36.8$\\
\hline $p_b^-$ & $.19$ & $.70$ & $2.51$ & $5.12$ & $8.338$ & $16.3$ & $25.9$ & $36.9$\\
\hline
\end{tabular}
\bigskip
\caption{The numerical results for blow up and scattering thresholds
for the off-centered Gaussian initial data with the phase $\gamma =
\pm\frac12$: blow up if $p \geq p_b^\pm$ and global
existence/scattering if $p \leq p_s^\pm$.}
 \label{T3:num-phase}
\end{table}

To compare all the above conditions \eqref{E3:posEphase} -
\eqref{E3:LAorangephase} with the numerical data, we graph the
dependence of $p$ on $\alpha$ in Figures \ref{F:off-gaussp} and
\ref{F:off-gaussn}. For clarity of presentation we plot $\ds
\frac{p}{\alpha^{3/2}}$ on the vertical axis.

\begin{figure}[p]
\includegraphics[scale=.78]{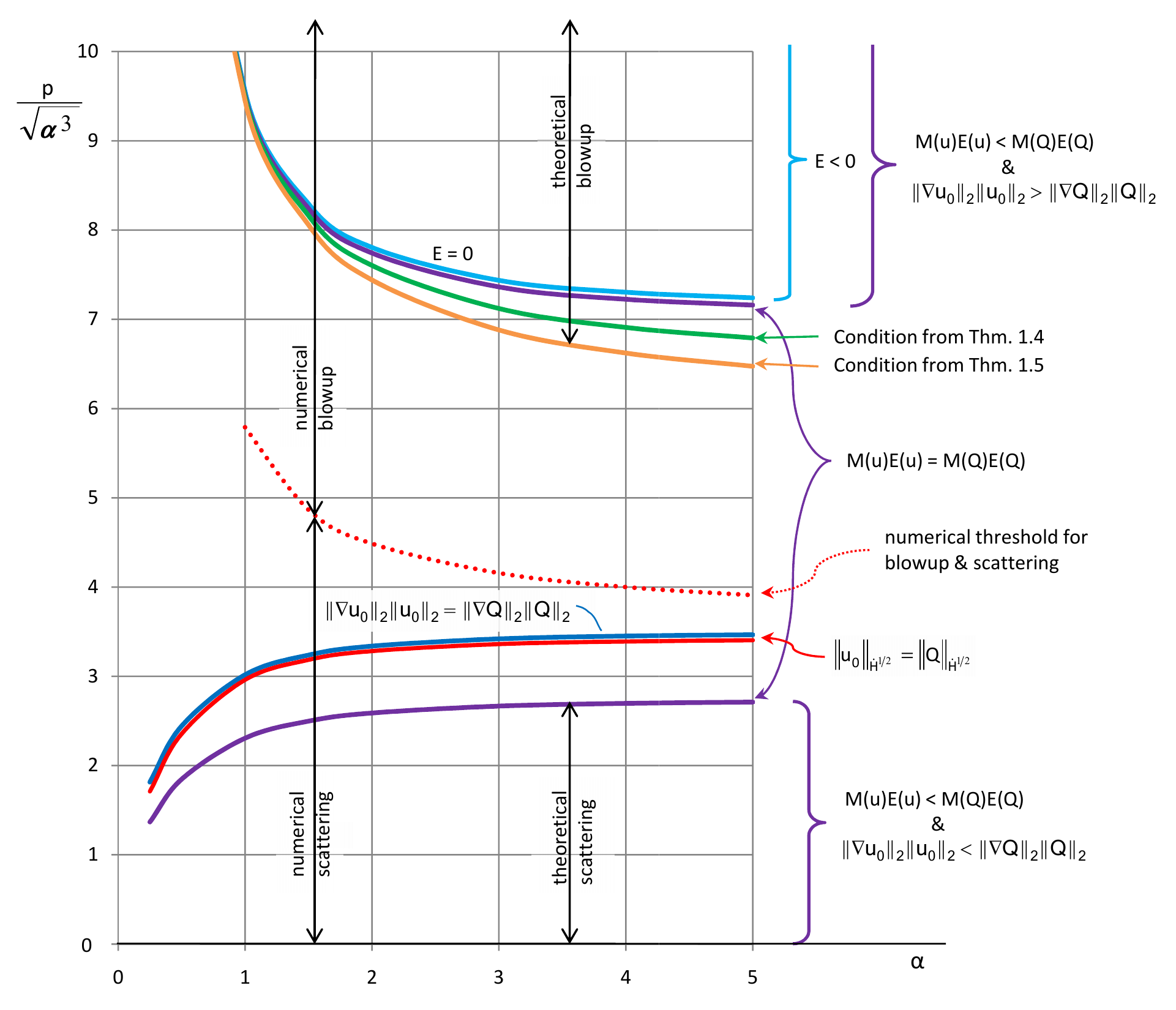}  %
\caption{Global behavior of the solutions for the off-centered
Gaussian initial data with positive quadratic phase. The curve
denoted by ``Condition from Thm. \ref{T:Lushnikov}" comes from the
largest positive root of the equation in \eqref{E3:Lsimple}, namely,
values $p_t/\sqrt[3] \alpha^2$ in Table \ref{T3:Lsimple}. Similarly,
the curve denoted by ``Condition from Thm.
\ref{T:Lushnikov-adapted}" comes from the largest positive root of
the equation in \eqref{E3:LAorangephase}, namely, values
$p_t/\sqrt[3] \alpha^2$ in Table \ref{T3:LAsimple}. The dotted line
``numerical threshold" is plotted from the values $p_s^+/\sqrt[3]
\alpha^2$ and $p_b^+/\sqrt[3] \alpha^2$ (indistinguishable) from
Table \ref{T3:num-phase}. The curve ``theoretical scattering" is
provided by Thm. \ref{T:DHR}, see \eqref{E3:MEphase} and Table
\ref{T3:MEphase}.}
 \label{F:off-gaussp}
\end{figure}

\begin{figure}[p]
\includegraphics[scale=.78]{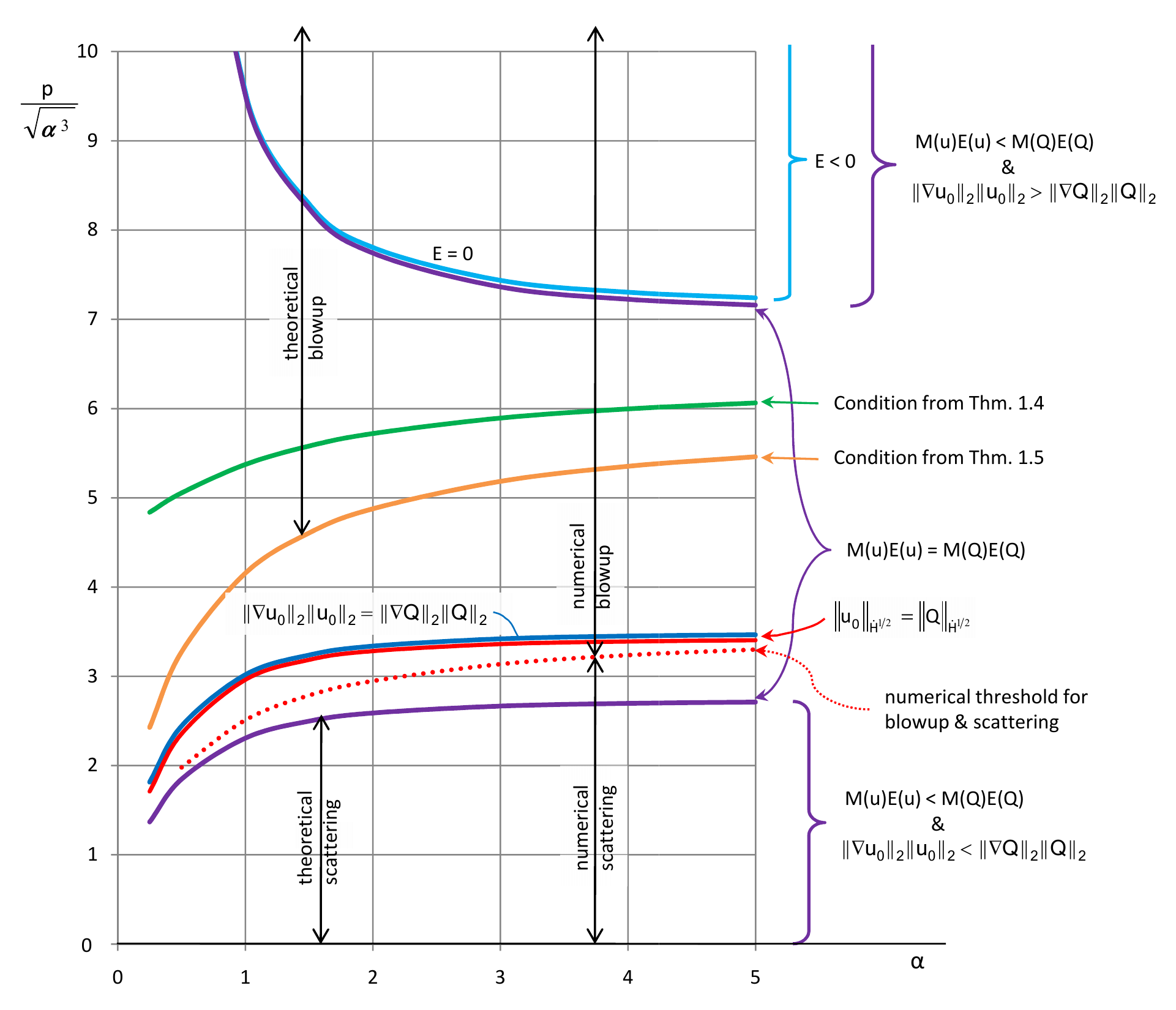}  %
\caption{Global behavior of the solutions for the off-centered
Gaussian initial data with negative quadratic phase. The curve
denoted by ``Condition from Thm. \ref{T:Lushnikov}" comes from the
smallest positive root of the equation in \eqref{E3:Lsimple},
namely, values $p_b/\sqrt[3] \alpha^2$ in Table \ref{T3:Lsimple}.
Similarly, the curve denoted by ``Condition from Thm.
\ref{T:Lushnikov-adapted}" comes from the smallest positive root of
the equation in \eqref{E3:LAorangephase}, namely, values
$p_b/\sqrt[3] \alpha^2$ in Table \ref{T3:LAsimple}. The dotted line
``numerical threshold" is plotted from the values $p_s^-/\sqrt[3]
\alpha^2$ and $p_b^-/\sqrt[3] \alpha^2$ (indistinguishable) from
Table \ref{T3:num-phase}. The curve ``theoretical scattering" is
provided by Thm. \ref{T:DHR}, see \eqref{E3:MEphase} and Table
\ref{T3:MEphase}.}
 \label{F:off-gaussn}
\end{figure}

\afterpage{\clearpage}

\subsection{Conclusions}
The above computations show
\begin{enumerate}
\item
Consistency with Conjecture 3: if $u_0$ is \emph{real}, then
$\|u_0\|_{\dot H^{1/2}}< \|Q\|_{\dot H^{1/2}}$ implies $u(t)$
scatters.

\item
Consistency with Conjecture 4:  Numerical simulations show
that the blow up threshold line is higher than the line
$\|u_0\|_{\dot H^{1/2}}= \|Q\|_{\dot H^{1/2}}$, although this profile is
not monotonic.

\item
The condition ``$\|u_0\|_{L^2}\|\nabla u_0\|_{L^2}<
\|Q\|_{L^2}\|\nabla Q\|_{L^2}$ {\it implies scattering}'' is not
valid (unless $M[u]E[u]<M[Q]E[Q]$ as in Theorem \ref{T:DHR}), for
example, when there is a negative initial phase present.

\item
In all three cases (real data, data with positive phase and with
negative phase) Theorems \ref{T:Lushnikov} and
\ref{T:Lushnikov-adapted} provide new range on blow up than was
previously known from our Theorem \ref{T:DHR}. Furthermore, Theorem
\ref{T:Lushnikov-adapted} provides the best results.
\end{enumerate}

\section{Oscillatory Gaussian profile}
\label{S:osc-Gaussian}

Lastly we consider a Gaussian profile with oscillations (referred
from now on as `oscillatory Gaussian'), it is a sign changing
profile:
\begin{equation}
 \label{E:osc-gaussian}
u_{\beta,0}(x) = p \, \cos (\beta r) \, e^{-r^2} \, e^{i \gamma
r^2}.
\end{equation}
Here, we fix the Gaussian $e^{-r^2}$ itself and change the frequency
of oscillation $\beta$. These data generate a solution to NLS
denoted by $u_\beta$. We obtain:
\begin{align*}
M[u_\beta] & = \frac{\pi^{3/2} p^2}{4 \sqrt 2} \, m(\beta), \quad
\|u_{\beta,0}\|^2_{{L}^{2}}
= \frac{\pi^{3/2}p^2}{4 \sqrt 2} \, a(\beta, \gamma),\\
E[u_\beta] & =\frac{\pi^{3/2} p^2}{8 \sqrt 2} \left(a(\beta,
\gamma)-p^2\, b(\beta)\right),\\
V(0) & = \frac{\pi^{3/2} p^2}{16 \sqrt 2} \, v(\beta),
\quad V_t(0) = \gamma \, \frac{\pi^{3/2} p^2}{2 \sqrt 2} \, v(\beta),\\
\|u_{\beta,0}\|^2_{\dot{H}^{1/2}} & = 32 \pi^2 \, p^2 \,
\int_0^\infty R \left|\int_0^{\infty} r \, \cos(\beta r) \,
\sin(2\pi R r) \,
e^{-r^2} \, dr \right|^2 \, dR\\
& = \frac{\pi}{8} \, p^2 \, \left( \sqrt{2\pi} \, \beta \,
\mbox{erf}\left(\frac{\beta}{\sqrt 2}\right) -2
\sqrt{\pi}(\beta^2-4)e^{-\beta^2/2} \right),
\end{align*}
where
\begin{align*}
m(\beta) &= 1+ (1-\beta^2)e^{-\beta^2/2},\\
a(\beta, \gamma) &= 3(1+\gamma^2)+\beta^2+\left(3(1
+\gamma^2)-\beta^2(1+6\gamma^2) +\beta^4\gamma^2\right) e^{-\beta^2/2},\\
b(\beta) &= \frac{1}{16 \sqrt 2}\left(3 +(1-2\beta^2)e^{-\beta^2}
+ 2(2 - \beta^2) e^{-\beta^2/4}\right),\\
v(\beta) & = 3+(3-6\beta^2+\beta^4)e^{-\beta^2/2}, \\
\mbox{erf}(x) & = \frac2{\sqrt \pi} \int_0^x e^{-t^2} \, dt .
\end{align*}
We list some values of the $\dot{H}^{1/2}$ norm of $u_{\beta, 0}$ in
Table \ref{T4:H12}.
\begin{table}[h]
\begin{align*}
&\begin{tabular} [c]{|l||l|l|l|l|l|l|l|l|l|l|l|l|} \hline $\beta$ &
$0$ & $0.25$ & $0.5$ & $1.0$ & $2.0$ & $3.0$ & $4.0$ & $5.0$ & $6.0$
\\
\hline\hline $\frac12 \|u_{\beta, 0}\|^2_{\dot{H}^{1/2}}$ & $\pi$ &
$3.046$ & $2.788$ & $2.101$ & $1.879$ & $2.901$ & $3.933$
& $4.922$ & $5.906$  \\
\hline $p_{1/2}$ & $2.97$ & $3.02$ & $3.15$ & $3.63$ & $3.84$ &
$3.09$ & $2.65$ & $2.37$ & $2.17$\\
\hline
\end{tabular}\\
&\begin{tabular} [c]{|l||l|l|l|l|l|l|l|l|l|}\hline $\beta$ & $7.0$ &
$8.0$ & $9.0$ & $10.0$& $15.0$ & $20.0$ & $25.0$\\
\hline\hline $\frac12 \|u_{\beta, 0}\|^2_{\dot{H}^{1/2}}$ & $6.890$
& $7.875$ & $8.859$ & $9.844$ & $14.765$ & $19.687$ & $24.609$\\
\hline $p_{1/2}$ & $2.01$ & $1.88$ & $1.78$ & $1.68$ & $1.37$ &
$1.19$ & $1.06$\\
\hline
\end{tabular}
\end{align*}
\caption{The $\dot{H}^{1/2}$ norm of $u_{\beta,0}$ and values of $p$
for which $\ds \|u_{\beta, 0}\|^2_{\dot{H}^{1/2}} = \| Q
\|^2_{\dot{H}^{1/2}}$. }
 \label{T4:H12}
\end{table}

\subsection{Real oscillatory Gaussian}
First, we consider the real oscillatory Gaussian data ($\gamma =
0$). We have
\begin{itemize}
\item
$E[u]>0$ if
\begin{equation}
 \label{E4:posE}
p < \left(\frac{a(\beta, 0)}{b(\beta)} \right)^{1/2};
\end{equation}

\item
the condition on the mass and gradient $\ds \|u_{\beta} \|^2_{L^2}
\|\nabla u_{\beta,0} \|^2_{L^2} < \|Q\|^2_{L^2}\|\nabla Q\|^2_{L^2}$
implies
\begin{equation}
 \label{E4:MG}
p < 2 \,\|Q\|_2 \, \left( \frac{6 }{\pi^{3} \, m(\beta) \, a(\beta,
0)} \right)^{1/4};
\end{equation}

\item
the mass-energy condition $ M[u]E[u] < M[Q]E[Q]$ gives
\begin{equation}
 \label{E4:ME}
{\pi^3 m(\beta) } \left(a(\beta, 0)-p^2 \,b(\beta)\right) p^4 -
32\|Q\|_2^4 < 0, \quad \mbox{or} \quad p < p_1 \quad \text{and}
\quad p > p_2,
\end{equation}
where $p_1$ and $p_2$, the real positive zeros of the cubic
polynomial (in $p^2$) above, given in Table \ref{T4:ME}.
\begin{table}[h]
\begin{align*}
&\begin{tabular} [c]{|l||l|l|l|l|l|l|l|l|l|l|l|} \hline $\beta$ &
$0$ & $0.25$ &
$0.5$ & $1.0$ & $1.45$ & $2.0$ & $2.5$ & $2.77$ & $3.0$ & $3.61$ & $4.0$\\
\hline\hline $p_1$ & $2.72$ & $2.76$ & $2.88$ & $3.27$ & $3.48$
& $3.16$ & $2.72$ & $2.55$ & $2.43$ & $2.22$ & $2.12$\\
\hline $p_2$ & $3.81$ & $3.87$ & $4.06$ & $5.00$ & $6.84$ &
$10.49$ & $13.11$ & $13.42$ & $13.31$ & $12.97$ & $13.14$\\
\hline
\end{tabular}\\
&\begin{tabular} [c]{|l||l|l|l|l|l|l|l|l|l|}\hline $\beta$  & $5.0$
& $6.0$ & $7.0$ &
$8.0$ & $9.0$ & $10.0$ & $15.0$ & $20.0$ & $25.0$\\
\hline\hline $p_1$  & $1.91$ & $1.76$ & $1.64$ & $1.53$
& $1.45$ & $1.38$ & $1.13$ & $0.98$ & $0.88$\\
\hline $p_2$  & $14.75$ & $17.17$ & $19.81$ & $22.48$
& $25.17$ & $27.87$ & $41.47$ & $55.13$ & $68.82$\\
\hline
\end{tabular}
\end{align*}
\caption{The values of $p$ in the mass-energy threshold for the real
oscillatory Gaussian.}
 \label{T4:ME}
\end{table}

\item
the values of $p$, denoted by $p_{1/2}$, for which
$\|u_{\beta, 0}\|^2_{\dot{H}^{1/2}} = \| Q \|^2_{\dot{H}^{1/2}}$,
are given in Table \ref{T4:H12}.

\item
(Theorem \ref{T:Lushnikov}) the condition \eqref{E:Lreal} is
\begin{equation}
 \label{E4:green}
p \geq {b(\beta)}^{-1/2}\left(a(\beta,0) - \frac{3 m(\beta)^2}{
v(\beta)}\right)^{-1/2} \, ;
\end{equation}

\item
(Theorem \ref{T:Lushnikov-adapted}) the condition \eqref{E:LAreal}
is
\begin{equation}
 \label{E4:orange}
p > 2 \sqrt 2 \, C^{7/2} \frac{a(\beta,0)^{1/2}
v(\beta)^{3/4}}{(\pi^{3/2} m(\beta)^{7/2} + 8 C^7 \,v(\beta)^{3/2}\,
b(\beta) )^{1/2}} \,;
\end{equation}

\item
Numerical simulations: the results for the (real) oscillatory
Gaussian initial data are in Table \ref{T4:num}. For $p \geq p_b$
the blow up was observed, for $p \leq p_s$ the solution dispersed
over time.
\end{itemize}
\begin{table}[h]
\begin{align*}
&\begin{tabular} [c]{|l||l|l|l|l|l|l|l|l|l|} \hline $\beta$
& $0$ & $0.25$ & $0.5$ & $1.0$ & $2.0$ & $3.0$ & $4.0$ & $5.0$ & $6.0$\\
\hline
\hline $p_s$ & $2.93$ & $2.98$ & $3.11$ & $3.601$ & $4.583$ & $4.13$ & $3.29$ & $2.92$ & $2.72$ \\
\hline $p_b$ & $2.94$ & $2.99$ & $3.12$ & $3.605$ & $4.585$ & $4.14$ & $3.30$ & $2.93$ & $2.73$ \\
\hline
\end{tabular}\\
&\begin{tabular} [c]{|l||l|l|l|l|l|l|l|}\hline $\beta$ & $7.0$ &
$8.0$ & $9.0$ & $10.0$& $15.0$ & $20.0$ & $25.0$\\
\hline
\hline $p_s$ & $2.61$ & $2.53$ & $2.48$ & $2.455$ & $2.42$ & $2.49$ & $2.52$\\
\hline $p_b$ & $2.62$ & $2.54$ & $2.49$ & $2.456$ & $2.43$ & $2.50$ & $2.53$\\
\hline
\end{tabular}
\end{align*}
\caption{Numerical simulations for the (real) oscillatory Gaussian.
Here, the blow up was observed for $p > p_b$ and global existence
for $p < p_s$.}
 \label{T4:num}
\end{table}

To compare all conditions \eqref{E4:posE} - \eqref{E4:orange} with
numerical data, we graph the dependence of $p$ on $\beta$ in Figure
\ref{F:oscgauss}. We plot $p/\sqrt {1+\beta^2}$ on the vertical
axis.
\begin{figure}[p]
\includegraphics[scale=.76]{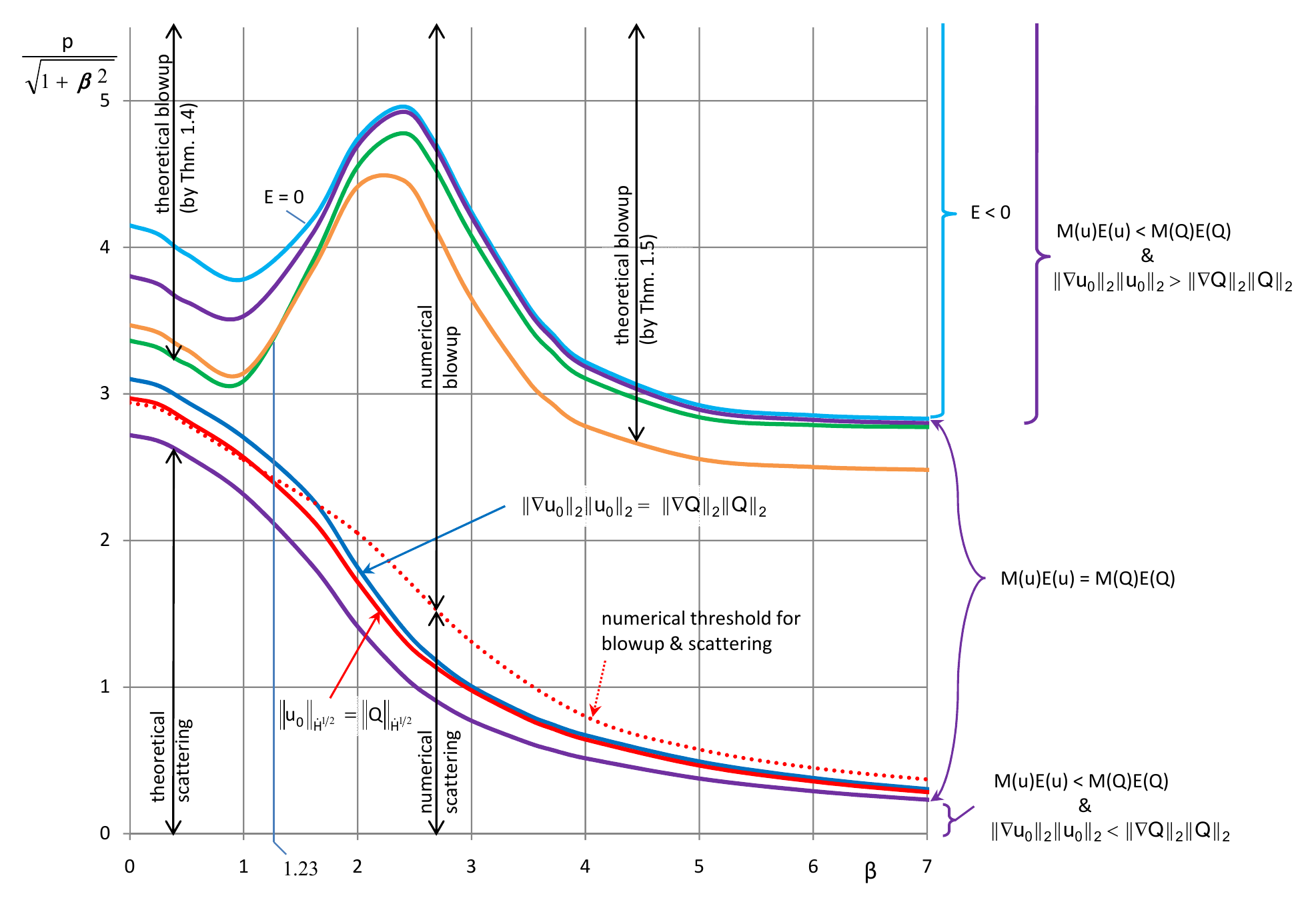}  %
\caption{Global behavior of the solutions to \eqref{E:NLSa} with the
real oscillatory Gaussian initial data \ref{E:osc-gaussian}, with
the rescaled vertical axis $p/\sqrt{1+\beta^2}$. The blow up
threshold curve denoted ``by Thm. \ref{T:Lushnikov}" is given by
\eqref{E4:green} and the blow up threshold curve denoted ``by Thm.
\ref{T:Lushnikov-adapted}" is given by \eqref{E4:orange}. Observe
that these curves intersect at $\beta \approx 1.23$. Therefore, for
small oscillations ($\beta < 1.23$) Theorem \ref{T:Lushnikov}
provides the best range for blow up, correspondingly, for large
oscillations ($\beta > 1.23$) Theorem \ref{T:Lushnikov-adapted}
gives a better range. The curve ``theoretical scattering" is given
by Thm. \ref{T:DHR}, see \eqref{E4:ME} and values $p_1$ in Table
\ref{T4:ME}. Observe that the numerical threshold (dotted curve,
values are given in Table \ref{T4:num}) for small oscillations
coincide with the curve $\|u_0\|_{\dot{H}^{1/2}} =
\|Q\|_{\dot{H}^{1/2}}$, see Table \ref{T4:H12}, and for large
oscillations these curves separate.}
 \label{F:oscgauss}
\end{figure}

\subsection{Oscillatory Gaussian with quadratic phase}
When $\gamma \neq 0$ we have
\begin{itemize}
\item
$E[u]>0$ if
\begin{equation}
 \label{E4:posEphase}
p < \left(\frac{a(\beta, \gamma)}{b(\beta)} \right)^{1/2};
\end{equation}

\item
the condition on the mass and gradient $\ds \|u_{\beta} \|^2_{L^2}
\|\nabla u_{\beta,0} \|^2_{L^2} < \|Q\|^2_{L^2}\|\nabla Q\|^2_{L^2}$
implies
\begin{equation}
 \label{E4:MGphase}
p < 2 \,\|Q\|_2 \, \left( \frac{6 }{\pi^{3} \, m(\beta) \, a(\beta,
\gamma)} \right)^{1/4};
\end{equation}

\item
the mass-energy condition $ M[u]E[u] < M[Q]E[Q]$ gives
\begin{equation}
 \label{E4:MEphase}
{\pi^3 m(\beta)} \left(a(\beta, \gamma)-p^2 b(\beta)\right) p^4 -
32\|Q\|_2^4 < 0 \quad \mbox{or} \quad p < p_1^{\gamma} \quad
\text{and} \quad p
> p_2^\gamma,
\end{equation}
where $p_1^\gamma$ and $p_2^\gamma$, the real positive zeros of the
polynomial above with $\gamma = \pm\frac12$, given in Table
\ref{T4:MEphase}.
\begin{table}[h]
\begin{align*}
&\begin{tabular} [c]{|l||l|l|l|l|l|l|l|l|l|l|l|}
\hline $\beta$ & $0$ & $0.25$ & $0.5$ & $1.0$ & $1.6$ & $2.0$ & $2.4$ & $2.66$ & $3.0$ & $4.0$\\
\hline \hline $p_1^\gamma$ & $2.42$ & $2.47$ & $2.62$ & $3.13$ &
$3.39$ & $3.09$ & $2.74$ & $2.56$ & $2.39$ & $2.10$\\
\hline $p_2^\gamma$ & $4.46$ & $4.50$ & $4.61$ & $5.29$ & $7.97$ &
$10.94$ & $13.37$ & $13.92$ & $13. 77$ & $13.40$\\
\hline
\end{tabular}\\
&\begin{tabular} [c]{|l||l|l|l|l|l|l|l|l|l|}\hline $\beta$  & $5.0$
& $6.0$ & $7.0$ & $8.0$ & $9.0$ & $10.0$ & $15.0$ & $20.0$ & $25.0$\\
\hline \hline $p_1^\gamma$ & $1.90$ & $1.75$ & $1.63$ & $1.53$ &
$1.45$ & $1.38$ & $1.13$ & $0.98$ & $0.88$\\
\hline $p_2^\gamma$ & $14.95$ & $17.34$ & $19.95$ & $22.61$ &
$25.28$ & $27.97$ & $41.54$ & $55.18$ & $68.86$\\
\hline
\end{tabular}
\end{align*}
\caption{The values of $p$ for the mass-energy threshold for the
oscillatory Gaussian with the phase $\gamma = \pm \frac12$.}
 \label{T4:MEphase}
\end{table}

\item
the values of $p$, denoted by $p_{1/2}^\gamma$, for which
$\|u_{\beta, 0}\|^2_{\dot{H}^{1/2}} = \| Q \|^2_{\dot{H}^{1/2}}$,
are given in Table \ref{T4:H12-phase}.
\begin{table}[h]
\begin{align*}
&\begin{tabular} [c]{|l||l|l|l|l|l|l|l|l|l|l|l|l|} \hline $\beta$
&$0$ & $0.25$ & $0.5$ & $1.0$ & $2.0$ & $3.0$ & $4.0$ & $5.0$ &
$6.0$\\
\hline\hline $\frac12 \|u_{0,\beta}\|^2_{\dot{H}^{1/2}}$ &
$\frac{\sqrt 5}2\pi$ & $3.384$ & $3.042$ & $2.165$ & $1.933$ &
$2.959$ & $3.942$ & $4.922$ & $5.906$  \\
\hline $p_{1/2}^\gamma$ & $2.81$ & $2.86$ & $3.02$ & $3.58$
& $3.79$ & $3.06$ & $2.65$ & $2.37$ & $2.17$\\
\hline
\end{tabular}\\
&\begin{tabular} [c]{|l||l|l|l|l|l|l|l|l|l|}\hline $\beta$ & $7.0$ &
$8.0$ & $9.0$ & $10.0$& $15.0$ & $20.0$ & $25.0$\\
\hline\hline $\frac12 \|u_{0,\beta}\|^2_{\dot{H}^{1/2}}$ & $6.890$ &
$7.875$ & $8.859$ & $9.844$ & $14.765$ & $19.687$ & $24.609$\\
\hline $p_{1/2}^\gamma$ & $2.01$ & $1.88$ & $1.78$ & $1.68$ & $1.37$
& $1.19$ & $1.06$\\
\hline
\end{tabular}
\end{align*}
\caption{The values of $\|u_{\beta, 0}\|^2_{\dot{H}^{1/2}} / p^2$
for the oscillatory Gaussian initial data with phase $\gamma =
\pm\frac12$ and values of $p^\gamma_{1/2}$ for which the
$\dot{H}^{1/2}$ norm threshold holds.}
  \label{T4:H12-phase}
\end{table}

\item
(Theorem \ref{T:Lushnikov}) the condition \eqref{E:Lreal} is
\begin{equation}
 \label{E4:Lphase}
p \geq {b(\beta)}^{-1/2}\left(a(\beta,\gamma) - \frac{3
m(\beta)^2}{v(\beta)}\right)^{-1/2},
\end{equation}
and the condition \eqref{E:L+simple} is
\begin{equation}
 \label{E4:Lgreen}
2\sqrt 3 m(\beta) + \left(\frac{v(\beta)}{3 m(\beta)^2}
\left(a(\beta, 0)-p^2\, b(\beta) \right) -3 \right) \left(a(\beta,
\gamma)-p^2 b(\beta) \right)^{1/2} v(\beta)^{1/2}
> 0,
\end{equation}
the real positive zeros of which are in Table \ref{T4:L-phase}.
\begin{table}[h]
\begin{tabular}[c]
{|l|l|l|l|l|l|l|l|l|l|l|}
\hline $\beta$ & $0$ & $ 0.25 $ & $ 0.5$ & $1$ & $1.6$ & $2$ & $2.4$ & $2.66$ & $3$ & $4$\\
\hline $p$ & $2.68$ & $2.74$ & $2.93$ & $3.83$ & $7.02$ & $9.86$  &
$12.05$ & $12.58$ & $12.53$ & $12.55$\\
\hline $p$ & $4.39$ & $4.42$ & $4.52$ & $5.07$ & $7.88$ & $10.89$  &
$13.30$ & $13.83$ & $13.67$ & $13.30$\\
\hline
\end{tabular}

\medskip

\begin{tabular}[c]
{|l|l|l|l|l|l|l|l|l|l|l|}
\hline $\beta$ & $5$ & $6$ & $7$ & $8$ & $9$ & $10$ & $15$ & $20$ & $25$ \\
\hline $p$ & $14.29$ & $16.79$ & $19.47$ & $22.19$ & $24.91$ &
$27.64$ & $41.31$ & $55.01$ & $68.73$\\
\hline $p$ & $14.86$ & $17.26$ & $19.88$ & $22.55$ & $25.23$ &
$27.93$ & $41.51$ & $55.16$ & $68.85$\\
\hline
\end{tabular}

\bigskip

\caption{Zeros of the function in the inequality \eqref{E4:Lgreen}
Theorem \ref{T:Lushnikov} .}
 \label{T4:L-phase}
\end{table}

\item
(Theorem \ref{T:Lushnikov-adapted}) the condition \eqref{E:LAreal}
is
\begin{equation}
 \label{E4:LAorange}
p >
\frac{2^{3/4}C^{7/2}v^{3/4}a(\beta,\gamma)^{1/2}}{(\pi^{3/2}m(\beta)^{1/2}
+ 2^{3/2} C^7 v(\beta)^{3/2} b(\beta))^{1/2}},
\end{equation}
and the condition \eqref{E:LA+simple} is
\begin{equation}
 \label{E4:LA+phase}
\begin{aligned}
\left(\frac{\pi^{3/2}m(\beta)^{3/2}}{2^{1/2} C^7 v(\beta)^{1/2}
}-\frac{v(\beta)b(\beta)}{m(\beta)^2} \right)p^2 +
\frac{v(\beta)a(\beta, 0)}{m(\beta)^2} \qquad & \\
- \frac{27 \pi^3 \,p^{4}}{8
\,C^{14}} \, m(\beta) \left(a(\beta, \gamma) - p^2 b(\beta)\right) 
& >0,
\end{aligned}
\end{equation}
the positive real zeros of which are listed in Table
\ref{T4:LAphase}.
\begin{table}[h]
\begin{center}
\begin{tabular}
[c]{|l|l|l|l|l|l|l|l|l|l|l|}
\hline $\beta$ & $0$ & $0.25$ & $0.5$ & $1$ & $1.6$ & $2$  & $2.4$ & $2.66$ & $3$  & $4$\\
\hline $p_b$ & $2.81$ & $2.87$ & $3.06$ & $3.90$ & $6.71$ & $9.10$ &
$10.45$ & $10.57$ & $10.35$ & $10.56$\\
\hline $p_t$ & $4.46$ & $4.49$ & $4.60$ & $5.14$ & $7.91$ & $10.86$
&
$12.05$ & $13.37$ & $12.01$ & $12.52$\\
\hline
\end{tabular}

\medskip

\begin{tabular}
[c]{|l|l|l|l|l|l|l|l|l|l|l|}
\hline $\beta$ & $5$   & $6$  & $7$    & $8$   & $9$   & $10$ & $15$ & $20$ & $25$\\
\hline $p_b$ & $12.22$ & $14.41$ & $16.74$ & $19.11$ & $21.49$ &
$23.88$ & $35.92$ & $48.03$ & $60.15$\\
\hline $p_t$ & $13.96$ & $16.11$ & $18.44$ & $20.80$ & $23.18$ &
$25.58$ & $37.61$ & $49.72$ & $61.85$\\
\hline
\end{tabular}
\end{center}

\bigskip

\caption{The positive real zeros of the polynomial in
\eqref{E4:LA+phase}, by Theorem \ref{T:Lushnikov-adapted}.}
 \label{T4:LAphase}
\end{table}

\item
Numerical simulations: the results for the oscillatory Gaussian
initial data with the phase $\gamma = \pm \frac12$ are in Table
\ref{T4:num-phase}. For $p \geq p_b$ the blow up was observed, for
$p \leq p_s$ the solution dispersed over time.
\end{itemize}
\begin{table}[h]
\begin{align*}
&\begin{tabular} [c]{|l||l|l|l|l|l|l|l|l|}
\hline $\beta$ & $0$ & $0.25$ & $0.5$ & $1.0$ & $2.0$ & $3.0$ & $4.0$ & $5.0$ \\
\hline
\hline $p_s^+$ & $3.56$ & $3.60$ & $3.70$ & $4.07$ & $4.94$  & $4.75$ & $3.70$ & $3.19$ \\
\hline $p_b^+$ & $3.57$ & $3.61$ & $3.71$ & $4.08$ & $4.941$ & $4.76$ & $3.71$ & $3.20$ \\
\hline $p_s^-$ & $2.42$ & $2.47$ & $2.63$ & $3.20$ & $4.263$ & $3.64$ & $2.98$ & $2.70$ \\
\hline $p_b^-$ & $2.43$ & $2.48$ & $2.64$ & $3.21$ & $4.26$  & $3.65$ & $2.99$ & $2.71$ \\
\hline
\end{tabular}\\
&\begin{tabular} [c]{|l||l|l|l|l|l|l|l|l|}\hline $\beta$ & $6.0$ &
$7.0$ &
$8.0$ & $9.0$ & $10.0$& $15.0$ & $20.0$ & $25.0$\\
\hline
\hline $p_s^+$ & $2.92$ & $2.76$ & $2.66$ & $2.59$ & $2.55$ & $2.48$ & $2.53$ & $2.56$ \\
\hline $p_b^+$ & $2.93$ & $2.77$ & $2.67$ & $2.60$ & $2.56$ & $2.49$ & $2.54$ & $2.57$ \\
\hline $p_s^-$ & $2.56$ & $2.47$ & $2.42$ & $2.38$ & $2.377$ & $2.37$ & $2.44$ & $2.48$ \\
\hline $p_b^-$ & $2.57$ & $2.48$ & $2.43$ & $2.39$ & $2.378$ & $2.38$ & $2.45$ & $2.49$ \\
\hline
\end{tabular}
\end{align*}
\caption{Threshold for blow up and scattering in numerical
simulations for the oscillatory Gaussian with the phase $\gamma =
\pm\frac12$.}
  \label{T4:num-phase}
\end{table}

To compare all the above conditions \eqref{E4:posEphase} -
\eqref{E4:LA+phase} with the numerical data, we graph the dependence
of $p$ on $\alpha$ in Figures \ref{F:oscgaussn} and
\ref{F:oscgaussp}. For clarity of presentation we plot $\ds
\frac{p}{\sqrt{1+\beta^2}}$ on the vertical axis.

\begin{figure}[p]
\includegraphics[scale=.76]{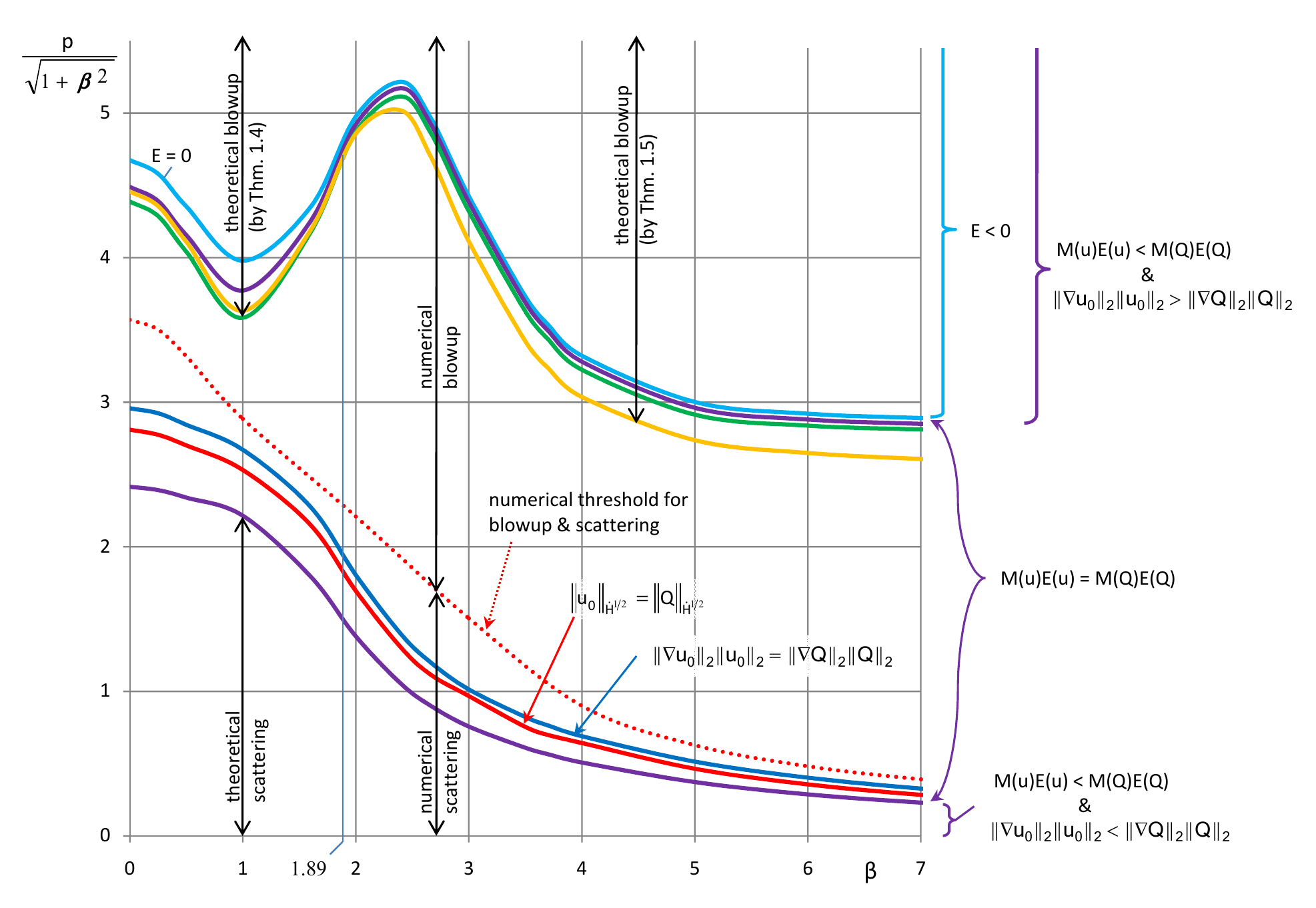}  %
\caption{Global behavior of the solutions to \eqref{E:NLSa} with the
oscillatory Gaussian initial data with positive quadratic phase
$\gamma = +\frac12$, see \eqref{E:osc-gaussian}. The vertical axis
is rescaled $p/\sqrt{1+\beta^2}$ to compare with the real case. The
blow up threshold curve denoted ``by Thm. \ref{T:Lushnikov}" is
given by \eqref{E4:Lgreen}, see values $p_t$ in Table
\ref{T4:L-phase}; the blow up threshold curve denoted ``by Thm.
\ref{T:Lushnikov-adapted}" is given by \eqref{E4:LA+phase}, see
values $p_t$ in Table \ref{T4:LAphase}. Observe that these curves
intersect at $\beta \approx 1.89$. Therefore, for small oscillations
($\beta < 1.89$) Theorem \ref{T:Lushnikov} provides the best range
for blow up, correspondingly, for large oscillations ($\beta >
1.89$) Theorem \ref{T:Lushnikov-adapted} gives a better range. The
curve ``theoretical scattering" is given by Thm. \ref{T:DHR}, see
\eqref{E4:MEphase} and values $p_1^\gamma$. The numerical threshold
(dotted curve) is given in Table \ref{T4:num-phase}, values $p_s^+$
and $p_b^+$. All $p$ values in this graph are normalized by
$\sqrt{1+\beta^2}$.}
 \label{F:oscgaussp}
\end{figure}

\begin{figure}[p]
\includegraphics[scale=.76]{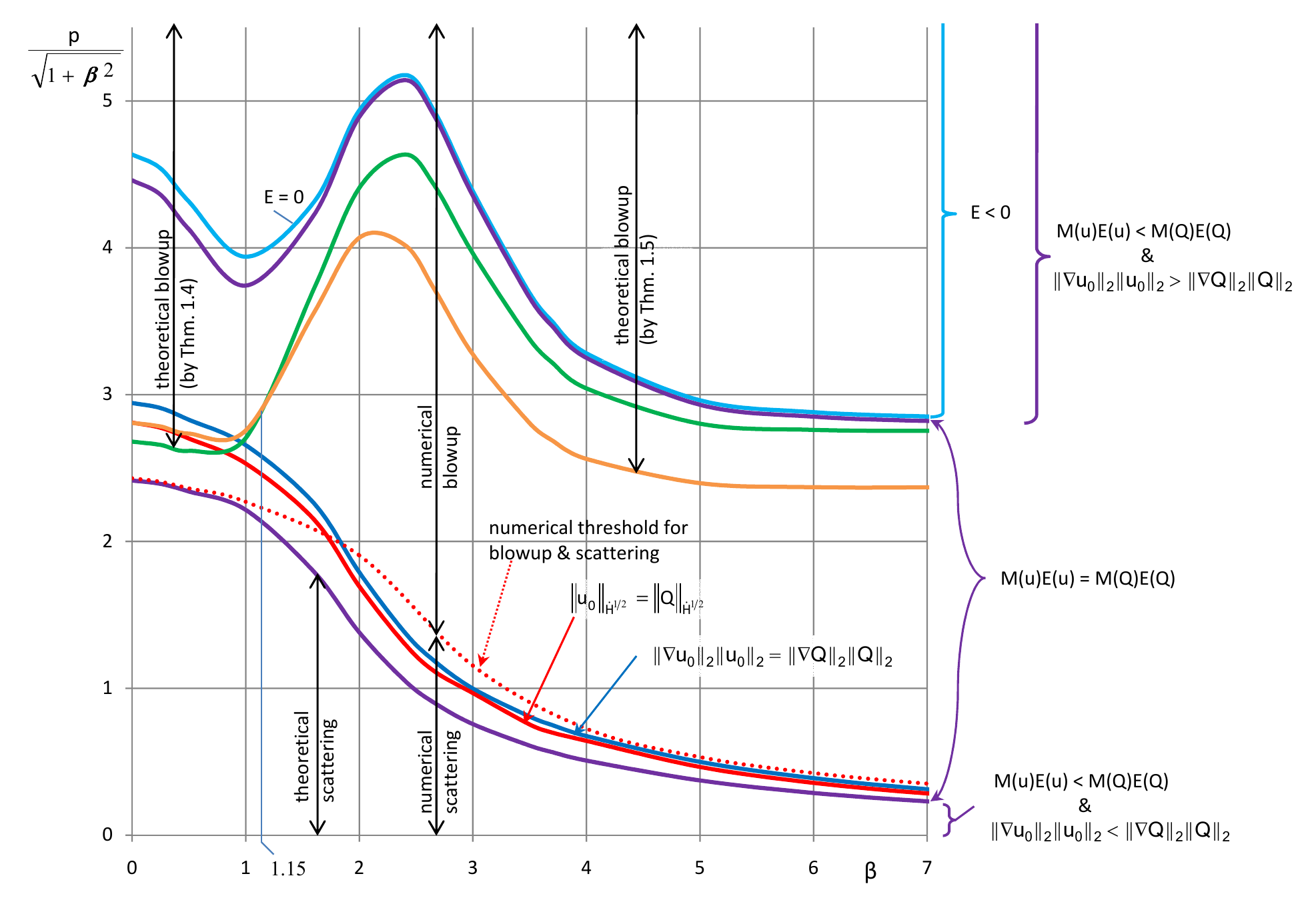}  %
\caption{Global behavior of the solutions to \eqref{E:NLSa} with the
oscillatory Gaussian initial data with negative quadratic phase
$\gamma = -\frac12$, see \eqref{E:osc-gaussian}. The vertical axis
is rescaled $p/\sqrt{1+\beta^2}$ to compare with the real case. The
blow up threshold curve denoted ``by Thm. \ref{T:Lushnikov}" is
given by the complement of \eqref{E4:Lgreen}, see values $p_b$ in
Table \ref{T4:L-phase}; the blow up threshold curve denoted ``by
Thm. \ref{T:Lushnikov-adapted}" is given by the complement of
\eqref{E4:LA+phase}, see values $p_b$ in Table \ref{T4:LAphase}.
Observe that these curves intersect at $\beta \approx 1.15$.
Therefore, for small oscillations ($\beta < 1.15$) Theorem
\ref{T:Lushnikov} provides the best range for blow up,
correspondingly, for large oscillations ($\beta > 1.15$) Theorem
\ref{T:Lushnikov-adapted} gives a better range. The curve
``theoretical scattering" is the same as in Figure \ref{F:oscgaussp}
and is given by Thm. \ref{T:DHR}, see \eqref{E4:MEphase} and values
$p_1^\gamma$. The numerical threshold (dotted curve) is given in
Table \ref{T4:num-phase}, values $p_s^-$ and $p_b^-$. All $p$ values
in this graph are normalized by $\sqrt{1+\beta^2}$. Note that for
small oscillations the numerical threshold dotted curve coincides
with the ``theoretical scattering" curve. }
 \label{F:oscgaussn}
\end{figure}

\subsection{Conclusions}
The above computations show
\begin{enumerate}
\item
Consistency with Conjecture 3: if $u_0$ is \emph{real}, then
$\|u_0\|_{\dot H^{1/2}}< \|Q\|_{\dot H^{1/2}}$ implies $u(t)$
scatters.

\item
Consistency with Conjecture 4: observe that the real oscillatory
Gaussian initial data $u_0$ is a radial profile that is not
\emph{monotonically decreasing}, thus, the condition $\|u_0\|_{\dot
H^{1/2}}> \|Q\|_{\dot H^{1/2}}$ does not necessarily imply that
$u(t)$ blows-up in finite time.

\item
The condition ``$\|u_0\|_{L^2}\|\nabla u_0\|_{L^2}<
\|Q\|_{L^2}\|\nabla Q\|_{L^2}$ {\it implies scattering}'' is not
valid (unless $M[u]E[u]<M[Q]E[Q]$ as in Theorem \ref{T:DHR}) even
for the real oscillatory Gaussian initial data.

\item
In all three cases (real data, data with positive phase and with
negative phase) Theorems \ref{T:Lushnikov} and
\ref{T:Lushnikov-adapted} provide new range on blow up than was
previously known from our Theorem \ref{T:DHR}. For small
oscillations (small $\beta$) Theorem \ref{T:Lushnikov} provides the
best range for blow up and for large oscillations Theorem
\ref{T:Lushnikov-adapted} provides a better result.
\end{enumerate}

\end{document}